\documentclass[a4paper,oneside,reqno,notitlepage,11pt]{amsart}
\usepackage{amsfonts}
\usepackage{amscd,amsmath}
\usepackage{euscript}
\usepackage[arrow,matrix,curve]{xy}\SilentMatrices
\def\xyma{\xymatrix@M.7em}
\usepackage{float}
\usepackage{amsmath}
\usepackage{amssymb,latexsym}
\usepackage{rotating}
\usepackage{verbatim}
\def\lotimes{\buildrel{L}\over\otimes}

\numberwithin{equation}{section}

\newtheorem{theorem}{Theorem}[section]
\newtheorem{lemma}[theorem]{Lemma}
\newtheorem{proposition}[theorem]{Proposition}
\newtheorem{corollary}[theorem]{Corollary}

\newtheorem{conjecture}[theorem]{Conjecture}

\theoremstyle{definition}
\newtheorem{definition}[theorem]{Definition}

\newtheorem{remark}[theorem]{Remark}
\newtheorem{example}[theorem]{Example}

\def\ga{\gamma}

\def\Ga{\Gamma}

\def\Lam{\Lambda}

\def\Z{{\mathbb{Z}}}

\def\N{\mathbb{N}}

\def\fp{\mathbb{F}_p}
\def\la{\longrightarrow}
\def\ot{\otimes}

\def\lot{\stackrel{L}{\ot}}
\def\bee{\begin{equation}}
\def\ee{\end{equation}}

\def\tor{\mathrm{Tor}}

\def\hom{\mathrm{Hom}}

\def\f2{\mathbb{F}_2}
\def\La{\Lambda}
\def\fp{\mathbb{F}_p}
\def\k{\mathbb{k}}

\newcommand{\Ext}{{\mathrm{Ext}}}
\renewcommand{\hom}{{\mathrm{Hom}}}
\newcommand{\Id}{{\mathrm{Id}}}
\renewcommand{\k}{\Bbbk}
\newcommand{\Fp}{{\mathbb{F}_p}}
\newcommand{\Fdeux}{{\mathbb{F}_2}}
\newcommand{\gr}{\mathrm{gr}\,}
\newcommand{\LL}{\mathcal{L}}
\newcommand{\NN}{\mathcal{N}}

\newcommand{\PP}{\mathcal{P}}

\newcommand{\I}{\mathcal{I}}
\newcommand{\FF}{\mathcal{F}}
\newcommand{\p}{{(p)}}

\begin{document}
\begin{flushright}
\end{flushright}

\title{Derived functors of the  divided power functors}
\author{Lawrence Breen, Roman Mikhailov and Antoine Touz\'e}

\begin{abstract}
We  study,  with special  emphasis on the $d=4$ case, the derived functors of the components $\Gamma^d_R(M)$ of the  divided power algebra $\Gamma_R(M)$ where  $M$ is a module over the ring   $R= \Z$.  While our results have applications both  to  representation theory and to algebraic topology, we illustrate them here by providing  a new functorial   description of certain  integral homology groups of the  Eilenberg-Mac Lane spaces $K(A,n)$, for $A$ a free abelian group. In particular, we give a complete functorial description of the groups $H_\ast(K(A,3);\Z)$ for such $A$. 
\end{abstract}

\maketitle

\tableofcontents

\section{Introduction}

\vspace{1cm}

In this article, we are interested in functors from the category of free $R$-modules to that of  $R$-modules (for $R$  a commutative ring) and particularly in  the functor  from such a module  $M$  to the module of  $d$-fold invariant tensors $ (M^{\ot d})^{\Sigma_d}$.  This module of invariant tensors is canonically isomorphic to the degree $d$ components $\Gamma^d_R(M)$  of the divided power algebra $\Gamma_R(M)$. Though these  are   non-additive functors of  $M$ whenever $d>1$, they may  be  derived by the Dold-Puppe theory \cite{d-p}. This associates to $M$ a family of $i$th-derived functors $L_i\Gamma^d(M,n)$ which depends on an additional  positive integer $n$. In a wider perspective,   $L_i\Gamma^d_R(M,n)$ is the $i$th homotopy group of Quillen's left-derived object  $L\Gamma^d_R(M[n])$, where   $M[n]$ is the module  $M$, viewed as a chain complex concentrated in degree $n$ \cite{quillen:homot}.
When $R$ is a field, the functors  $L_i\Gamma^d_R(M,n)$ are known. Our aim is to compute  certain of its values  when $R = \Z$. Such computations are motivated by applications in both  algebraic topology and in representation theory.

\bigskip

In algebraic topology, the  functors  $L_i\Gamma^d_\Z(A,n)$ where $A$ is an arbitrary abelian group are  closely related to the integral homology of the Eilenberg-Mac Lane spaces  $K(A,n+2)$, i.e those spaces whose only non-trivial homotopy group is $A$ in degree $n+2$. Indeed, even though a truly functorial description of  the  homology groups $H_\ast(K(A,n+2);\Z)$ is complicated, they are endowed with  a natural filtration  inducing  functorial isomorphism
\bee
\bigoplus_{d\ge 0}  L_{*-2d}\Gamma^d_\Z(A,n) \simeq \gr (H_*(K(A,n+2);\Z))\label{h-ast}
\ee
on the associated graded components. This filtration splits functorially when we restrict ourselves to the subcategory of free abelian groups, so that there  are then  functorial isomorphisms
\bee
\label{h-ast2} \bigoplus_{d\ge 0}  L_{*-2d}\Gamma^d_\Z(A,n)  \simeq \,H_\ast(K(A,n+2);\Z) \ee
for all $n \geq 0$.  When $A$ non free, such an isomorphism still exists, but is no longer functorial in $A$ (we refer to our appendix \ref{han} for a further discussion of this issue). The abelian groups $L_i\Gamma^d_\Z(A,n)$ are thus quite fundamental for topology, and the description of their dependence on the group $A$ carries much useful information which remains hidden when such a (finitely generated) group is decomposed into a direct sum of cyclic groups.

\bigskip

In representation theory, the representations of integral Schur algebras can be described in terms of the strict polynomial functors of Friedlander and Suslin. Such strict polynomial functors can be thought of as functors from finitely generated free abelian groups to abelian groups, equipped with an additional strict polynomial structure. The derived category  of weight $d$  homogeneous strict polynomial functors   $\mathbf{D}(\PP_{d,\Z})$ is equipped with a Ringel duality operator $\Theta$, which is a self-equivalence of $\mathbf{D}(\PP_{d,\Z})$.
For $A$ free finitely generated there is  an isomorphism
$$ H^\ast(\Theta^n\Gamma^d_\Z(A))\simeq  L_{nd - \ast}\Gamma^d_\Z(A,n) $$
Such an isomorphism holds for all strict polynomial functors, but the case of divided powers is fundamental since these functors constitute a family of projective generators of the category. In this context, the description of the $L_*\Gamma^d_\Z(A,n)$ as functors of $A$ is essential, since the functoriality is necessary in order to determine the action of the Schur algebra, hence to obtain the expressions $ H_*(\Theta^n\Gamma^d_\Z(A))$ as representations, and not simply as abelian groups. Actually we need more than the functoriality in order to understand the action of the Schur algebra, we really need to describe the strict polynomial structure of these functors.

\bigskip

Let us briefly review  what was known previously on the derived functors $L_*\Gamma^d_\Z(A,n)$. The integral homology of Eilenberg-Mac Lane spaces was computed by H. Cartan: a non-functorial description of these homology groups is given in \cite{cartan} Exp. 11, Thm 1, and a functorial one in \cite[Exp. 11, Thm 5]{cartan}. Work in this direction was pursued by  two students of Mac Lane, Hamsher \cite{hamsher} and Decker \cite{decker}, who studied the cases $n=1$ and $n \geq  1$ respectively. By the isomorphism \eqref{h-ast}, it is possible to retrieve from these results a description of the derived functors $L_*\Gamma^d_\Z(A,n)$. The answers however are very complicated. For example, the homology groups $H_i(K(A,n);\Z)$ for a finitely generated abelian group $A$ are described functorially by an infinite list of generators and relations, even though they are of finite type. Some additional progress in computing the functors
 $L_\ast\Gamma^d_\Z(A,n)$
 was made by a direct study of its properties  by Bousfield in \cite{bousfield-hom}, \cite{bousfield-pens},  and a complete description of these functors for any  $A$  was obtained for $d=2$ by Baues and Pirashvili  \cite{baupira} and in \cite{BM} for $d=3$. On the other hand,   no description of  the strict polynomial structure of the functors $L_\ast\Gamma^d_\Z(A,n)$ has been given so far.

\bigskip

We now list the results which we obtain in this paper. We give full description (as strict polynomial functors) :
\begin{enumerate}
\item[(1)] of the $L_*\Gamma^d_\Z(A,1)$ for all $d$ and  $A$ free.
\end{enumerate}
By \eqref{h-ast2}, this determines  a new  and  functorial description of  the groups $H_\ast(K(A,3);\Z)$ for $A$ a free abelian group. 
\begin{enumerate}
\item[(2)] of the $L_*\Gamma^d_\Z(A,n)$ for $d \leq 4$ and $A$ free.
\end{enumerate}
Our treatment is elementary and self-contained,  and does not use computations from the literature. The strict polynomial structures come in  as a help in the computations. Our computations have  a number of  interesting byproducts, in particular:
\begin{enumerate}
\item[(3)] short proofs of some computations  first obtained in \cite{antoine} including that of $L_*\Gamma^d_\k(V,n)$ where $\k$ is a  field of characteristic $2$,
\item[(4)] a new family of exact complexes `of Koszul type' involving divided powers over a field of characteristic 2.
\end{enumerate}
The kernels of these new complexes yield new families of functors, related to the $2$-primary component of $L_*\Gamma^d_\Z(A,1)$. We  describe these functors in a variety of ways.  The simplest is the  2-primary component of $L_3\Gamma^4(A,1)$, which we denote by $\Phi^4(A)$, and which can be described as the cokernel of the natural transformation $\Lambda^4_{\f2}(A/2) \to \Gamma^4_{\f2}(A/2)$  determined by the algebra structure of $\Gamma^\ast_{\f2}(A/2)$ (such a map only exists in characteristic 2).

\begin{enumerate} \item[(5)] We give a  description of the groups $H_{n+i}(K(A,n);\Z)$ for $A$ free  in the range $0 \leq i \leq 10$ which is both more natural and more precise than  Cartan's in \cite{cartan}.
\end{enumerate}

Finally in section \ref{conj}, we return to the derived functors $L_*\Gamma^d_\Z(A,n)$ for all $n$, $d$ and all abelian groups $A$. We present a conjectural description of these functors, up to a filtration, in terms of simpler derived functors of divided and exterior powers. We then make use here of our results and  of  those from \cite{BM} in order to provide some evidence in support of this conjecture.

\bigskip

We will now discuss the content of this article in more detail. In sections \ref{sec-classical} and \ref{sec-der} we collect the main properties of divided power algebras, strict polynomial functors and derived functors which  will be  required.
In section  \ref{qt-fil}, we introduce the useful notion of a quasi-trivial filtration which will allow us to prove that certain spectral sequences degenerate. For $V$ a vector space over a characteristic $p$ field $\k$, the divided power algebra  $\Gamma^\ast_\k(V)$ possesses such a filtration, which we call the principal filtration. We use this filtration to compute in 
section \ref{descr-gvn} those values of the derived functors $L_i\Gamma^d_\k(V,n)$ which will be needed in the sequel. 

\bigskip

In sections \ref{der1-gamma-z}  and  \ref{sec-der1-gamma-z-bis} we compute the strict polynomial functors $L_i\Gamma^d_\Z(A,1)$ for any free abelian group $A$ with the help of the Bockstein maps. These  may be viewed as a family of  differentials  on the mod $p$ graded groups $L_\ast\Gamma^d_\Fp(A/p,1)$ and  can  be explicitly described. While in odd characteristic these  mod $p$ derived functors  together the Bockstein differentials simply provide us with a standard dual Koszul complex, the result is slightly different in characteristic 2. In that case the  Bockstein differentials determine on  $L_\ast\Gamma^d_{\mathbb{F}_2}(A/2,1)$ a sort of dual Koszul complex structure, in which the  differentials are  twisted by a Frobenius map. We call this characteristic 2 complex the skew Koszul complex. In both odd and even characteristic  the  integral homology  groups $L_\ast\Gamma^d_\Z(A/p,1)$ are  simply  the groups of cycles in these
 Koszul and twisted Koszul complexes, and we are able to  analyze these more precisely in a number of situations.

\bigskip

In section \ref{sec-max} to \ref{der-gamma4-sec}, we present  another method for computing  the derived functors of
$\Gamma^d_\Z(A)$. Here we use the adic filtration associated to the
augmentation ideal of the algebra $\Gamma^\ast_\Z(A)$. We call this filtration the maximal filtration on $\Gamma_\Z(A)$. It was the
filtration chosen in \cite{BM} in order to study the derived functors
of $\Gamma^3_\Z(A)$. With the help of the associated spectral sequence, we  determine the values of
$L_\ast\Gamma^4_\Z(A,n)$.
The easiest cases $n= 1,2$ are studied in section \ref{sec-derG12} where we pay particular attention  to  the first new functor occurring
among the derived functors of $\Gamma^4_\Z(A)$. This is the functor
$\Phi^4(A)$ mentioned above, the 2-primary component of $ L_{3}\Gamma^4_\Z(A,1)$, which we describe in three distinct ways.
As  $n$ increases  the situation  becomes  more involved,  and a detailed study of the differentials in the maximal filtration spectral sequence for $\Gamma^4_\Z(A,n)$ for a general $n$  is carried out in section \ref{der-gamma4-sec}. The computation of certain differentials is delicate, and involves functorial considerations and careful dimension counts.
This yields  an  inductive formula for $L_\ast\Gamma^4_\Z(A,n)$ (theorem \ref{deriveddescr}), from which the complete description
  of the  derived functors of $\Gamma^4_\Z(A,n)$ for all $n$ follows directly  (theorem \ref{thm-G4Z}).

\bigskip

In section \ref{conj}, we return to the derived functors $L_*\Gamma^d_\Z(A,n)$ for all $n$, $d$ and all abelian groups $A$.
We propose a conjectural  description, up to a filtration, of the functors $L_r\Gamma^d_\Z(A,n)$. To state this requires that we first discuss the stable homology groups  $H_r(K(A,n);\Z)$ , i.e. those for which $n \leq r < 2n$. These additive groups are discussed in a number of sources, and we first review  here for the reader's convenience  certain aspects of the  admissible words formalism of Cartan.    We then reformulate  in theorem \ref{thm-new-formula}  his admissible words   in terms of a  simpler labelling which only involves those words which do not involve the transpotence operation. We  refer    in proposition \ref{prop-new}  to Cartan's computation of these stable values, and this is the only place in the text  where we make use, for simplicity,  of his results. We then state our conjecture \ref{conject}, and we verify that it is compatible with theorems \ref{thm-calcul-LG-un-Z} and \ref{thm-G4Z} as well as with the results in \cite{BM}.

\bigskip

We view the appendices as an integral part of our text. 
In  appendix \ref{app-comput}, we review some classical methods of
computations of $\rm Hom$'s and $ {\rm Ext}^1$'s in functor
categories which are used many times in our arguments. Appendix \ref{han} begins with a 
discussion of  the relation between the derived functors of the functor 
$\Gamma^d_\Z(A)$ and the integral homology of $K(A,n)$. We then
display a complete  table of the  functorial  values of the groups
$H_{n+i}(K(A,n);\Z)$  for $A$ free in the range $0 \leq i \leq
10$. While the constraints in choosing this range  of values of $i$ were
to some extent typographical, the fact that we  only know the complete set of 
derived functors of $\Gamma^d_\Z(A)$ for $d \leq 4$ would have precluded
the display of a complete table for much higher values of $i$.  This
table already features most of the unexpected phenomena which we
encounter in our computations, as will be  seen from the
discussion which precedes the table. A final appendix provides for
convenience a complete list   of  non-trival derived functors
$L_{*}\Gamma^4_\Z(A,n)$ for   $n \leq 4$.

\bigskip
The following diagram summarizes the relations between the various parts of our text:

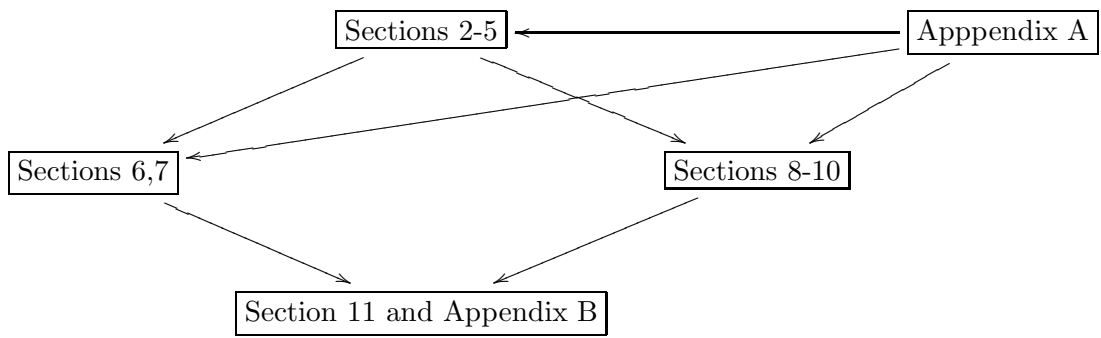
\[\xymatrix@R=30pt@C=15pt{
& \fbox{Sections 2-5}\ar[dl] \ar[dr]&& \ar[ll]\ar[llld]\ar[ld] \fbox{Apppendix A}\\
\fbox{Sections 6,7} \ar[dr]&& \fbox{Sections 8-10} \ar[dl]&\\
&\fbox{Section 11 and  Appendix  B} &&
}
\]

\thanks{The authors wish to thank A.K. Bousfield for correspondence concerning topics  related to the matter discussed here. }

 \bigskip

\thanks{The research of the second author is supported by the Chebyshev Laboratory  (Department of Mathematics and Mechanics, St.
Petersburg State University)  under RF Government grant 11.G34.31.0026 and by JSC "Gazprom Neft", as well as by the  RF Presidential  grant MD-381.2014.1.}

\bigskip

\noindent {\it Notation}. Throughout the text the notation $A/p$,  where $A$ is an  abelian group  and $p$ prime number,  stands for $A \ot_\Z \Z/p$. For any abelian group $A$, the notation $\Gamma (A)$ will  stand for $\Gamma_\Z(A)$, the divided power algebra associated to the $\Z$-module $A$ unless otherwise specified. On the other hand for any $\fp$-vector space $V$, the notation  $\Gamma (V)$ stands for the divided power algebra  $\Gamma_{\fp}(V)$ in the category of $\fp$-vector spaces.

\section{Classical functorial algebras}\label{sec-classical}

\subsection{The divided power algebra}   Let  $M$ be an $R$-module, where $R$ be a commutative ring with unit. The symmetric algebra $S_R(M)$ and
 the exterior algebra $\Lambda_R(M)$ are well-known. The
    divided power algebra $\Gamma_R(M)$ deserves equal attention. We recall here its basic properties, refering to \cite{roby} for the proofs. The algebra $\Gamma_R(M)$ is defined as the commutative $R$-algebra generated, for all $x \in M$ and all non-negative integers $i$, by elements $ \ga_i(x)$ which satisfy the following relations for all $x,y \in M$ and $\lambda \in R$:
\begin{align} & 1)\  \gamma_0(x) =
1 \hspace{1cm} x \neq 0\label{rel1} \\
\label{rel3a} & 2)\
\gamma_s(x)\gamma_t(x)=\binom{s+t}{s}\gamma_{s+t}(x) \\\label{rel4}
& 3)\ \gamma_n(x+y)=\sum_{s+t=n}\gamma_s(x)\gamma_t(y),\ n\geq 1\\
& 4)\
 \ \gamma_n(\lambda x)=\lambda^n\gamma_n(x), \ n\geq 1 \label{rel5}
\end{align}
Setting $s=t=1$ in \eqref{rel3a} one finds by induction on $n$ that $x^n = n!\ga_n(x)$, which justifies the name of divided power algebra for $\Gamma_R(M)$. 
The definition $\Gamma_R(M)$ is functorial with respect to $M$: an $R$-linear map $f:M\to N$ induces a morphism of $R$-algebras $\Gamma_R(M)\to \Gamma_R(N)$ which sends  $\gamma_i(x)$ to $\gamma_i(f(x))$.

The relations \eqref{rel1}-\eqref{rel5} are homogeneous, so that $\Gamma_R(M)$ can be given a graded algebra structure by setting  $\text{deg}\,\ga_i(x) = i$. We denote by $\Gamma^d_R(M)$ the homogeneous component of degree $d$, so that we have functorial $R$-linear isomorphisms $\Gamma^0_R(M)\simeq R$ and $\Gamma^1_R(M)\simeq M$ and a functorial decomposition
$$\Gamma_R(M)=\bigoplus_{d\ge 0}\Gamma^d_R(M)\;.$$

Consider the graded commutative algebra $TM=\bigoplus_{d\ge 0} V^{\otimes d}$, equipped with the shuffle product, and let us  denote by $TS(M)$ the graded subalgebra of invariant tensors $TS(M)=\bigoplus_{d\ge 0} (V^{\otimes d})^{\Sigma_d}$.
There is a well defined functorial homomorphism of graded algebras 
\begin{align}\Gamma_R(M)\to TS(M)\;,\label{eq-identification-dp}\end{align}
which sends $\gamma_n(x)$ to $x^{\otimes n}$. The morphism \eqref{eq-identification-dp} is an isomorphism if $M$ is  projective \cite[Prop. IV.5]{roby}. In particular, for a projective $R$-module $M$ we have functorial isomorphisms
 \[
\label{inv} \Gamma^d_R(M) \simeq (M^{\otimes \,
d})^{\Sigma_d}. \]
This description of $\Gamma^d(M)$ by invariant tensors is no longer valid for arbitrary modules $M$. For example, it follows by inspection from the relations \eqref{rel1}-\eqref{rel5} above  that  \bee
\label{gammators}\Gamma^2_\Z(\Z/2) \simeq \Z/4,\ee  so that this group cannot live in $\Z/2\, \ot \,\Z/2$. This example also shows that divided power algebras behave differently from symmetric and exterior algebras with respect to  torsion. Indeed, $S_\Z(M)$ and $S_{\fp}(M)$ coincide for any $\fp$-module $M$, as do $\Lam_\Z(M)$ and $\Lam_{\fp}(M)$, but the similar assertion  is not  true for   $\Gamma_\Z(M)$ and $\Gamma_{\fp}(M)$. In particular, for any $\fp$-module $M$,  we  will  always carefully distinguish between the functors $\Gamma_\Z(M)$ and $\Gamma_{\fp}(M)$.

\subsection{Exponential functors}
A functor $F(M)$ from  $R$-modules to $R$-algebras is said to be an
exponential functor if for all pair of $R$-modules  $M$ and $N$, the composite
\begin{align}F(M)\otimes F(N)\xrightarrow[]{F(i_1)\ot F(i_2)} F(M\oplus N)\otimes F(M\oplus N)\xrightarrow[]{\mathrm{mult}} F(M\oplus N)\;, \label{eq-exp-prop}\end{align}
induced by the canonical inclusions $i_1$,  $i_2$ of $M$ and $N$ into $M\oplus N$, is an isomorphism.
We refer to the map \eqref{eq-exp-prop} as the \emph{exponential isomorphism} for
$F(M)$.
For example, the  algebras $S_R(M)$, $\Lambda_R(M)$, and $\Gamma_R(M)$ \cite[Thm III.4]{roby} determine  exponential functors.

Exponential functors are  endowed
with a canonical bialgebra structure. Indeed, there is a coproduct induced by the diagonal
inclusion  $\Delta$ of $M$ into $M\oplus M$:
$$F(M)\xrightarrow[]{F(\Delta)} F(M\oplus M)\simeq F(M)\otimes F(M)\;. $$
It follows from \eqref{rel4}
that the coalgebra structure obtained in this way on $\Gamma_R(M)$ is the one  determined  by the comultiplication map $\ga_i(x)\mapsto \sum_{i= j+k}\ga_j(x) \ot \ga_k(x)$.

\subsection{Duality}
\label{duality}

Given a $R$-module $M$, we let $M^\vee:=\hom_R(M,R)$.
The dual $\Gamma_R (M)^\sharp$  of the divided power algebra is the $R$-module
$$\Gamma_R (M)^\sharp  :=\bigoplus_{d\ge 0} (\Gamma^{d}_R(M^\vee))^\vee\;.$$
The bialgebra structure on $\Gamma_R (M^\vee)$ defines a bialgebra structure on $\Gamma_R (M)^\sharp$. The dual of symmetric and exterior algebras are defined similarly.
An explicit computation shows that for all projective $R$-modules $M$, there are a natural isomorphisms of $R$-bialgebras
$$\Gamma_R(M)^\sharp \simeq S_R(M)\;,\qquad S_R(M)^\sharp\simeq \Gamma_R(M)\;,\qquad\Lambda_R(M)^\sharp \simeq
\Lambda_R(M)\;.$$

\subsection{Base change}
For any $R$-module $M$ and any commutative $R$-algebra $A$, there is a base change isomorphism of $A$-algebras \cite[Thm III.3]{roby}, which sends $\gamma_n(x)\otimes 1$ to $\gamma_n(x\otimes 1)$:
\begin{equation}
\label{basechange}
\Gamma_R(M) \ot_R A \simeq \Gamma_{A}(M\ot_R A).
\end{equation}
There are similar base change isomorphisms for symmetric and exterior algebras:
$$S_A(M) \ot_R A \simeq S_{A}(M\ot_R A)\;,\qquad  \Lambda_R(M) \ot_R A \simeq \Lambda_{A}(M\ot_R A)\;.$$

\section{Derived functors and strict polynomial functors}\label{sec-der}

\subsection{Derived functors}
Let $R$ be a commutative ring. The Dold-Kan correspondence states that the
normalized chain complex functor $\NN$ is an equivalence of categories preserving
homotopy equivalences, with inverse $K$ (also preserving homotopy equivalences):
$$\NN: \mathrm{simpl}(R\text{-Mod})\rightleftarrows \mathrm{Ch}_{\ge
0}(R\text{-Mod}): K\;. $$
If $M$ is a
$R$-module, and $F(M)$ denotes an endofunctor of the category of $R$-modules,
Dold and Puppe defined \cite{d-p} its derived functor $L_iF(M,n)$ by the formula:
$$L_iF(M,n) = \pi_iF K(P^M[n])\;,$$
where $P^M$ is a projective resolution of $M$, and $[n]$ is the degree $n$ shift
of complexes (i.e. $C[n]_i=C_{i+n}$). More generally, if $C$ is a complex of $R$-modules, we denote by
$LF(C)$ the simplicial $R$-module $F(K(P))$, where $P$ is a complex of projective $R$-modules quasi-isomorphic to $C$,
and by $L_iF(C)$ its homotopy groups.

\begin{remark}
The definition of the derived functors of $F$ only depends on the restriction of $F$ to the category of free $R$-modules. Furthermore, if $F$ commutes with directed colimits of free $R$-modules (as the divided power functors do), then $F$ is completely determined by its restriction to the category of free finitely generated $R$-modules. For example, for a free $R$-module $M$, we have an isomorphism $L_*\Gamma^d(M,n)=\lim_iL_*\Gamma^d(M_i,n)$, where the limit is taken over the directed system of free finitely generated submodules $M_i$ of $M$.
\end{remark}




\subsection{Strict polynomial functors}\label{subsec-str_der}
We are mainly interested in the divided powers functors $\Gamma^{d}_R(M)$ and functors related to these.
All these functors actually belong to a class of very rigid functors called
`strict polynomial functors', introduced by Bousfield \cite{bousfield-hom} in the context of derived functors  (Bousfield called them `homogeneous functors') and independently by Friedlander and Suslin \cite{FS} in the context of the cohomology of affine algebraic group schemes. We also recommend   \cite{krause} for a presentation of strict polynomial functors.
We will only recall here the basic facts required for our computations.

\subsubsection{The category of strict polynomial functors}
Strict polynomial functors can be thought of as functors from the category of finitely generated projective $R$-modules to the category of $R$-modules, equipped with an additional `scheme-theoretic' structure. Morphisms of strict polynomial functors are natural transformations which preserve this additional structure.
The category $\PP_R$ of strict polynomial functors is abelian.
Let $\FF_R$ denote the category of functors from finitely generated projective $R$-modules to $R$-modules.
There is an exact forgetful functor:
$$\PP_R\to \FF_R \;.$$
The tensor product $F(M)\otimes_R G(M)$ of two strict polynomial functors has a canonical structure of a strict polynomial functor, as well as the composition $F(G(M))$ provided that $G(M)$ has values in finitely generated projective $R$-modules.

The additional scheme-theoretic structure of a strict polynomial functor $F(M)$ determines the weight of $F(M)$ (this weight is called `degree' in \cite{FS}. In the present article, we prefer to  reserve the term  degree for the  homological degrees). A strict polynomial functor $F(M)$ of weight $d$ is \emph{homogeneous of weight $d$} if for all $G(M)$ of weight less than $d$, $\hom_{\PP_R}(G(M),F(M))$ is zero. The full subcategory of $\PP_R$ whose objects are homogeneous strict polynomial functors of weight $d$ is abelian. It is usually denoted by $\PP_{d,R}$, and there is an isomorphism of abelian categories:
$$\PP_{R}\simeq \bigoplus_{d\ge 0}\PP_{d,R}\;.$$
In practice, the weights can be determined by the following rules. The functors $S^d_R(M)$, $\Lambda^d_R(M)$, $\Gamma^d_R(M)$ and $M^{\otimes d}$ are homogeneous of weight $d$. If $F(M)$ and $G(M)$ are homogeneous of degree $d$, resp. $e$, then $F(M)\otimes G(M)$ is homogeneous of degree $d+e$ and $F(G(M))$ is homogeneous of degree $de$.

A useful structure on the category of strict polynomial functors is the (weight preserving) duality functor
\bee
\label{duality1}
^\sharp: \PP_{R}^{\mathrm{op}}\to \PP_{R}\;,\ee
which sends a strict polynomial functor $F(M)$ to the functor $F(M)^\sharp:=F(M^\vee)^\vee$, where `$^\vee$' denotes $R$-linear duality, i.e. $M^\vee=\hom_R(M,R)$. For example we have $S_R^d(M)^\sharp=\Gamma^d_R(M)$ and $\Lambda^d_R(M)^\sharp=\Lambda^d_R(M)$. When $R$ is a field, the duality functor is exact.

The advantage of working with strict polynomial functors rather than with ordinary functors from finitely generated projective $R$-modules to $R$-modules is twofold.
\begin{enumerate}
\item The category of strict polynomial functors is graded by the weight. In all computations involving strict polynomial functors this information is automatically and transparently carried by the `scheme theoretic' structure of the functors, see remark \ref{rk-poids}. The price to pay is that one has be a take care  that a given ordinary functor may sometimes be given several non-isomorphic scheme-theoretic structures, as the example of the Frobenius twists in section \ref{subsubsec-Frob} shows.
\item It is easier to compute extensions in $\PP_R$ than in $\FF_R$, and there are often less extensions possible in $\PP_R$. This fact will be of great help in solving extension problems coming from spectral sequences. We have gathered in appendix \ref{app-comput} some standard methods and results regarding the computation of extensions in these categories.  An illustration of the difference between extensions in $\PP_R$ and in  $\FF_R$ is provided by comparing the results from  lemma \ref{lm-calc2} with those of  lemma \ref{lm-calc3}.
\end{enumerate}

\subsubsection{Frobenius twists functors}\label{subsubsec-Frob}
We now take a field $\k$ of positive characteristic $p$ as our ground ring $R$, and we let $V$ be a generic finite dimensional $\k$-vector space.
We denote by $V^{(r)}$ the strict polynomial subfunctor of $S_\k^{p^r}(V)$ generated by
the $p^r$-th powers of elements of $V$. Equivalently, $V^{(r)}$ is the the kernel of the map $S_\k^{p^r}(V)\to \bigoplus_{k=1}^{p^r-1}S_\k^{k}(V)\otimes S_\k^{p^r-k}(V)$ induced by the comultiplication of the symmetric power bialgebra. The functor $V^{(r)}$ is called
the \emph{$r$-th
Frobenius twist functor}. It has an important role in the theory of representations of affine algebraic group schemes in positive characteristic \cite{FS}, and it will also appear in our computations. It enjoys the following basic properties.
\begin{enumerate}
\item The functor $V^{(r)}$ is a homogeneous strict polynomial functor of weight $p^r$.
\item The functor $V^{(r)}$ is additive.
\item The dimension of $V^{(r)}$ is equal to the dimension of $V$.
\item The  Frobenius twist functors can be composed together $(V^{(r)})^{(s)}=V^{(r+s)}$, and the functor $V^{(0)}$ is the identity
functor.
\item The  Frobenius twist functors are self dual: $(V^{(r)})^\sharp\simeq V^{(r)}$, so that $V^{(r)}$ is the cokernel of the map $ \bigoplus_{k=1}^{p^r-1}\Gamma_\k^{k}(V)\otimes \Gamma_\k^{p^r-k}(V)\to \Gamma_\k^{p^r}(V)$ induced by the multiplication of the algebra of divided powers.
\end{enumerate}

The inclusion $V^{(r)}\hookrightarrow S_\k^{p^r}(V)$ is called the Frobenius map and the dual
epimorphism $\Gamma^{p^r}_\k(V)\twoheadrightarrow V^{(r)}$ therefore deserves to be  called the
Verschiebung map following the usage in arithmetic. As explained in appendix \ref{app-comput}, these two morphisms provide bases  of the vector spaces  $\hom_{\PP_\k}(V^{(r)},S_\k^{p^r}(V))$ and  $\hom_{\PP_\k}(\Gamma_\k^{p^r}(V), V^{(r)})$ respectively.

Observe that if $\k=\mathbb{F}_q$ with $q$ dividing $p^r$, the strict polynomial functors $V^{(nr)}$, $n\ge 0$, are not isomorphic  to each other since they do not  have the same weight.  However,  if we forget their scheme-theoretic structure and view them as ordinary functors, that is as objects of $\FF_R$, they all become isomorphic to the identity functor.

\subsection{Derived functors and differential graded $\PP_R$-algebras} 
Derivation in the sense of Dold and Puppe yield functors for all nonnegative integers $i$ and $n$:
$$
\begin{array}{ccc}
\PP_{d,R} &\to &\PP_{d,R}\\
F(M) & \mapsto &  L_iF(M,n)
\end{array}
$$
We shall denote by $\NN F(M,n)$ the normalized chains of the simplicial object $F(K(M[n]))=F(K(R[n])\otimes M)$. Thus, $\NN F(M,n)$ is a complex of homogeneous strict polynomial functors of weight $d$.

\begin{remark}\label{rk-derfrobtw}
If $\k$ is a field of positive characteristic and $G(V)=F(V^{(r)})$, the additivity of the Frobenius twist yields an isomorphism
$\NN G(V,n)\simeq \NN F(V^{(r)},n)$. Thus precomposition by Frobenius twist is harmless in computations, i.e. there are isomorphisms:
\begin{align*}L_iG(V,n)\simeq L_iF(V^{(r)},n)\;.\end{align*}
\end{remark}

By composing the shuffle map in the Eilenberg-Zilber theorem and the product of the divided power algebra:
\begin{align*}\NN\Gamma^d_R(M,n)\otimes \NN\Gamma^e_R(M,n)\to \NN(\Gamma^d_R\otimes \Gamma^e_R)(M,n)\to \NN\Gamma^{d+e}_R(M,n)\;, \end{align*}
we obtain an algebra structure on the direct sum $\NN\Gamma_R(M,n)=\bigoplus_{d\ge 0}\NN\Gamma^d_R(M,n)$.
The following definition axiomatizes the properties of this direct sum, as well as similar objects which will appear in our computations.

\begin{definition}\label{def-PR-alg}
Let $M$ be a projective finitely generated module over a ring $R$. A
\emph{differential graded strict polynomial algebra} (dg-$\PP_R$-algebra, for short) is a
bigraded $R$-algebra
$A(M)=\bigoplus_{d\ge 0, i\ge 0} A^d_i(M)$, endowed with a differential
$\partial:A^d_i(M)\to A^d_{i-1}(M)$, and
satisfying the following
properties.
\begin{enumerate}
\item The $A^d_i(M)$ are homogeneous strict polynomial functors
of weight $d$ with respect to $M$. We call these the homogeneous components of degree
$i$ and weight
$d$ of $A(M)$.
\item The multiplication $A^d_i(M)\otimes A^e_j(M)\to
A^{d+e}_{i+j}(M)$ and the differential $\partial$ are morphisms
of strict polynomial functors.
\end{enumerate}
A morphism of differential graded $\PP_R$-algebras
$f:A(M)\to B(M)$ is a morphism of differential graded
$R$-algebras provided by a family of morphisms of strict polynomial
functors $f^d_i:A^d_i(M)\to B^d_i(M)$.
\end{definition}

Definition \ref{def-PR-alg} admits obvious variants for dg-$\PP_R$-coalgebras, dg-$\PP_R$-bialgebras, etc. whose formulation is left to the reader.
A dg-$\PP_R$-algebra with zero differential is simply called a graded
$\PP_R$-algebra, and a graded $\PP_R$-algebra concentrated in degree zero is simply called a $\PP_R$-algebra. Here are some basic examples of graded $\PP_R$-algebras.
\begin{itemize}
\item We denote by $\Gamma_R(M[i])$ the divided power algebra generated by a
finitely generated projective $R$-module $M$ placed in degree $i$. It is a
graded $\PP_R$-algebra whose homogeneous  component of  degree $di$ and  weight $d$   is  $\Gamma^d_R(M)$.
\item If $R=\k$ is a field of positive characteristic $p$, and $V$ is a finite dimensional $\k$-vector space, we similarly denote by
$\Gamma_\k(V^{(r)}[i])$ the divided power algebra generated by a copy of the
Frobenius twist functor $V^{(r)}$ placed in degree $i$. It is a graded
$\PP_R$-algebra with $\Gamma^d_\k(V^{(r)})$ as homogeneous part of degree $di$
and weight $dp^r$.
\item  We will use the similar notations $\Lambda_R(M[i])$,
$\Lambda_\k(V^{(r)}[i])$, $S_R(M[i])$ and $S_\k(V^{(r)}[i])$ for the corresponding
exterior and symmetric algebras.
\end{itemize}
Derivation of functors (at the level of chain complexes) can be restated as a functor:
\begin{align}\begin{array}{ccc}
\left\{\text{$\PP_R$-algebras}\right\} & \to &
\left\{\text{dg-$\PP_R$-algebras}\right\} \\
 A(M) & \mapsto & \NN A(M,n)
\end{array}\;.\label{deriv-dgp}
\end{align}
More generally, if $A(M)$ is a dg-$\PP_R$-algebra, we define $\NN A(M,n)$ by placing the degree $j$ object of the complex $\NN A^d_i(M)$ in weight $d$ and degree $i+j$, with product defined by the shuffle map and the product of $A(M)$, and with differential defined as the sum of the differential of $A(M)$ and the differential coming from the simplicial structure. We thus obtain a functor:
\begin{align}\begin{array}{ccc}
\left\{\text{dg-$\PP_R$-algebras}\right\} & \to &
\left\{\text{dg-$\PP_R$-algebras}\right\} \\
 A(M) & \mapsto & \NN A(M,n)
\end{array}\;.\label{deriv-dgp2}
\end{align}
\begin{remark}
When $A(M)$ is graded commutative and exponential, the dg-$\PP_R$-algebra $\NN A(M,n)$ coincides (up to homotopy equivalence) with the $n$-fold bar construction of $A(M)$ \cite[Chap X]{ML}. For arbitrary dg-$\PP_R$-algebras, these two constructions are different.
\end{remark}

The homology of a dg-$\PP_R$-algebra is a graded $\PP_R$-algebra. For example
we have the graded $\PP_R$-algebra
\begin{align}L_*\Gamma_R(M,n)=\bigoplus_{d\ge
0}L_*\Gamma^d_R(M,n)\;.\label{eq-gP-alg}\end{align}
\begin{remark}\label{rk-poids}
The left hand side of \eqref{eq-gP-alg} does not display the integer
$d$. However, if we know $L_*\Gamma_R(M,n)$ as a graded $\PP_R$-algebra, it is always
easy to retrieve $L_*\Gamma^d_R(M,n)$ as its homogeneous component of weight $d$. In
this paper, we will usually not explicitly write down the weight of functors,
because this information is already implicitly encoded in the strict polynomial
structure of the objects. For example, if $R$ is a field of characteristic $2$
and $M$ a finite dimensional $R$-vector space, then the  weight
$4$ component of $\Gamma_R(M[1])\otimes \Gamma_R(M^{(1)}[3])$ is equal to
$\Gamma_R^4(M)\oplus \Gamma^2_R(M)\otimes M^{(1)}\oplus \Gamma^2_R(M^{(1)})$, where these three  summands live in degrees $4$, $5$ and $6$ respectively.
\end{remark}

\subsection{Some basic facts regarding derived functors of the divided power algebra}

Before starting our computations, we recall in this section some basic facts on derived functors of the
symmetric algebras, the exterior algebras and the divided powers algebras, which
give a clearer picture of the situation. The following formula is due to Bousfield \cite{bousfield-hom} and Quillen \cite{quillen-rings}  (see also  \cite{illusie} vol. I, chapter I,\: \S  4.3.2).
\begin{proposition}[D\'ecalage]
Let $R$ be a commutative ring, and let $M$ be a finitely generated projective $R$-module.
There are isomorphisms of graded
$\PP_R$-algebras:
\begin{align}&\bigoplus_{i,d\ge 0}L_{i}\Gamma_R^d(M, n)\simeq \bigoplus_{i,d\ge 0}L_{i+d}\Lambda_R^d(M,n+1)
\simeq \bigoplus_{i,d\ge 0} L_{i+2d}S_R^d(M,n+2)\;.\label{dec} \end{align}
\end{proposition}

If $R$ is a $\mathbb{Q}$-algebra, there is an isomorphism of graded
$\PP_R$-algebras $S_R(M[i])\simeq \Gamma_R(M[i])$. Thus the d\'ecalage formula
implies the following statement.
\begin{corollary}
If $R$ is a $\mathbb{Q}$-algebra, and $M$ is a finitely generated
$R$-module, then the graded $\PP_R$-algebra $L_*\Gamma_R(M, n)$ is isomorphic to $\Gamma_R(M[n])$ if $n$ is even and to $\Lambda_R(M[n])$ if $n$ is odd.
\end{corollary}
There is thus no issue in  computing derived functors of divided power algebras
over $\mathbb{Q}$-algebras $R$. We now give some elementary properties of divided powers
when $R$ is noetherian.
\begin{proposition}\label{lm-free-tors}
Let $R$ be noetherian and let $M$ be a finitely generated projective
$R$-module. Then $L_j\Gamma_R^d(M, n)$ is
\begin{itemize}
\item zero if $j<n$ or $j>nd$,
\item a finitely generated $R$-module if $n\le j\le nd$. If $R=\Z$,
then for $n\le j<nd$, $L_j\Gamma_R^d(M, n)$ is a finite abelian group.
\end{itemize}
Finally, the graded $\PP_R$-subalgebra
$$\bigoplus_{d\ge 0} L_{nd}\Gamma_R^d(M, n) \subset \bigoplus_{d,j\ge 0}
L_{j}\Gamma_R^d(M, n)=L_*\Gamma_R(M, n)$$
is equal to
$\Lambda_R(M)$ if $n$ is odd and $2\ne 0$ in $R$, and  to $\Gamma_R(M)$ otherwise.
\end{proposition}

\section{A quasi-trivial filtration of the divided power algebra}
\label{qt-fil}
In this section, we fix a field $\k$ of positive characteristic $p>0$. All the
functors considered are strict polynomial functors defined over $\k$. We write $\Gamma^d$,
$\Lambda^d$ and $S^d$ for $\Gamma^d_\k$, $S^d_\k$ and $\Lambda^d_\k$.
A generic finite dimensional $\k$-vector space will be denoted by the
letter `$V$'. The goal of this section is to introduce some particularly nice filtrations of
dg-$\PP_\k$-algebras, which we call `quasi-trivial', and to exhibit a
quasi-trivial filtration of the divided power algebra.

\subsection{Quasi-trivial filtrations}
A nonnegative decreasing filtration on a dg-$\PP_\k$ algebra $A(V)$ is a family
 of subfunctors of $A(V)$
$$\textstyle\bigcap_{i\ge 0}F^iA(V)=0\subset\dots\subset F^{i+1}A(V)\subset F^i A(V)\subset\dots\subset F^0A(V)=A(V)\;,$$
satisfying the following conditions
\begin{enumerate}
\item $F^iA(V)=\bigoplus_{d\ge 0,j\ge 0} (F^iA(V))^d_j$  with
$(F^iA(V))^d_j\subset A_j^d(V)$.
\item The differential and the multiplication  in $A(V)$ restrict to morphisms
$d:F^iA(V)\to F^iA(V)$ and $F^iA(V)\otimes F^jA(V)\to F^{i+j}A(V)$.
\end{enumerate}
The associated graded object is then a dg-$\PP_\k$ algebra. The grading which we will consider is the one coming from $A(V)$, not the filtration grading.
We will simply write
$$\gr A(V)=\bigoplus_{i,j,d} \gr^i A^d_j(V)\;,\quad\text{ and } (\gr A)^d_j(V)=\bigoplus_{i} \gr^i A^d_j(V)\;.$$

\begin{definition}\label{def-quasi-trivial}
Let $A(V)$ be a dg-$\PP_\k$-algebra equipped with a nonnegative decreasing filtration
$(F^iA(V))_{i\ge 0}$. We will say that this filtration is quasi-trivial if
\begin{enumerate}
\item the algebra  $\gr A(V)$ is exponential and the components $\gr A_j^d(V)$  have
finite dimensional values  for all $j$ and $d$, and
\item there is a weight preserving isomorphism of differential graded
$\k$-algebras \linebreak $\phi:A(\k)\xrightarrow[]{\simeq}\gr A(\k)$.
\end{enumerate}
\end{definition}

The next lemma gives some straightforward consequences of definition 
\ref{def-quasi-trivial}. The property (c) says that the graded $\PP_\k$-algebras $A(V)$ and $\gr A(V)$ are `as
close as possible': the filtration modifies the functoriality but not the algebra structure of $A(V)$.
\begin{lemma}\label{lm-triv}
Let $A(V)$ be a dg-$\PP_\k$-algebra equipped  with a quasi-trivial filtration.
Then:
\begin{enumerate}
\item[(a)] The filtration of each summand $A^d_j(V)$ is bounded.
\item[(b)] The algebra $A(V)$ is exponential.
\item[(c)] The choice of a basis of $V$ determines a
\emph{non-functorial} weight-preserving  isomorphism of differential graded $\k$-algebras $A(V)\simeq \gr A(V)$.
\end{enumerate}
\end{lemma}
\begin{proof}
(a) Since each $A_j^d(V)$ is a strict polynomial functor with finite dimensional values, it is a finite functor, in particular all filtrations are bounded \cite[lemma 14.1]{antoine}.

(b) The map $\psi:A(V)\otimes A(W)\to A(V\oplus W)$ induced by the
multiplication preserves filtrations, and the associated graded map
$\gr\psi  :\gr A(V)\otimes \gr A(W)\xrightarrow[]{\simeq}\gr A(V\oplus W)$ is an
isomorphism by the first condition of definition \ref{def-quasi-trivial}.
If we restrict ourselves to homogeneous components of a given  weight $d$, the filtrations on
$A(V)\otimes A(W)$ and $A(V\oplus W)$ are finite, so that $\psi$ is an isomorphism.

(c) A basis of $V$ determines an isomorphism $V\simeq \k^s$. We obtain the
required isomorphism of differential graded algebras as the composite
$$A(\k^s)\simeq A(\k)^{\otimes s}\xrightarrow[\simeq]{\phi^{\otimes s}} (\gr
A(\k))^{\otimes s}\simeq \gr A(\k^s)\;, $$
where the first and third isomorphisms are induced by the multiplications
(and are isomorphisms since $A(V)$ and $\gr A(V)$ are exponential).
\end{proof}

A crucial property of quasi-trivial filtrations is that they commute
with derivation.

\begin{proposition}\label{prop-quasi-trivial-collapse}
Let $A(V)$ be a graded $\PP_\k$-algebra, endowed with a quasi-trivial
filtration. For all $n\ge 0$, there exists a filtration of the graded
$\PP_\k$-algebra $L_*A(V,n)$ and a
 functorial  isomorphism of graded $\PP_\k$-algebras:
$$ \gr (L_*A(V,n))\simeq L_*(\gr A)(V,n)\;.$$
\end{proposition}
\begin{proof}
The filtration of $A(V)$ induces a filtration of the dg-$\PP_\k$-algebra
$\NN A(V,n)$. The associated spectral sequence of graded $\PP_\k$-algebras has
the form:
\begin{align}E^1_{i,j}=H_{i+j}(\gr^{-i}\NN A(V,n))\Longrightarrow
H_{i+j}(\NN A(V,n))\;.\label{eq-ss-filtr}\end{align}
While this is  a second quadrant spectral sequence, there
is no problem with convergence since it splits as a direct sum of spectral
sequences of homogeneous strict polynomial functors of   given weight $d$, and
the lemma \ref{lm-triv}(a) ensures that each summand bounded converges \cite[Thm 5.5.1]{weibel}.
The first page of the spectral sequence may be rewritten as
$E^1_{i,j}=L_{i+j}(\gr^{-i}A)(V,n)$. To prove proposition
\ref{prop-quasi-trivial-collapse} it  therefore suffices to prove that the spectral
sequence \eqref{eq-ss-filtr} degenerates at  $E_1$. This will be the case if we are
able to prove that for all $i$ and $d$, the homogeneous components  of degree $i$ and weight
$d$ of the graded $\PP_\k$-algebras $L_*A(V,n)$ and $L_*(\gr A)(V,n)$ have the
same dimension (note that we already know that they both  are finite-dimensional by proposition  \ref{lm-free-tors}).

This equality of dimensions follows directly from the observation that $A(K(V,n))$ and $\gr A(K(V,n))$ 
coincide as semi-simplicial $\k$-vector spaces. 
Indeed the graded $\k$-algebras 
with weights $A(\k)$ and $\gr A(\k)$ are isomorphic, and we claim that for exponential graded $\PP_\k$-algebras
$E$, the graded $\k$-algebra with weights $E(\k)$ determines completely 
the semi-simplicial $\k$-vector space $E(K(V,n))$.

The latter claim follows from the explicit construction
of the Dold-Kan functor $K$
\cite[Sec. 8.4]{weibel}. The simplicial $\k$-vector
space $K(V[n])$ is degreewise finite-dimensional, say of some dimension $d_k$ in
degree $k$. If we choose a basis of $V$, each $K(V[n])_k$ has a canonical basis determined by the
basis of $V$, and if we use coordinates relative to these bases, then for each
$i$ the face operators $\partial_i:K(V[n])_k\to K(V[n])_{k-1}$ is given by a formula
$\partial_i(x_1,\dots,x_{d_k})=(y_1,\dots,y_{d_{k-1}})$, where $y_j=\sum_{i\in
I_j} x_i$ and $I_1,\dots,I_{d_{k-1}}$ is some partition of the set
$\{1,\dots,d_k\}$. We therefore have commutative diagrams
\begin{align*}\xymatrix{
E(\k)^{\otimes d_k}\ar[rr]^-{\simeq}_-{\mu}\ar[d]^-{\overline{\partial_i}} &&
E(K(V[n])_k)\ar[d]^-{{\partial_i}} \\
E(\k)^{\otimes d_{k-1}}\ar[rr]^-{\simeq}_-{\mu}&& E(K(V[n])_{k-1})
}
\end{align*}
where the horizontal isomorphisms $\mu$ are induced by the multiplication of
$E(V)$, and where $\overline{\partial_i}$ sends $x_1\otimes\dots\otimes x_{d_k}$
to $y_1\otimes\dots\otimes y_{d_{k-1}}$, with $y_j=\prod_{i\in I_j} x_i$. This proves our claim, and finishes the proof of proposition
\ref{prop-quasi-trivial-collapse}.
\end{proof}

Another useful property of quasi-trivial filtrations is their compatibility with
kernels. To be more specific,
let $A(V)$ be a filtered dg-$\PP_\k$-algebra, and denote by $Z(V)$ the
subalgebra of cycles of $A(V)$. The filtration of $A(V)$ induces a
filtration on $Z(V)$ by setting $F^iZ(V)=F^iA(V)\cap Z(V)$. Let $Z'(V)$ be the
subalgebra of cycles of $\gr A(V)$. Then we have a canonical injective morphism
of algebras:
\begin{align}\gr Z(V)\hookrightarrow Z'(V)\;.\label{eq-Z}\end{align}
In general, this morphism is not surjective, but it turns out to be the case if
the filtration of $A(V)$ is quasi-trivial.

\begin{proposition}\label{compat-Z}
Let $A(V)$ be a dg-$\PP_\k$-algebra, endowed with a quasi-trivial filtration.
Let us denote by $Z(V)$ the cycles of $A(V)$ and by $Z'(V)$ the cycles of $\gr
A(V)$.
The canonical morphism \eqref{eq-Z} is an isomorphism of graded $\PP_\k$-algebras.
\end{proposition}
\begin{proof}
It suffices to prove that the homogeneous components  $(\gr Z)^d_i(V)$ and $Z^d_i(V)$
have the same dimension for all degrees $i$ and all weights $d$, or equivalently that the maps  $\partial:A^d_i(V)\to A^d_{i-1}(V)$ and
$\gr \partial: (\gr A)^d_i(V)\to (\gr A)^d_{i-1}(V)$ have the same rank. This follows
from lemma \ref{lm-triv}(c).
\end{proof}

\subsection{Truncated polynomial algebras}\label{subsec-trunc}
The truncated polynomial algebra $Q(V)$ is the $\PP_\k$-algebra obtained as
the quotient of $S(V)$ by the ideal generated by $V^{(1)}$. 
For all $i\ge 0$ we denote by $Q(V[i])$ the truncated polynomial algebra
on a generator $V$ placed in degree $i$. This is defined in a similar way, as the
quotient of $S(V[i])$ by the ideal generated by $V^{(1)}[i]$.
Truncated polynomial algebras enjoy the following properties.
\begin{enumerate}
\item If $p=2$, $Q(V)=\Lambda(V)$  (but this
is no longer true in odd characteristics). 
\item The $\PP_\k$-algebra $Q(V)$ is an exponential functor (in particular a $\PP_\k$-bialgebra).
\item\label{it-3} Let us denote by $\phi:S(V)\to \Gamma(V)$ the unique morphism of $\PP_\k$-algebras whose restriction  $V=S^1(V)\to\Gamma^1(V)=V$ to  the summand of weight one is equal to the identity. It follows that  $Q(V)$ is equal to  the image of $\phi$. 
\item\label{it-4} $Q(V)$ is self-dual, namely, there is an isomorphism of $\PP_\k$-bialgebras $Q(V)\simeq Q^\sharp(V)$.
\end{enumerate}
All these properties are well-known and easy to check, we just indicate how to retrieve \eqref{it-4} from \eqref{it-3}. Since $\phi$ is a morphism between exponential $\PP_R$-algebras, it is actually a morphism of $\PP_R$-bialgebras. Hence its dual yields a morphism of $\PP_R$-bialgebras  $\phi^\sharp:S(V)\simeq\Gamma^\sharp(V)\to S^\sharp(V)\simeq\Gamma(V)$. Since $\phi$ and $\phi^\sharp$ coincide when restricted to  the homogeneous summand of weight $1$, they must be equal. So we obtain: $Q(V)=\mathrm{Im}\,\phi\simeq \mathrm{Im}\,\phi^\sharp = Q^\sharp(V)$.

We finish this paragraph with a slightly less known result on trunctated polynomials, namely the construction of functorial resolution of $Q(V)$. 
We equip the graded $\PP_{\k}$-algebra $S(V)\otimes
\Lambda(V^{(1)}[1])$ with a differential $\partial $ defined as the composite:
$$S^d(V)\otimes \Lambda^e(V^{(1)})\to S^d(V)\otimes V^{(1)}\otimes
\Lambda^{e-1}(V^{(1)})\to S^{d+p}(V)\otimes \Lambda^{e-1}(V^{(1)}) $$
where the first map is induced by the comultiplication in
$\Lambda(V^{(1)})$ and the second  by  composition of the inclusion
$V^{(1)}\hookrightarrow S^p(V)$ and the multiplication in $S(V)$.
The composite morphism of graded $\PP_\k$-algebras $S(V)\otimes
\Lambda(V^{(1)}[1])\twoheadrightarrow S(V)\twoheadrightarrow Q(V)$
induces a morphism of differential graded
algebras
\bee
\label{resq}
f: (S(V)\otimes \Lambda(V^{(1)}[1]),\partial)\to (Q(V),0)\;. \ee
\begin{proposition}\label{prop-qis-trunc}
The morphism $f$ is a quasi-isomorphism.
\end{proposition}
\begin{proof} The graded $\PP_\k$-algebras $S(V)\otimes \Lambda(V^{(1)}[1])$ and
$Q(V)$ are exponential functors. Hence for $V=\k^d$ there is  a commutative diagram of differential graded $\k$-algebras, whose
 vertical isomorphisms are induced by the multiplication:
$$\xymatrix{
\left(S(\k)\otimes \Lambda(\k^{(1)}[1])\right)^{\otimes d}\ar[rr]^-{f^{\otimes
d}}\ar[d]^-{\simeq}&& Q(\k)^{\otimes d}\ar[d]^-{\simeq}\\
S(\k^d)\otimes \Lambda(\k^{d\,(1)}[1])\ar[rr]^-{f}&& Q(\k^d)
}\;.$$
Thus, by the K\"unneth formula, the proof reduces to the easy case $V=\k$.
\end{proof}

\begin{example}
In characteristic 2, the weight 4 component  of  the morphism of differential
graded algebras  \eqref{resq} determines the following resolution of $
\Lambda^4(V)$, where $\partial_1(x \wedge y) = x^2 \ot y - y^2 \ot x$ and $\partial_0(xy \ot z)
= xyz^2$:
\begin{equation}
\label{s12cx}
\xymatrix{
0 \ar[r] & \Lambda^2(V^{(1)}) \ar[r]^(.47){\partial_1} & S^2(V) \ot V^{(1)}
\ar[r]^(.6){\partial_0}&
S^4(V) \ar[r]^{f_4} & \Lam^4(V) \ar[r] & 0
}
\end{equation}

\end{example}

\subsection{The principal filtration on the divided power algebra}
\label{filt-1}
We denote by $\I(V)$ the ideal of $\Gamma(V)$ generated by $V=\Gamma^1(V)$. We
call this ideal the principal ideal of $\Gamma(V)$ (although it is not strictly
speaking a principal ideal). The adic filtration relative to $\I(V)$ will be called  the
principal filtration of $\Gamma(V)$. The associated graded object is the
$\PP_\k$-algebra:
\[
\gr \Gamma(V):= \bigoplus_{n\ge 0} \gr^n (\Gamma(V)) = \bigoplus_{n\ge
0}\I(V)^n/\I(V)^{n+1}\;.
\]
In this section, we will compute in proposition
\ref{prop-filtration-Gamma} the graded object associated to this principal
filtration. The result which will be  obtained in proposition
\ref{prop-filtration-Gamma} below deserves  to be compared to
the following well-known assertion \cite[Expos\'e 9, p.9-07]{cartan}:

\begin{proposition}\label{prop-iso-Nonnat}
The choice of a basis of the finite dimensional vector space $V$ determines a
(non-natural) weight-preserving algebra  isomorphism
$$\Gamma(V)\xrightarrow[]{\simeq}
Q(V)\otimes\Gamma(V^{(1)})\;.$$
\end{proposition}
\begin{proof}
By the exponential properties of $\Gamma(V)$ and $Q(V)\otimes\Gamma(V^{(1)})$
(as in the proof of lemma \ref{lm-triv}), the proof reduces to
the case $V=\k$ which is a straightforward
computation.
\end{proof}

To describe the $\PP_\k$-algebra $\gr \Gamma(V)$, we first need to interpret $\I(V)$ as a kernel.
By the universal property of the symmetric algebra, the inclusion
$V^{(1)}\hookrightarrow S^{p^r}_\k(V)$, induces an injective morphism of $\PP_\k$-algebras:
$ S_\k(V^{(1)})\hookrightarrow S_\k(V)$.
Since $S_\k(V)$ is an exponential functor, and
since Frobenius twists are additive functors, $S_\k(V^{(1)})$ is also an
exponential functor. Thus the natural inclusion above is also a
morphism of $\PP_\k$-bialgebras and it induces by  duality an
epimorphism  of $\PP_\k$-(bi)algebras
$\Gamma(V)\twoheadrightarrow \Gamma(V^{(1)})$. 
\begin{lemma}\label{prop-cokernel} The principal ideal $\I(V)$ is the kernel of the morphism $\Gamma(V)\twoheadrightarrow \Gamma(V^{(1)})$. In other words,
the multiplication in $\Gamma(V)$ yields a short exact sequence:
\begin{align}V\otimes \Gamma(V) \xrightarrow[]{\mathrm{mult}} \Gamma(V)\twoheadrightarrow
\Gamma(V^{(1)})\to 0\;. \label{ex-sq}\end{align}
\end{lemma}
\begin{proof}
If $V=V_1\oplus V_2$, by using the exponential properties of
$\Gamma(V)$ and $\Gamma(V^{(1)})$, we obtain that \eqref{ex-sq} is isomorphic to the
short exact sequence:
$$\begin{array}{c}
V_1\otimes\Gamma(V_1)\otimes\Gamma(V_2)\oplus\\
\Gamma(V_1)\otimes V_2\otimes\Gamma(V_2)
\end{array}\to \Gamma(V_1)\otimes \Gamma(V_2)\twoheadrightarrow
\Gamma(V_1^{(1)})\otimes \Gamma(V_2^{(1)})\to 0\;.$$
Hence it suffices to check exactness for $V=\k$, which is easy.
\end{proof}
\begin{proposition}\label{prop-filtration-Gamma}
There is an
isomorphism of $\PP_\k$-algebras, which maps $\gr^n \Gamma(V)$ isomorphically onto $Q^n(V)\otimes
\Gamma(V^{(1)})$:
$$
\gr \Gamma(V)\simeq
Q(V)\otimes \Gamma(V^{(1)})\;.$$
\end{proposition}

\begin{proof}
Let $\I(V)_d^n$ be the direct summand of $\I(V)^n$
contained in $\Gamma^d(V)$, so that $$\gr^n (\Gamma^d(V))=
\I(V)_d^n/\I(V)_d^{n+1}\;.$$ 
Then $\I(V)^n_d=0$ if $d< n$, and for $d\ge n$  the multiplication
of $\Gamma(V)$
induces an epimorphism:
\begin{equation}\label{eqn-local-1}
V^{\otimes n}\otimes \Gamma^{d-n}(V)\twoheadrightarrow \I(V)^n_d\;.
\end{equation}
Since the multiplication $V^{\otimes n}\to \Gamma^n(V)$ factors
through the canonical inclusion of  $Q^n(V)$ in  $\Gamma^n(V)$, the maps
(\ref{eqn-local-1}) induce commutative diagrams:
$$\xymatrix{
Q^n(V)\otimes
V\otimes \Gamma^{d-n-1}(V)\ar[d]^-{Q^n(V)\otimes\mathrm{mult}} \ar@{->>}[r]&
\I(V)^{n+1}_{d}\ar@{^{(}->}[d]\\
 Q^n(V)\otimes \Gamma^{d-n}(V)\ar@{->>}[r]&\I(V)^n_{d}
 }\;.$$
By lemma \ref{prop-cokernel}, the cokernel of the multiplication
$V\otimes \Gamma^{d-n-1}(V)\to \Gamma^{d-n}(V)$ is equal to
$\Gamma^{(d-n)/p}(V^{(1)})$ if $p$ divides $d-n$, and to zero
otherwise. Hence, if $p$ divides $d-n$, the map $Q^n(V)\otimes
\Gamma^{d-n}(V)\twoheadrightarrow \I(V)^n_{d}$ induces an
epimorphism
$$ Q^n(V)\otimes \Gamma^{(d-n)/p}(V^{(1)})\twoheadrightarrow
\I(V)^n_d/\I(V)^{n+1}_d\;,
$$
and the quotient $\I(V)^n_d/\I(V)^{n+1}_d$ equals zero if $p$ does
not divide $d-n$. We thus have a surjective morphism of
$\PP_\k$-algebras:
\begin{equation}\label{eqn-local-2}Q(V)\otimes
\Gamma(V^{(1)})\twoheadrightarrow
\bigoplus_{n\ge 0} \gr^n \Gamma(V)\;,
\end{equation}
which sends $Q^n(V)\otimes
\Gamma(V^{(1)})$ onto $\gr^n \Gamma(V)$.
To finish the proof, we observe that epimorphism \eqref{eqn-local-2} is
actually an isomorphism for dimension reasons: it follows from
proposition
\ref{prop-iso-Nonnat} that the
direct summands of a given weight $d$ of the source and the target of the epimorphism
\eqref{eqn-local-2} have the same finite dimension.
\end{proof}

Propositions \ref{prop-iso-Nonnat} and
\ref{prop-filtration-Gamma} have the following consequence.

\begin{corollary}\label{cor-qtf}
There exists a quasi-trivial filtration of $\Gamma(V)$, and an isomorphism of
$\PP_\k$-algebras $\gr\Gamma(V)\simeq Q(V)\otimes \Gamma(V^{(1)})$.
\end{corollary}

In the previous statements, we considered the divided power algebra
$\Gamma(V)$ as a nongraded algebra (or equivalently as a graded algebra
concentrated in degree zero). But we can define an extra degree on the
divided power algebra by placing the generator $V$ in degree $i$. In that
case the statements of propositions \ref{prop-iso-Nonnat},
\ref{prop-filtration-Gamma} and corollary \ref{cor-qtf} remain valid with
`$V$' replaced by $V[i]$ and `$V^{(1)}$' replaced by $V^{(1)}[pi]$, since all
the morphisms in those propositions preserve the weights, and the extra degree
is equal to $i$ times the weight. By iterating corollary \ref{cor-qtf} we then
obtain the following result.

\begin{corollary}\label{cor-iterated-princ}For any non-negative integer $i$,
there exists a quasi-trivial filtration on $\Gamma(V[i])$, and an isomorphism of
graded $\PP_\k$-algebras
$$\gr \Gamma(V[i])\simeq \bigotimes_{r\ge 0} Q(V^{(r)}[ip^r])\;.$$
\end{corollary}

\section{The derived functors of $\Gamma^d_\k(V)$ in positive characteristic}
\label{descr-gvn}
In this section, $\k$ is a field of positive characteristic $p$. All the
functors considered are strict polynomial functors defined over $\k$. In
particular, we write $\Gamma^d$,
$\Lambda^d$ and $S^d$ for $\Gamma^d_\k$, $S^d_\k$ and $\Lambda^d_\k$.
A generic finite dimensional $\k$ vector space will be denoted by the
letter `$V$'.

The main results of this section are theorems \ref{thm-derived-p2} and \ref{thm-derived-podd}, which
describe the derived functors of $\Gamma(V)$. These results were already proved by one of us
in \cite{antoine}, where the proof was rather technical and relied heavily on the computations of
Cartan \cite{cartan}. The proofs which  we will give here are more elementary and independent
of \cite{antoine, cartan}.
We will require  theorems \ref{thm-derived-p2} and \ref{thm-derived-podd} as an input for the computations
of sections \ref{der1-gamma-z}, \ref{sec-derG12}, and \ref{der-gamma4-sec}.

\subsection{The description of $L_*\Gamma^d(V,n)$ in characteristic $2$}

\begin{theorem}\label{thm-derived-p2}
Let $\k$ be a field of characteristic $2$, and let $V$ be a finite
dimensional $\k$-vector space. For all $n\ge 1$, there is an isomorphism of
graded
$\mathcal{P}_\k$-algebras
\begin{align}
\label{lgammavn}
&L_*\Gamma(V,n)\simeq \bigotimes_{r_1,\dots, r_n\ge 0}
\Gamma \left(V^{(r_1+\dots
+r_n)}[2^{r_2+\dots+r_n}+2^{r_3+\dots+r_n}+\dots+2^{r_n}
+1]\right)\;.
\end{align}
\end{theorem}
For example, there is an isomorphism
$L_*\Gamma(V,1)\simeq \bigotimes_{r\ge 0} \Gamma(V^{(r)}[1])$.
The homogeneous component  of weight $d$ of the graded
$\PP_\k$-algebra $L_*\Gamma(V,n)$ provides us the derived functors of $\Gamma^d(V)$.
Let us spell this out   in the $d=4$  case.
\begin{example}\label{ex-G4}
For $n\ge 1$, the derived functors $L_*\Gamma^4(V,n)$ are given by the following formula (where $F(V)\,[k]$ means a copy of the strict polynomial functor $F(V)$ placed in degree $k$).
\begin{align*}
L_*\Gamma^4(V,n)\simeq &\qquad \Gamma^4(V)\;[4n]\\
&\oplus\;\bigoplus_{i=1\dots n} \Gamma^2(V)\otimes V^{(1)}\;[3n+i-1]\\
&\oplus\;\bigoplus_{1\le i<j\le n} V^{(1)}\otimes V^{(1)}\;[2n+i+j-2]\quad  \oplus \;\bigoplus_{i=1\dots n} \Gamma^{2}(V^{(1)})\;[2n+2i-2]\\
&\oplus\;\bigoplus_{1\le i\le j\le n} V^{(2)}\;[n+2i+j-3]\;.
\end{align*}
\end{example}
\begin{proof}[Explanation of example \ref{ex-G4}]
In order to unpackage the compact formula \eqref{lgammavn}, we list those generators of the graded $\mathcal{P}_\k$-algebra $L_*\Gamma(V,n)$ which can contribute (after applying a divided power functor or after  taking tensor products) to a summand of weight $4$ of $L_*\Gamma(V,n)$ .
These generators are of the following four distinct types:
\begin{enumerate}
\item[(i)] one generator $V[n]$, corresponding to the $n$-tuple $(0,\dots,0)$,
\item[(ii)] $n$ generators of the form $V^{(1)}[n+i-1]$, corresponding to the $n$-tuples $(r_1,\dots,r_n)$ with $r_i=1$, $r_{k}=0$ if $k\ne i$,
\item[(iii)] $n$ generators of the form $V^{(2)}[n+3i-3]$, corresponding to the $n$-tuples $(r_1,\dots,r_n)$ with $r_i=2$, $r_{k}=0$ if $k\ne i$,
\item[(iv)] $n(n-1)/2$ generators of the form $V^{(2)}[n+j+2i-3]$, corresponding to the $n$-tuples $(r_1,\dots,r_n)$ with $r_i=r_j=1$ for a given pair $\{i,j\}$ with $i<j$, and $r_{k}=0$ if $k\ne i,j$.
\end{enumerate}
Then we determine all possible manners in which  these generators can contribute  to  a direct summand of weight $4$ of $L_*\Gamma(V,n)$:
\begin{itemize}
\item The generator (i) can contribute to  a summand of weight $4$ in two ways namely (a) via a summand $\Gamma^4(V)$, and (b) via a summand $\Gamma^2(V)\otimes V^{(1)}$, where $V^{(1)}$ is a generator of type (ii).
\item The generators of type (ii) can contribute to a summand of weight $4$ in three ways. First of all by the method (b)  listed before, secondly via a summand $V^{(1)}\otimes V^{(1)}$ where two generators of type (ii) are involved, or thirdly via a summand $\Gamma^2(V^{(1)})$.
\item The generators of type (iii) and (iv) are already of weight $4$, hence they can only contribute to the part of weight $4$ as  summands of the form $V^{(2)}$.
\end{itemize}
Finally, we compute the degree of each of  these summands of weight $4$  and thereby  obtain the  sought-after  expression  for $L_\ast\Gamma^4(V,n)$.
\end{proof}

More generally, we may extract from theorem \ref{thm-derived-p2}
 the homogeneous component of an arbitrary given weight $d$. This 
 yields the following result.
\begin{corollary}\label{cor-derived-p2-Gd}
Let $\k$ be a field of characteristic $2$,  $d$ be a positive integer and
$V$ be a finite dimensional $\k$-vector space. There exists an isomorphism of
strict polynomial functors:
$$L_i\Gamma^d(V,n)\simeq \bigoplus_{\delta}\:
\bigotimes_{r_1,\dots,r_n\ge 0} \Gamma^{\delta(r_1,\dots,r_n)}(V^{(r_1
+\dots+r_n)}[2^{r_2+\dots+r_n}+2^{r_3+\dots+r_n}+\dots+2^{r_n}
+1]) $$
where the sum is taken over all the maps $\delta:\N^n\to \N$ satisfying the
following two summability
conditions:
\begin{align*}
& \sum_{r_1,\dots,r_n\ge 0}  \delta(r_1,\dots,r_n)2^{r_1+\dots +r_n}=d\;,& (1)\\
& \sum_{r_1,\dots,r_n\ge 0}
\delta(r_1,\dots,r_n)\left(2^{r_2+\dots+r_n}+2^{r_3+\dots+r_n}+\dots+2^{r_n}
+1\right)=i\;. &(2)
\end{align*}
\end{corollary}

\subsection{Proof of theorem \ref{thm-derived-p2}}

In this proof, we will constantly use the graded $\PP_\k$-algebras $\Gamma(V^{(r)}[i])$ and $\Lambda(V^{(r)}[i])$. To keep formulas in a compact form and to handle the degrees and the twists in a confortable way, we denote these graded $\PP_\k$-algebras respectively by $\Gamma^{(r,i]}(V)$ and $\Lambda^{(r,i]}(V)$. For example, the homogeneous summand of weight $dp^r$ and degree $di$ of  $\Gamma^{(r,i]}(V)$ is $\Gamma^d(V^{(r)})$. And the homogeneous summand of weight $dp^r$ and degree $di+j$ of  $L_*\Gamma^{(r,i]}(V,n)$ is equal to $L_{j+di}\Gamma^d(V^{(r)},n)$. With these notations, theorem
\ref{thm-derived-p2} appears the special case $r=i=0$ of the following theorem,
which is the statement that we actually prove.
\begin{theorem}\label{thm-derived-p2-prime}
Let $\k$ be a field of characteristic $2$, and let $V$ be a finite
dimensional $\k$-vector space. For all $n\ge 1$ and all $r,i\ge 0$, the graded
$\mathcal{P}_\k$-algebra $L_*\Gamma^{(r,i]}(V,n)$ is isomorphic to the tensor
product
\begin{align*}
\bigotimes_{r_1,\dots, r_n\ge 0}
\Gamma \left(V^{(r+r_1+\dots
+r_n)}[i2^{r_1+r_2+\dots+r_n}+2^{r_2+r_3+\dots+r_n}+\dots+2^{r_n}
+1]\right)\;.\end{align*}
\end{theorem}

\begin{proof}
{\bf Step 1: The quasi-trivial filtration of the divided power
algebra.}
Corollary \ref{cor-iterated-princ} yields a quasi-trivial filtration of
$\Gamma^{(r,i]}(V)$ with associated graded $\PP_\k$-algebra
\begin{align}\gr \Gamma^{(r,i]}(V)\simeq \bigotimes_{s\ge 0}\Lambda^{(r+s,i2^s]}(V)\;.\label{eq-iso2-1}\end{align}
By proposition
\ref{prop-quasi-trivial-collapse}, derivation commutes with quasi-trivial filtrations and by the 
Moreover, by the Eilenberg-Zilber theorem and the K\"unneth formula, derivation commutes with tensor products. We therefore have isomorphisms of graded $\PP_\k$-algebras
\begin{align}\gr (L_* \Gamma^{(r,i]}(V,n))\simeq L_*(\gr \Gamma^{(r,i]})(V,n)\simeq \bigotimes_{s\ge 0}
L_*\Lambda^{(r+s,i2^s]}(V,n) \;.\label{eq-iso2-2}\end{align}
{\bf Step 2: D\'ecalage.}
Now we use the d\'ecalage formula of proposition \ref{dec}:
\begin{align*} L_*\Lambda^{(r+s,i2^s]}(V,n)\simeq
L_*\Gamma^{(r+s,i2^s+1]}(V,n-1)\;  \end{align*}
to rewrite the right-hand side of isomorphism \eqref{eq-iso2-2}. In this way we
 obtain for all
$n\ge 1$ an isomorphism:
\begin{align}\gr  (L_*\Gamma^{(r,i]}(V,n))\simeq \bigotimes_{s\ge 0}
L_*\Gamma^{(r+s,i2^s+1]}(V,n-1)\label{eq-iso2-4} \end{align}
{\bf Step 3: Induction.}
We now prove theorem \ref{thm-derived-p2-prime} by induction on $n$. For $n=1$,
$L_*\Gamma^{(r+s,i2^s+1]}(V,0)$ is isomorphic to
$\Gamma^{(r+s,i2^s+1]}(V)$ so that theorem \ref{thm-derived-p2-prime} holds.
Let us assume that we have computed $L_*\Gamma^{(r+s,i2^s+1]}(V,n-1)$.
By inserting the formula giving $L_*\Gamma^{(r+s,i2^s+1]}(V,n-1)$ in the
right hand side of isomorphism \eqref{eq-iso2-4}, we obtain that the isomorphism
of theorem \ref{thm-derived-p2-prime} holds, up to a filtration.
To prove theorem \ref{thm-derived-p2-prime}, it remains to verify that the
filtration on the left-hand side of
of isomorphism \eqref{eq-iso2-4} is trivial. This is a direct consequence of the
following proposition.
\end{proof}

\begin{proposition}[{\cite[Prop 14.5]{antoine}}]\label{prop-split}
Let $\k$ be a field of characteristic $2$, and let $A(V)$ be a filtered
graded commutative $\PP_\k$-algebra, whose summands $A^d_i(V)$ are finite
dimensional.
Assume that $\gr A(V)$ is isomorphic to a tensor product of graded
$\PP_\k$-algebras of the form $\Gamma(V^{(r)}[i])$. Then there exists an
isomorphism of graded $\PP_\k$-algebras $A(V)\simeq \gr A(V)$.
\end{proposition}
\begin{proof}
The proof is already given in \cite{antoine}, we sketch it here for the sake of completeness.
The starting point is the vanishing lemma \ref{prop-vanish}, which yields an isomorphism of graded strict polynomial functors $f:A(V)\xrightarrow[]{\simeq} \gr A(V)$. The isomorphism $f$ is not an isomorphism of algebras, but we can use it to build one in the following way.
Let us first recall that $\gr A(V)$ is exponential, hence $A(V)$ also is by the same reasoning as in lemma \ref{lm-triv}.
In particular, both algebras also have a coalgebra structure determined by the multiplication, as explained in section \ref{sec-classical}. One can check that the primitives of an exponential functor form an additive functor. So the isomorphism $f$ shows that the primitives of  $A(V)$ form a direct summand of the primitives of $\gr A(V)$. Let us denote by $F(V)$ the primitive part of $\gr A(V)$, that is the direct sum of all the generators $V^{(r)}[i]$.
Then $f$ induces a monomorphism $f:A(V)\hookrightarrow F(V)$. But $\gr A(V)$ is the universal cofree coalgebra on $F(V)$, hence $f$ extends uniquely to a morphism of graded $\PP_\k$-coalgebras $\overline{f}:A(V)\to \gr A(V)$. Since $\overline{f}$ induces an injection between the primitives of $A(V)$ and those of $\gr A(V)$, it is injective, and it is an isomorphism for dimension reasons. Finally, the coalgebra structure of an exponential functor uniquely determines its algebra structure and vice versa, so the isomorphism $\overline{f}$ is also a morphism of algebras.
\end{proof}

\subsection{The computation of $L_*\Gamma^d(V,1)$ over a field of odd
characteristic}

\begin{theorem}\label{thm-derived-podd}
Let $\k$ be a field of odd characteristic $p$, and let $V$ be a finite
dimensional $\k$-vector space. There is an isomorphism of
graded
$\mathcal{P}_\k$-algebras
\begin{align}
\label{lgammav1odd}
&L_*\Gamma(V,1)\simeq \bigotimes_{r\ge 1}
\Gamma \left(V^{(r)}[2]\right) \otimes \bigotimes_{r\ge 0}\Lambda\left(V^{(r)}[1]\right)\;.
\end{align}
\end{theorem}

\begin{remark}
The reader might find it surprising that derived
functors of $\Gamma$ are so different in characteristic $2$ from
what they are in odd characteristic. However, proposition
\ref{prop-iso-Nonnat} shows that theorem \ref{thm-derived-podd} is
valid in characteristic $2$ in a nonnatural way. And proposition
\ref{prop-filtration-Gamma} shows that  theorem
\ref{thm-derived-podd} remains valid in characteristic $2$, up to a
filtration.
\end{remark}

\begin{proof}
The proof is similar to the characteristic $2$ case, i.e. to the proof of theorem \ref{thm-derived-p2-prime}.
{\bf Step 1: The quasi-trivial filtration of the divided power
algebra.}
Corollary \ref{cor-iterated-princ} yields a quasi-trivial filtration of
$\Gamma(V)$ with associated graded $\PP_\k$-algebra
$\bigotimes_{s\ge 0}Q(V^{(s)})$.
Thus, by proposition
\ref{prop-quasi-trivial-collapse} and the Eilenberg-Zilber theorem, the graded $\PP_\k$-algebra
$L_*\Gamma(V,1)$ is filtered and we have isomorphisms
\begin{align}\gr (L_* \Gamma(V,1))\simeq
L_*(\gr \Gamma)(V^{(s)},1)\simeq \bigotimes_{s\ge 0} L_* Q(V^{(s)},1)\;.\label{eq-isoodd-1bis}\end{align}
{\bf Step 2: Derived functors of truncated polynomials.}
In characteristic $2$, the derived functors of truncated polynomials (i.e. of exterior powers) can be computed by the d\'ecalage formula of proposition \ref{dec}.
This is not the case anymore in odd characteristic. The aim of step 2 is to prove an analogue of the d\'ecalage formula (for $n=1$), namely an isomorphism of
graded $\PP_\k$-algebras:
\begin{align} L_* Q(V^{(s)},1)\simeq \Lambda(V^{(s)}[1])\otimes \Gamma(V^{(s+1)}[2])\;. \label{eq-decpodd}\end{align}

Since derivation commutes with precomposition by the Frobenius twist (cf. remark \ref{rk-derfrobtw}), it suffices to prove \eqref{eq-decpodd} for $s=0$. Let us denote by $A(V)$ the differential graded $\PP_\k$-algebra  $(S(V)\otimes\Lambda(V^{(1)}[1]),\partial)$ introduced in \eqref{resq}. Thus there is a quasi-isomorphism: $f:A(V)\to Q(V)$, where the target has zero differential. By deriving this quasi-isomorphism, we obtain a morphism of differential graded $\PP_\k$-algebras:
\begin{align}\NN A(V,1)\to \NN Q(V,1)\;.\label{eq-qis-f}\end{align}
By definition, $\NN A(V,1)$ is the total object of a bigraded $\PP_\k$-algebra equipped with two differentials: the differential $\partial$ of the dg-$\PP_\k$-algebra $A(V)$ and the differential $d$ coming from the simplicial structure. Thus we have two spectral sequences of graded $\PP_\k$-algebras converging to the homology of $\NN A(V,1)$. The first one is obtained by computing first the homology along the differential $\partial$, and secondly the homology along the differential $d$. Thus, its second page is given by
$ E^2_{0,t}= L_t Q(V,1)$ and $E^2_{s,t}=0$ if $s\ne 0$, and this proves that \eqref{eq-qis-f} is a quasi-isomorphism. The second spectral sequence is given by computing first the homology along the simplicial differential $d$. By the d\'ecalage formula of proposition \ref{dec}, its first page is given by
$$ 'E^1_{s,t} =
\Lambda^{t-s}(V)\otimes \Gamma^{s}(V^{(1)})\;,
$$
with the convention that $\Lambda^{t-s}(V)$ is zero if $t-s<0$. We observe that $ 'E^1_{s,t}$ is a strict polynomial functor of weight $t+(p-1)s$. Since the differentials of the spectral sequence are weight preserving maps: $d_i:\,'E^i_{s,t}\to \,'E^i_{s+i,t-i+1}$, they must be zero for lacunary reasons. Hence $'E^1_{s,t}=\,'E^\infty_{s,t}$. Thus we can conclude that there exists a filtration on the graded-$\PP_\k$-algebra $L_*Q(V,1)$ such that the quasi-isomorphism \eqref{eq-qis-f} induces an isomorphism of graded $\PP_\k$-algebras:
$$ \gr (L_* Q(V^{(s)},1))\simeq \Lambda(V^{(s)}[1])\otimes \Gamma(V^{(s+1)}[2])\;.$$
This is almost the formula \eqref{eq-decpodd} that we want to prove. To finish the proof, we have to get rid of the filtration on the left hand side. This follows from  proposition \ref{prop-split-odd} below.

{\bf Step 3: Conclusion.}
The isomorphisms \eqref{eq-isoodd-1bis} and \eqref{eq-decpodd} together yield the formula of theorem \ref{thm-derived-podd} up to a filtration. To finish the proof of theorem \ref{thm-derived-podd}, we prove that this filtration splits by applying one more time proposition \ref{prop-split-odd}.
\end{proof}

In the course of the proof of theorem \ref{thm-derived-podd}, we have made use of the following statement whose proof is similar to the one of proposition \ref{prop-split}.

\begin{proposition}[\cite{antoine}]\label{prop-split-odd}
Let $\k$ be a field of odd characteristic, and let $A(V)$ be a filtered
graded commutative $\PP_\k$-algebra, whose summands $A^d_i(V)$ are finite
dimensional.
If $\gr A(V)$ is isomorphic to a tensor product of graded
$\PP_\k$-algebras of the form $\Gamma(V^{(r)}[2i])$ or $\Lambda(V^{(r)}[2i+1])$, then there exists an
isomorphism of graded $\PP_\k$-algebras $A(V)\simeq \gr A(V)$.
\end{proposition}
\section{The first derived functors of $\Gamma$ over the integers}
\label{der1-gamma-z}

In this section, we work over the ground ring $\Z$.
In particular, we write $\Gamma^d$,
$\Lambda^d$ and $S^d$ for $\Gamma^d_\Z$, $S^d_\Z$ and $\Lambda^d_\Z$.
A generic free finitely generated abelian group will be denoted by the
letter `$A$', and we will denote by `$A/p$' the quotient $A/pA$.
Strict polynomial functors defined over prime fields will enter the picture under the following
form: if $F\in\PP_{d,\Fp}$, the functor $A\mapsto F(A/p)$ lives in $\PP_{d,\Z}$.
In particular, Frobenius twist functors yield strict polynomial functors $(A/p)^{(r)}$.
We most often drop the parentheses and simply denote those functors by $A/p^{\,(r)}$.
We denote by $\Gamma^d_\Fp$, $\Lambda^d_\Fp$ and $S^d_\Fp$ the symmetric,
exterior and divided powers functors considered as objects of $\PP_{d,\Fp}$.

The goal of this section is to compute the derived functors $L_*\Gamma(A,1)$.
The main result in this section is theorem \ref{thm-calcul-LG-un-Z},
which gives a first description of these derived functors. In section \ref{descr-LGA1} we present and illustrate this result.  The proof will be given in section \ref{sec-calcul-LG-un-Z}, while sections  \ref{subsec-Koszuldiff} and \ref{subsec-SkewKoszuldiff} contain preliminary results which will be needed for this proof.  
We will further elaborate on this theorem in section \ref{sec-der1-gamma-z-bis} to obtain other forms of the result, in the hope that one or another  of these descriptions  will be of direct  use to the reader.

\subsection{The description of $L_*\Gamma(A,1)$}\label{descr-LGA1}

Recall from proposition \ref{lm-free-tors} that the graded $\PP_\Z$-algebra $L_*\Gamma(A,1)$ decomposes as
$$L_*\Gamma(A,1)= \underbrace{\bigoplus_{d\ge 0} L_d\Gamma^d(A,1)}_{D(A)}\oplus\underbrace{\bigoplus_{0<i< d} L_i\Gamma^d(A,1)}_{I(A)}$$
where the diagonal subalgebra $D(A)$ is isomorphic to $\Lambda(A[1])$, and where the ideal $I(A)$ has values in torsion abelian groups.
Thus, if $_\p L_*\Gamma(A,1)$ denotes the $p$-primary part of the abelian group $L_*\Gamma(A,1)$, there is an equality:
$$I(A)=\bigoplus_{\text{$p$ prime}} {_\p} L_*\Gamma(A,1)\;.$$
To describe the graded $\PP_\Z$-algebra $L_*\Gamma(A,1)$, it therefore suffices to compute the $p$-primary summands ${_\p}L_*\Gamma(A,1)$ as graded $\PP_\Z$-algebras (without unit), and to describe their $D(A)$-module structure
$D(A)\otimes {_\p L_*\Gamma(A,1)}\to {_\p L_*\Gamma(A,1)}$.

We shall describe the $p$-primary summands ${_\p} L_*\Gamma(A,1)$ by the means of `Koszul kernel algebras' (for odd $p$) and `skew Koszul kernel algebras' (for $p=2$), which we now introduce.

\begin{definition}\label{def-dKos}
Let $p$ be a prime. (1) We denote by
$\partial_{\mathrm{Kos}}$ the unique differential of graded $\PP_\Z$-algebras
on the connected algebra
$$\Lambda_\Fp(A/p[1])\otimes \bigotimes_{r\ge 1}
\left(\;\Gamma_\Fp\left(A/p^{(r)}[2]\right)\otimes
\Lambda_\Fp\left(A/p^{(r)}[1] \right)
\;\right) $$
sending the
generators $A/p^{(r)}[2]=\Gamma^1(A/p^{(r)}[2])$ identically to
$A/p^{(r)}[1]=\Lambda^1(A/p^{(r)}[1])$.

(2) The Koszul kernel algebra $K_\Fp(A/p)$ is the graded $\PP_\Z$-subalgebra
of $L_*\Gamma_\Fp(A/p,1)$ consisting of  the cycles
relative to the differential $\partial_{\mathrm{Kos}}$.
\end{definition}

The notation $\partial_{\mathrm{Kos}}$ and the name `Koszul kernel algebra' are justified by the fact (which will be proved in section \ref{subsec-Koszuldiff}) that the differential graded $\PP_\Z$-algebra $(L_*\Gamma_\Fp(A/p,1),\partial_{\mathrm{Kos}})$ is the tensor product of all the algebras $\Gamma_\Fp\left(A/p^{(r)}[2]\right)\otimes \Lambda_\Fp\left(A/p^{(r)}[1] \right)$ with a Koszul differential and of  the algebra $\Lambda_\Fp(A/p[1])$ with the  zero differential.
We will prove in corollary \ref{cor-filtr-SK} that the Koszul kernel algebra $K_\Fdeux(A/2)$ is very closely related to the following skew Koszul kernel algebra $SK_\Fdeux(A/2)$.
\begin{definition}\label{def-dSKos} (1)
We let
$\partial_{\mathrm{SKos}}$ be the unique differential of graded
$\PP_\Z$-algebras
on
$$\bigotimes_{r\ge 0}\Gamma_\Fdeux(A/2^{(r)}[1])$$
whose restriction to the each of the  summands $\Gamma^2_\Fdeux(A/2^{(r)}[1])$, $r\ge 0$ is
equal to the Verschiebung map $\Gamma^2_\Fdeux(A/2^{(r)}[1])\stackrel{\pi}{\to} A/2^{(r+1)}[1]$.

(2) The skew Koszul kernel algebra $SK_\Fdeux(A/2)$ is the graded
$\PP_\Z$-subalgebra
of $L_*\Gamma_\Fdeux(A/2,1)$ consisting of  the
cycles relative to the differential $\partial_{\mathrm{SKos}}$.
\end{definition}

The following theorem provides our  first description of the graded $\PP_\Z$-algebra $L_*\Gamma(A,1)$. It will be proved in section  \ref{sec-calcul-LG-un-Z}.

\begin{theorem}\label{thm-calcul-LG-un-Z}
Let $A$ be a finitely generated  free abelian group.
\begin{enumerate}
\item[(i)] The diagonal algebra $D(A)$ is isomorphic to $\Lambda(A[1])$.
 \item[(ii)]  For any prime number $p$, the $p$-primary component  $_\p L_*\Gamma(A,1)$ is entirely $p$-torsion. In particular, there is  an isomorphism of graded $\PP_\Z$-algebras
$$L_*\Gamma(A,1)\otimes\Fp\simeq D(A)\otimes\Fp\;\oplus \; {_\p} L_*\Gamma(A,1)\;.$$
\item[(iii)] There are  isomorphisms of
graded $\mathcal{P}_\Z$-algebras:
\begin{align}
&L_*\Gamma(A,1)\otimes\Fp\simeq K_\Fp(A/p)\quad \text{ if $p$ is an odd prime,}\\
&L_*\Gamma(A,1)\otimes\Fdeux\simeq SK_\Fdeux(A/2)\quad\text{ if $p=2$.}
\end{align}
\end{enumerate}
\end{theorem}

\begin{remark}
The description of the $D(A)$-module structure on  $_\p L_*\Gamma(A,1)$ is contained in part (iii) of theorem \ref{thm-calcul-LG-un-Z}.
Indeed, part (ii) yields an isomorphism
$$D(A)\otimes {_\p} L_*\Gamma(A,1)\simeq (D(A)\otimes\Fp)\otimes {_\p} L_*\Gamma(A,1)$$
so that the $D(A)$-module structure is obtained by restriction to $(D(A)\otimes\Fp)\otimes {_\p} L_*\Gamma(A,1)$ of the multiplication of $L_*\Gamma(A,1)\otimes\Fp$.
\end{remark}

The differentials $\partial_{\mathrm{Kos}}$ and $\partial_{\mathrm{SKos}}$ will be
very precisely described in sections \ref{subsec-Koszuldiff} and
\ref{subsec-SkewKoszuldiff} so that one can easily write down explicitly the
homogeneous component of weight $d$ of each of  the differential graded algebras
$(L_*\Gamma_\Fp(A/p,1),\partial_{\mathrm{Kos}})$ and
$(L_*\Gamma_\Fdeux(A/2,1),\partial_{\mathrm{SKos}})$, and thereby compute
explicitly $L_*\Gamma^d(A,1)$ for a given $d$. More details regarding
the systematic description of $L_*\Gamma^d(A,1)$ will be given in section \ref{sec-der1-gamma-z-bis}.

For the moment, we simply provide the  flavour of the explicit description of $L_*\Gamma^d(A,1)$ by writing down in detail the homogeneous summands of weight $d$ of $(L_*\Gamma_\Fdeux(A/2,1),\partial_{\mathrm{SKos}})$, for low $d$. The family of complexes of functors obtained here does not seem to have appeared  elsewhere in the literature.
In weight $1$, the complex consists simply of a copy of  $A/2$, placed in degree $1$. The complexes corresponding to the homogeneous summands of weight $d$ ranging from $2$ to $6$ are the following ones, where each differential can be characterized as the unique non-zero morphism having the specified strict polynomial functors as source and target:
\begin{align}\underbrace{\Gamma_\Fdeux^2(A/2)}_{\text{deg $2$}}\xrightarrow[]{f_2}
& \underbrace{A/2^{(1)}}_{\text{deg $1$}}\;,\\
\underbrace{\Gamma_\Fdeux^3(A/2)}_{\text{deg $3$}}\xrightarrow[]{f_3}
&\underbrace{A/2\otimes A/2^{(1)}}_{\text{deg $2$}}\;,\\
\underbrace{\Gamma_\Fdeux^4(A/2)}_{\text{deg $4$}}\xrightarrow[]{f_4}
&\underbrace{\Gamma_\Fdeux^2(A/2)\otimes A/2^{(1)}}_{\text{deg $3$}}\xrightarrow[]{g_4} \underbrace{\Gamma_\Fdeux^{2}(A/2^{(1)})}_{\text{deg $2$}}\xrightarrow[]{h_4} \underbrace{A/2^{(2)}}_{\text{deg $1$}}\;,\\
\intertext{}
\underbrace{\Gamma_\Fdeux^5(A/2)}_{\text{deg $5$}}\xrightarrow[]{f_5}
&\underbrace{\Gamma_\Fdeux^3(A/2)\otimes A/2^{(1)}}_{\text{deg $4$}}\xrightarrow[]{g_5} \underbrace{A/2\otimes\Gamma_\Fdeux^2(A/2^{(1)})}_{\text{deg $3$}}\xrightarrow[]{h_5} \underbrace{A/2\otimes A/2^{(2)}}_{\text{deg $2$}}\;,\\
\underbrace{\Gamma_\Fdeux^6(A/2)}_{\text{deg $6$}}\xrightarrow[]{f_6}
&\underbrace{\Gamma_\Fdeux^4(A/2)\otimes A/2^{(1)}}_{\text{deg $5$}}
\xrightarrow[]{g_6}
\underbrace{\Gamma_\Fdeux^2(A/2)\otimes\Gamma_\Fdeux^2(A/2^{(1)})}_{\text{deg $4$}}\nonumber\\
 & \qquad \xrightarrow[]{\text{\footnotesize $\left[\begin{array}{c}h_6 \\h_6'\end{array}\right]$}}
\underbrace{\begin{array}{c}
\Gamma_\Fdeux^2(A/2)\otimes A/2^{(2)}\\
\oplus \Gamma_\Fdeux^3(A/2^{(1)})
\end{array}
}_{\text{degree $3$}}
\xrightarrow[]{k_6+f_3^{(1)}}\underbrace{A/2^{(1)}\otimes A/2^{(2)}}_{\text{degree $2$}}\;.
\end{align}
More explicitly, the morphisms $f_n$ above are defined as the composites
$$ \Gamma_\Fdeux^n(A/2)\to \Gamma_\Fdeux^{n-2}(A/2)\otimes \Gamma^2_\Fdeux(A/2)\to \Gamma_\Fdeux^{n-2}(A/2)\otimes A/2^{(1)}\;, $$
where the first map is induced by the comultiplication of $\Gamma_\Fdeux(A/2)$, and the second one by the Verschiebung morphism. The morphisms $g_n$ are defined as the composites
$$ \Gamma_\Fdeux^{n-2}(A/2)\otimes A/2^{(1)}\to \Gamma_\Fdeux^{n-4}(A/2)\otimes \Gamma^2_\Fdeux(A/2)\otimes A/2^{(1)}\to \Gamma_\Fdeux^{n-4}(A/2)\otimes \Gamma_\Fdeux^2(A/2^{(1)})\;, $$
where the first map is induced by the comultiplication of $\Gamma_\Fdeux(A/2)$, and the second one by the Verschiebung morphism and the multiplication $A/2^{(1)}\otimes A/2^{(1)}\to \Gamma^2_\Fdeux(A/2^{(1)})$. The maps $h_n$ and $k_n$ are induced by the Verschiebung morphism, and the map $h'_6$ is induced by the Verschiebung morphism and the multiplication of the algebra $\Gamma_\Fdeux(A/2)$.
\begin{example}\label{ex-2tG4}
The $2$-primary component  of $L_*\Gamma^4(A,1)$ is given by
$$_{(2)}L_*\Gamma^4_\Z(A,1)=  A/2^{(2)}[1]\;\oplus\; 
\Lambda^2_\Fdeux(A/2^{(1)})[2]\;\oplus\;
\Phi^4(A)[3]\;,$$
where $\Phi^4(A)$ is the kernel of the morphism $g_4:\Gamma_\Fdeux^2(A/2)\otimes A/2^{(1)}\to\Lambda^2_\Fdeux(A/2^{(1)})$, and a term `$F(A)[i]$'  means a copy of the functor $F(A)$ placed in degree $i$.
\end{example}

\subsection{The Koszul kernel algebra}\label{subsec-Koszuldiff}
The purpose of this section is to justify definition \ref{def-dKos}, that is to define the differential $\partial_{\mathrm{Kos}}$ and to study some of its properties.
\subsubsection{Koszul algebras}
Let $M$ be a projective finitely generated module over a commutative ring $R$.
We equip the graded $\PP_R$-algebra
$\Gamma_R(M[2])\otimes\Lambda_R(M[1])$ with the
differential $d_{\mathrm{Kos}}$ defined as the  composite
$$\Gamma^d_R(M)\otimes \Lambda^e_R(M)\to \Gamma^{d-1}_R(M)\otimes M\otimes
\Lambda^{e}_R(M)\to \Gamma^{d-1}_R(M)\otimes \Lambda^{e+1}_R(M) $$
where the first map is induced by the comultiplication in
$\Gamma_R(M)$ and the second one by the multiplication in
$\Lambda_R(M)$ (if $d=0$, the differential with source
$\Lambda^e_R(M)$ is zero). The resulting commutative
differential graded $\PP_R$-algebra $(\Gamma_R(M[2])\otimes \Lambda_R(M[1]),
d_{\mathrm{Kos}})$ is called the Koszul algebra (on $M$).

\begin{proposition}\label{prop-homology-Koszul}
The homology of the Koszul algebra is equal to  $R$ in degree zero
and is zero in all other degrees.
\end{proposition}
\begin{proof}
Using the fact that $\Gamma_R(M[2])\otimes \Lambda_R(M[1])$ is exponential (proceed as in the proof of proposition \ref{prop-qis-trunc}), the proof reduces to the case elementary $M=R$.
\end{proof}

\begin{remark}
The Koszul algebra is a particular case of more general constructions
\cite{illusie, FFSS}.
Its name  is illustrated  by the fact that its summand of weight
$d$ is the dual (via the duality $^\sharp$) of  the more  familiar Koszul
complex:
\begin{equation*}
\label{kos-v}
\Lambda^d_R(M)\to S^1_R(M)\otimes \Lambda^{d-1}_R(M)\to \dots\to S^{d-1}_R(M)\otimes
\Lambda^1_R(M)\to S^{d}_R(M)\;.
\end{equation*}
\end{remark}

\subsubsection{The Koszul differential on $L_*\Gamma_{\Fp}(A/p,1)$}

Let $p$ be a prime integer. To be concise, we denote by $\LL_\Fp(A/p)$ the graded commutative $\PP_\Z$-algebra
\begin{align}\LL_\Fp(A/p):= \Lambda_\Fp(A/p [1])\,\otimes
\,\bigotimes_{r\ge
1}\left(\;\Gamma_\Fp\left(A/p^{(r)}[2]\right)\otimes
\Lambda_\Fp\left(A/p^{(r)}[1] \right)
\;\right)\;.\label{notation}\end{align}
If $p$ is odd, 
$\LL_\Fp(A/p)$ is isomorphic to the derived functors $L_*\Gamma_{\Fp}(A/p,1)$ by theorem \ref{thm-derived-podd},
but this is not the case for $p=2$. However, the algebra $\LL_\Fdeux(A/2)$ will be considered later on.
We can endow $\LL_\Fp(A/p)$ with the structure of a commutative
dg-$\PP_\Z$-algebra in the following way.
Let us consider the factor $\Lambda_\Fp(A/p [1])$ as a differential graded
algebra with zero differential, and the other factors of \eqref{notation} as
Koszul algebras on the vector spaces $A/p^{(r)}$. We define
$(\LL_{\Fp}(A/p),\partial_{\mathrm{Kos}})$ to be the tensor product of these differential graded
$\PP_\Z$-algebras. By the K\"unneth formula and proposition \ref{prop-homology-Koszul} we have:
\begin{proposition}\label{prop-HKos}
The homology of $(\LL_*\Gamma_\Fp(A/p,1),\partial_{\mathrm{Kos}})$ is isomorphic to the graded $\PP_\Z$-algebra $\Lambda_\Fp(A/p[1])$.
\end{proposition}

We will now justify definition \ref{def-dKos}, that is we will characterize the differential $\partial_{\mathrm{Kos}}$ by its values on the generators $\Gamma^1(A/p^{(r)}[2])$ of $\LL_{\Fp}(A/p)$. For this, we  use the following tool.

\begin{lemma}[Uniqueness principle]\label{lm-unique}
Let $\k$ be a field of positive characteristic $p$, and let $V$ be a finite dimensional $\k$-vector space.
Let $A(V)$ be a graded commutative $\PP_\k$-algebra of the form
$$ A(V)=\left(\bigotimes_{k}\Gamma_\k(V^{(r_k)}[i_k])\right)\otimes \left(\bigotimes_{\ell}\Lambda_\k(V^{(r_\ell)}[j_\ell])\right)\;.$$
We denote by $G(V)$ the graded functor
$$G(V)= \left(\bigoplus_{k}\Gamma^1_\k(V^{(r_k)}[i_k])\right)\oplus\left(\bigoplus_{\ell}\Lambda^1_\k(V^{(r_\ell)}[j_\ell])\right)\;,$$
and by $\pi:A(V)\twoheadrightarrow G(V)$ the surjection induced by the projections of $\Gamma_\k(V^{(r_k)}[i_k])$, resp. $\Lambda_\k(V^{(r_\ell)}[j_\ell])$ onto $V^{(r_k)}[i_k]$, resp. $V^{(r_\ell)}[j_\ell]$. Then the map $\partial\mapsto \pi\circ\partial$ induces a injection
between the set of differentials on the graded $\PP_\k$-algebra $A(V)$ and the set of degree $-1$ morphisms of graded functors $A(V)\to G(V)$:
$$\begin{array}{ccc}
\left\{\text{
\begin{tabular}{c}
differentials\\
on $A(V)$
\end{tabular}
}\right\}
& \hookrightarrow &   \mathrm{Hom}_{-1}(A(V),G(V))\\
\partial & \mapsto & \pi\circ\partial
\end{array}\;.$$
\end{lemma}
\begin{proof}[Proof of lemma \ref{lm-unique}]
Since $A(V)$ is exponential and graded commutative, the multiplication in $A(V)$ induces an isomorphism of bialgebras $\phi:A(V)\otimes A(W)\simeq A(V\oplus W)$. A morphism $\partial:A(V)\to A(V)$ of degree $-1$ is a derivation if and only if $\partial$ commutes with $\phi$, which holds if and only if $\partial^\sharp$ commutes with $\phi^\sharp$, which holds if and only if $\partial^\sharp: A(V)^\sharp\to A(V)^\sharp$ is a derivation.
By duality, lemma \ref{lm-unique} is therefore equivalent to the statement that differentials on $A^\sharp(V)$ are completely determined by their restriction $G^{\sharp}(V)\to A^\sharp(V)$. The latter statement is true since  $A^\sharp(V)$ is the free graded commutative algebra on $G^{\sharp}(V)$.
\end{proof}

\begin{proposition}\label{prop-charact}
Let $\partial$ be a differential on the graded $\PP_\Z$-algebra $\LL_{\Fp}(A/p)$.
\begin{enumerate}
\item[(i)] The differential $\partial$ is determined by its restriction to the summands $\Gamma^1_\Fp(A/p^{(r)}[2])$, for all  $r\ge 1$.
\item[(ii)] The restriction of $\partial$ sends the summand $\Gamma^1_\Fp(A/p^{(r)}[2])$ into $\Lambda^1_\Fp(A/p^{(1)}[1])$.
\end{enumerate}
\end{proposition}
\begin{proof}
To prove (ii), we observe that by definition, $\partial$ must send the summand $\Gamma^1(A/p^{(r)}[2])$ into the homogeneous summand of degree $1$ and weight $p^r$, which is equal to  $\Lambda^1(A/p^{(1)}[1])$.
We now prove (i). By lemma \ref{lm-unique}, $\partial$ is uniquely determined by the morphism
\begin{align}\pi\circ\partial_{\mathrm{Kos}}:\LL_{\Fp}(A/p)\to A/p[1]\oplus \bigoplus_{r\ge 0}(A/p^{(r)}[2]\oplus A/p^{(r)}[1])\;. \label{eq-pid}\end{align}
If we denote by $\pi_r$, resp $\pi'_r$, resp. $\pi'_0$, the canonical projection of the right-hand side of isomorphism \eqref{eq-pid} onto the summands $A/p^{(r)}[2]$,  $A/p^{(r)}[1]$, and  $A/p[1]$ respectively, then
$$\pi\circ \partial=\pi_0'\circ(\pi\circ \partial)+\sum \pi'_r\circ(\pi\circ \partial)+\sum\pi_r\circ(\pi\circ \partial)\;.$$

The source of the morphism $\pi_0'\circ (\pi\circ \partial)$ is the direct summand of weight $1$ and degree $2$ of $\LL_{\Fp}(A/p)$, which is  equal to  zero,  so  that $\pi_0'\circ(\pi\circ \partial)=0$.

The source of $\pi'_r\circ(\pi\circ \partial)$ is the summand of weight $p^r$ and degree $2$ of $\LL_{\Fp}(A/p)$, which equals
$\Gamma^1_\Fp(A/p^{(r)}[2])$ if $p$ is odd, and $\Gamma^1_\Fdeux(A/2^{(r)}[2])\oplus \Lambda^2_\Fdeux(A/2^{(r-1)}[1])$ if $p=2$. Now an easy computation shows that $\hom_{\PP_{\Fdeux}}(\Lambda^2(V^{(r-1)}), V^{(r)})$ is zero so that  for any prime $p$, $\pi'_r\circ(\pi\circ \partial)$ is determined by the restriction of $\partial$ to the direct summand $\Gamma^1_\Fp(A/p^{(r)}[2])$.

Similarly, the source of $\pi_r\circ(\pi\circ \partial)$ is the summand of weight $p^r$ and degree $3$ of $\LL_{\Fp}(A/p)$. The latter is  equal to $0$ is $p\ge 5$, to $\Lambda^3_{\mathbb{F}_3}(A/3A^{(r-1)}[1])$ is $p=3$, and to 
$$ \Gamma^1_\Fdeux(A/2^{(r-1)}[2])\otimes \Lambda^1_\Fdeux(A/2^{(r-1)}[1]) \,\oplus\, \Lambda^2_\Fdeux(A/2^{(r-2)}[1])\otimes \Lambda^1_\Fdeux(A/2^{(r-1)}[1]) $$
if $p=2$.
An easy computation shows that there are no nonzero morphisms from such functors to the functor $A/p^{(r)}$ so that  $\pi_r\circ(\pi\circ \partial)=0$.
 It follows that $\pi\circ \partial$ (hence $\partial$) is completely determined by the restriction of $\partial$ to the summands $\Gamma^1_\Fp(A/p^{(r)}[2])$.
\end{proof}

\begin{corollary}\label{cor-partialKos}
The morphism $\partial_{\mathrm{Kos}}$ is the unique differential on the graded
$\PP_\Z$-algebra $\LL_{\Fp}(A/p)$ which sends the generators
$\Gamma^1_{\fp}(A/p^{(r)}[2])$ identically to
$\Lambda^1_{\fp}(A/p[1])$.
\end{corollary}

The following variant of corollary \ref{cor-partialKos} will be useful in the proof of theorem \ref{thm-calcul-LG-un-Z}.
\begin{corollary}\label{cor-versal-prop}
Let $\partial$ be a differential of the graded $\PP_\Z$-algebra $\LL_{\Fp}(A/p)$. Assume that all the summands $\Lambda^1(A/p^{(r)}[1])$, $r\ge 1$, lie in the image of $\partial$. Then there exists an isomorphism of dg-$\PP_\Z$-algebras:
$$(\LL_{\Fp}(A/p),\partial)\simeq (\LL_{\Fp}(A/p),\partial_{\mathrm{Kos}})\;. $$
\end{corollary}
\begin{proof}
The only morphisms of strict polynomial functors $f:A/p^{(r)}\to A/p^{(r)}$ are the scalar multiples of the identity. By proposition \ref{prop-charact}, $\delta$ is completely determined by its restrictions to $\Gamma^1(A/p^{(r)}[2])$. These restrictions must be nonzero in order  to ensure that the expressions $\Lambda^1(A/p^{(r)}[1])$ lie in the image of $\delta$, so they are of the form $\lambda_r\Id$ with $\lambda_r\in \mathbb{F}_p^{\ast}$. Now the required isomorphism is induced by the automorphism of the graded $\PP_\Z$-algebra $\LL_{\Fp}(A/p)$, which sends the generators $A/p[1]$ and $A/p^{(r)}[1]$, $r\ge 1$ identically to themselves and whose restrictions to the generators $A/p^{(r)}[2]$, $r\ge 1$ are  equal to  $\lambda_r\Id$.
\end{proof}

\subsection{The skew Koszul kernel algebra}\label{subsec-SkewKoszuldiff}
The purpose of this section is to justify definition \ref{def-dSKos}, that is to define the differential $\partial_{\mathrm{SKos}}$ and to study some of its properties. We also prove that the skew Koszul kernel algebra is, up to a filtration, isomorphic to the Koszul kernel algebra introduced in  definition \ref{def-dKos}.

\subsubsection{The skew Koszul algebras in characteristic $2$}
Let $\k$ be a field of characteristic $2$, and let $V$ be a finite dimensional $\k$ vector space.
Consider the graded $\PP_\k$-algebra
$\Gamma_\k(V[1])\otimes\Gamma_\k(V^{(1)}[1])$, equipped
with the differential $d_{\mathrm{SKos}}$ defined as a composite:
$$\Gamma^d_\k(V)\otimes \Gamma^e_\k(V^{(1)})\to \Gamma^{d-2}_\k(V)\otimes
\Gamma^2_\k(V)\otimes
\Gamma^{e}_\k(V^{(1)})\to \Gamma^{d-2}_\k(V)\otimes \Gamma^{e+1}_\k(V^{(1)})
$$
where the first map is induced by the comultiplication of
$\Gamma_\k(V)$ and the second one is induced by the Verschiebung map
$\Gamma^{2}_\k(V)\twoheadrightarrow V^{(1)}$ and the multiplication
of $\Gamma_\k(V^{(1)})$ (if $d\le 1$, the differential with source
$\Gamma^0_\k(V)\otimes \Gamma^e_\k(V^{(1)})$ is zero). The resulting commutative
differential graded $\PP_\k$-algebra $(\Gamma_\k(V[1])\otimes
\Gamma_\k(V^{(1)}[1]),
d_{\mathrm{SKos}})$ will be called the \emph{skew Koszul algebra} (on $V$). This name
is justified by the following result, which is a differential graded algebra version of corollary \ref{cor-qtf}.

\begin{proposition}\label{prop-qtfiltr-SK}
Let $V$ be a finite dimensional vector space over a field $\k$ of characteristic $2$.
The tensor product of the principal filtrations of $\Gamma_\k(V[1])$ and of $\Gamma_\k(V^{(1)}[1])$ yields a quasi-trivial filtration of the skew Koszul algebra, and the associated graded object is isomorphic to the dg-$\PP_\k$-algebra:
$$\Lambda_\k(V[1])\otimes \left(\Gamma_\k(V^{(1)}[2])\otimes
\Lambda_\k(V^{(1)}[1])\right)\otimes \Gamma_\k(V^{(2)}[2])\quad,\quad\Id_{\Lambda_\k(V[1])}\otimes
d_{\mathrm{Kos}}\otimes \Id_{\Gamma_\k(V^{(2)}[2])}\;.$$
\end{proposition}
\begin{proof}
Let us denote by $(A(V),d_{\mathrm{SKos}})$ the skew Koszul algebra and by $(B(V),d)$ the tensor product of $(\Lambda_\k(V[1]),0)$, of the Koszul algebra on a generator $V^{(1)}[1]$, and of $(\Gamma_\k(V^{(2)}[2],0)$.

The tensor products of the principal filtration of $\Gamma_\k(V[1])$ and $\Gamma_\k(V^{(1)}[1])$ coincides with  the adic filtration of $A(V)$ relative to the ideal $J(V)$ generated by $\Gamma^1_\k(V[1])\oplus \Gamma^{1}_\k(V^{(1)}[1])$. By definition, the image of the differential $d_{\mathrm{SKos}}$ is contained in the image of the multiplication $A(V)\otimes V^{(1)}[2]\to A(V)$. In particular, $d_{\mathrm{SKos}}$ sends $J(V)$ to $J(V)$ so the $J(V)$-adic filtration yields a filtration on the dg-$\PP_\k$-algebra
$(A(V),d_{\mathrm{SKos}})$.

By Proposition \ref{prop-filtration-Gamma}, we have an isomorphism of graded $\PP_\k$-algebras $\gr A(V)\simeq B(V)$. We have to prove that $\gr(d_{\mathrm{SKos}})=d$. By proposition \ref{prop-charact} it suffices to show that the restriction of $\gr(d_{\mathrm{SKos}})$ to the direct summand $\Gamma^{1}_\k(V^{(1)}[2])=V^{(1)}[2]$ sends this generator identically to $V^{(1)}[1]=\Lambda^1_\k(V^{(1)}[1])$. To prove this, we write down explicitly the homogeneous  weight $2$ component of the $J(V)$-adic filtration. We have:
$$A(V)_2=\Gamma^2_\k(V[1])\oplus \Gamma^1_\k(V^{(1)}[1])\;,\quad J(V)_2= \Lambda^2_\k(V[1])\;,$$
and the component of weight $2$ of the power $J(V)^n$ is zero for all   $n\ge 2$. The restriction of $d_{\mathrm{SKos}}$ to $\Gamma^2_\k(V[1])$ is the Verschiebung map so that $\gr(d_{\mathrm{SKos}}): \Lambda^2_\k(V[1])\oplus V^{(1)}[2]\to V^{(1)}[1]$ is zero on the summand $\Lambda^2_\k(V[1])$ and maps $ V^{(1)}[2]$ identically to $V^{(1)}[1]$.
The
 fact that $(A(\k),d_{\mathrm{SKos}})$ is isomorphic to $(B(\k),d)$ is easily proved by direct inspection.
\end{proof}

\subsubsection{The skew Koszul differential on $L_*\Gamma_{\Fdeux}(A/2,1)$}
We are going to define a differential $\partial_{\mathrm{SKos}}$ on the graded
$\PP_\Z$-algebra
$$L_*\Gamma_\Fdeux(A/2,1)\simeq \bigotimes_{r\ge
0}\Gamma_\Fdeux(A/2^{(r)}[1])\;.$$
To do this, we consider for all $r\ge 0$ the factor
$\Gamma_{\f2}(A/2^{(r)}[1])\otimes \Gamma_{\f2}(A/2^{(r+1)}[1])$ as the skew Koszul
algebra (on the vector space $A/2^{(r)}$). Tensoring by identities on the left
and the right, this defines  a differential $\partial_r$ on
$L_*\Gamma_\Fdeux(A/2,1)$. In other words, each element in
$L_*\Gamma_\Fdeux(A/2,1)$ can be written as a finite tensor product of
elements $x_i\in \Gamma_{\Fdeux}(A/2^{(i)}[1])$ and $\partial_r$ is given by:
$$\partial_r(x_0\otimes \dots \otimes x_r\otimes x_{r+1}\otimes \dots \otimes
x_n)=x_0\otimes \dots \otimes d_{\mathrm{SKos}}(x_r\otimes x_{r+1})\otimes
\dots \otimes x_n\;.$$

\begin{lemma}The differentials $\partial_r$ commute with each other.
\end{lemma}
\begin{proof}
We have to check that $\partial_i\circ\partial_j=\partial_j\circ\partial_i$.
Since the $\partial_i$ are differentials of algebras, we can use the
exponential property to reduce the proof to the trivial case in which $A/2$ is one-dimensional.
\end{proof}

We define the skew Koszul differential $\partial_{\mathrm{SKos}}$ as the sum:
$\partial_{\mathrm{SKos}}=\sum_{r\ge 0}\partial_r $
(this infinite sum reduces to a finite one on each summand of given degree
and weight). We will now justify definition \ref{def-dSKos}, that is characterize
$\partial_{\mathrm{SKos}}$ is by its value on the summands $\Gamma^2_\Fdeux(A/2^{(r)})$.
The following proposition is proved in the same way as proposition \ref{prop-charact}.
\begin{proposition}\label{prop-charact-p2}
Let $\partial$ be a differential of the graded $\PP_\Z$-algebra $L_*\Gamma_{\Fdeux}(A/2,1)$.
\begin{enumerate}
\item[(i)] The differential $\partial$ is determined by its restriction to the summands $\Gamma^2_{\Fdeux}(A/2^{(r)}[1])$, for $r\ge 0$.
\item[(ii)] The restriction of $\partial$ sends the summand $\Gamma^2_{\Fdeux}(A/2^{(r)}[1])$ into $\Gamma^1_{\Fdeux}(A/2^{(r+1)}[1])$.
\end{enumerate}
\end{proposition}

\begin{corollary}
The morphism $\partial_{\mathrm{SKos}}$ is the unique differential on the graded
$\PP_\Z$-algebra
$$L_*\Gamma_\Fdeux(A/2,1)=\bigotimes_{r\ge 0}\Gamma_\Fdeux(A/2^{(r)}[1])$$
whose restriction to the summands $\Gamma^2_\Fdeux(A/2^{(r)}[1])$, $r\ge 0$
equals the Verschiebung map $\Gamma^2_\Fdeux(A/2^{(r)}[1])\to A/2^{(r+1)}[1]$.
\end{corollary}

We also have the analogue of corollary \ref{cor-versal-prop}. Since the Verschiebung is the only nonzero morphism $\Gamma^2_\Fdeux(A/2^{(r)})\to A/2^{(r+1)}$, this characteristic $2$ analogue yields a slightly stronger statement.

\begin{corollary}\label{cor-versal-prop-p2}
The differential $\partial_{\mathrm{SKos}}$ is the unique differential on the graded
$\PP_\Z$-algebra $L_*\Gamma_\Fdeux(A/2,1)$ whose image contains all the generators $\Gamma^1_{\Fdeux}(A/2^{(r)}[1])$ for $r\ge 1$.
\end{corollary}

\subsubsection{Koszul versus skew Koszul kernels} The definition of the Koszul differential on $\LL_\Fdeux(A/2)$ and of the Skew Koszul differential on $L_*\Gamma_\Fdeux(A/2,1)$ are completely parallel. We now compare explicitly these two constructions.
The following proposition follows directly from proposition \ref{prop-qtfiltr-SK}.
\begin{proposition}\label{prop-qt-SK}
The tensor product of the principal filtrations on the graded $\PP_\Z$-algebras $\Gamma_\Fdeux(A/2^{(r)}[1])$, $r\ge 1$, yields a quasi-trivial filtration on $(L_*\Gamma_{\Fdeux}(A/2,1),\partial_{\mathrm{SKos}})$, whose associated graded object is the differential graded $\PP_\Z$-algebra $(\LL_{\Fdeux}(A/2,1),\partial_{\mathrm{Kos}})$.
\end{proposition}

\begin{corollary}
\label{cor-HSKos}
The homology of $(L_*\Gamma_\Fdeux(A/2,1),\partial_{\mathrm{SKos}})$ is isomorphic to the graded $\PP_\Z$-algebra $\Lambda_\Fdeux(A/2[1])$.
\end{corollary}
\begin{proof}[Proof of corollary \ref{cor-HSKos}]
The morphism of algebras $\Lambda_\Fdeux(A/2)\hookrightarrow \Gamma_\Fdeux(A/2)$ induces a morphism of algebras
$\Lambda_\Fdeux(A/2[1])\hookrightarrow L_*\Gamma_\Fdeux(A/2,1)\;.$
The image of this morphism consists of cycles, and since $(L_*\Gamma_\Fdeux(A/2,1))^d_i$ is zero for $i>d$ there is an injective morphism
\begin{align}\Lambda_\Fdeux(A/2[1])\hookrightarrow H_*(L_*\Gamma_\Fdeux(A/2,1),\partial_{\mathrm{SKos}})\;.
\label{eq-inj}\end{align}
We want to prove that the morphism \eqref{eq-inj} is an isomorphism. For this purpose, it suffices to check that its source and its target have the same dimensions in each summand of a given weight and degree. But proposition \ref{prop-qt-SK} and lemma \ref{lm-triv}(c) yield a \emph{non functorial} isomorphism of differential graded algebras
which preserves the weights:
$$(L_*\Gamma_\Fdeux(A/2,1),\partial_{\mathrm{SKos}})\simeq
(\LL_\Fdeux(A/2),\partial_{\mathrm{Kos}}) \;.$$
By proposition \ref{prop-descr-Kos}, the homology of $(\LL_\Fdeux(A/2),\partial_{\mathrm{Kos}})$ is isomorphic to $\Lambda_\Fdeux(A/2[1])$. Hence the dimensions of the source and the target of morphism \eqref{eq-inj} agree.
\end{proof}

We also mention another consequence of proposition \ref{prop-qt-SK}, which shows that the skew Koszul kernel algebra is very close to the Koszul kernel algebra. It follows directly from the properties of quasi-trivial filtrations (lemma \ref{lm-triv} and proposition \ref{compat-Z}).
\begin{corollary}\label{cor-filtr-SK}
Let $A$ be a finitely generated free abelian group. The choice of a basis of $A$ determines a \emph{non functorial} isomorphism of algebras which  preserves the weights:
$$SK_\Fdeux(A/2)\simeq K_\Fdeux(A/2)\;.$$
Moreover, there is a filtration of the graded $\PP_\Z$ algebra $SK_\Fdeux(A/2)$ and a functorial  isomorphism of graded $\PP_\Z$-algebras
$$\gr SK_\Fdeux(A/2)\simeq K_\Fdeux(A/2)\;.$$
\end{corollary}

\subsection{Proof of theorem \ref{thm-calcul-LG-un-Z}}
\label{sec-calcul-LG-un-Z} Theorem 
\ref{thm-calcul-LG-un-Z}(i) is already known by proposition \ref{lm-free-tors}.
In this section, we will prove theorem \ref{thm-calcul-LG-un-Z}(ii) and (iii)
in corollary \ref{cor-proof-11}. The proof proceeds by running the Bockstein spectral sequence.
We first need the following result.
\begin{lemma}\label{prop-calc-G1}
Let $r$ be a positive integer, and let $p$ be a prime integer.
Then the $p$-primary part of the functor $L_1\Gamma^{p^r}(A,1)$ is equal to
$A/p^{(r)}$.
Moreover, the morphism induced by mod $p$ reduction yields an
isomorphism:
$${_\p}L_1\Gamma^d(A,1)=L_1\Gamma^d(A,1)\otimes\Fp \xrightarrow[]{\simeq}
L_1\Gamma^d_\Fp(A/p,1)\;.$$
\end{lemma}
\begin{proof}
The complex $\NN\Gamma^{p^r}(A,1)\otimes \Fp$ is isomorphic to the
complex $\NN\Gamma^{p^r}_\Fp(A/p,1)$, and $\NN\Gamma^{p^r}(A,1)$
is zero in degree zero. Thus, from the long exact sequence associated to the
short exact sequence of complexes
$$0\to \NN\Gamma^{p^r}(A,1)\xrightarrow[]{\times p} \NN\Gamma^{p^r}(A,1)\to
\NN\Gamma_\Fp^{p^r}(A/p,1)\to 0$$
we obtain  an isomorphism
$L_1\Gamma^{p^r}(A,1)\otimes\Fp \simeq
L_1\Gamma^{p^r}_\Fp(A/p,1)$.
But  we know by \cite[Korollar 10.2]{d-p} that the $p$-primary part of
$L_1\Gamma^d(A,1)$
only contains $p$-torsion. The result follows.
\end{proof}
Let $p$ be a prime number.  The Bockstein spectral sequence \cite[5.9.9]{weibel}
is a device which enables one to recover the $p$-primary part of homology of a
complex $C$ of free abelian groups from the homology of the complex
$C\otimes\Fp$ (provided one is able to compute the
differentials in the spectral sequence).
We consider  the Bockstein spectral sequence for the complex $C=\NN\Gamma(A,1)$, which computes the derived functors of $\Gamma(A)$. Since $\NN\Gamma(A,1)\otimes \Fp$ is isomorphic to $\NN\Gamma_\Fp(A/p,1)$, the Bockstein spectral sequence is a spectral sequence of graded $\PP_\Z$-algebras, starting at page $0$
$$ E^0(A)_i=L_i\Gamma_\Fp(A/p,1)\Longrightarrow \left(L_i\Gamma(A,1)/\mathrm{Torsion}\right)\otimes \Fp = \Lambda_\Fp(A/p[1])\, ,$$
 with differentials $d^r: E^r(A)_i\to E^r(A)_{i-1}$.
The differentials in this  spectral sequence are related to the $p$-primary torsion of $L_*\Gamma(A,1)$ in the following way.
\begin{enumerate}
\item[(i)] Given an integer $r\ge 0$, the $p$-torsion of $L_*\Gamma(A,1)$ is
killed by multiplication by $p^r$ if and only if $E^{r}(A)=E^\infty(A)$.
\item[(ii)] The injective map induced by the mod $p$ reduction of the complex
$L\Gamma(A,1)$
\begin{align}L_*\Gamma(A,1)\otimes\Fp\to L_*\Gamma_\Fp(A/p,1)
=E^0(A)_*\;,\label{eq-psi}\end{align}
has an image contained in the permanent cycles of the spectral sequence, and the
image of the
$p$-torsion part of $L_*\Gamma(A,1)$ lies in the image of $d^0$.
\end{enumerate}
We will now completely compute  this Bockstein spectral sequence.
\begin{proposition}\label{prop-Bockstein}
The Bockstein spectral sequence degenerates at the first page: $E^1(A)=E^\infty(A)$, and the kernel $d^0$ equals $K_\Fp(A/p)$ if $p$ is odd, and $SK_\Fdeux(A/2)$ if $p=2$.
\end{proposition}
\begin{proof}
By  condition (ii) and proposition \ref{prop-calc-G1}, we know that the
image of $d^0$ must contain the generators $A/p^{(1)}[1]$. Thus by corollary
\ref{cor-versal-prop} in odd characteristic $p$ or by corollary
\ref{cor-versal-prop-p2} in characteristic $p=2$, we know that the
dg-$\PP_\Z$-algebra is isomorphic to
$(L_*\Gamma_\Fp(A/p,1),\partial_{\mathrm{Kos}})$ if $p$ is odd and to
$(L_*\Gamma_\Fdeux(A/2,1),\partial_{\mathrm{SKos}})$ if $p=2$. This proves the
statement about the kernel of $d^0$.
Both dg-$\PP_\Z$ algebras have homology equal to $\Lambda_\Fp(A/p[1])$,
whence the degeneration of the spectral sequence at the  $E^1$ page.
\end{proof}

\begin{corollary}[Theorem \ref{thm-calcul-LG-un-Z}(ii)-(iii)]
\label{cor-proof-11}
 For all primes $p$, the $p$-primary component of $L_*\Gamma(A,1)$ consists only of $p$-torsion, and there are isomorphisms of
graded $\PP_\Z$-algebras:
\begin{align}
&L_*\Gamma(A,1)\otimes\Fp\simeq K_\Fp(A/p)\quad \text{ if $p$ is an odd prime,}\\
&L_*\Gamma(A,1)\otimes\Fdeux\simeq SK_\Fdeux(A/2)\quad\text{ if $p=2$.}
\end{align}
\end{corollary}
\begin{proof}
Since the Bockstein spectral sequence degenerates at the page $E^1$, the $p$-primary component of $L_\ast\Gamma (A,1)$ is entirely  $p$-torsion. 
In particular, if $0<i<d$ the $p$-primary component $_{\p} L_i\Gamma^d(A,1)$ is isomorphic to $L_i\Gamma^d(A,1)\otimes\Fp$. Hence the map \eqref{eq-psi} sends $_\p L_i\Gamma^d(A,1)$ to the weight $d$ and degree $i$ boundaries of $d^0$. The homogeneous part of $E^1(A)$ of degree $i$ and weight $d$ is zero (since this is the case at the   $E^\infty(A)$ level). It follows that these  boundaries are equal to the weight $d$ and degree $i$ cycles of $d^0$.
\end{proof}

\section{Further descriptions of the derived functors of $\Gamma$ over the integers}\label{sec-der1-gamma-z-bis}
We keep the notations of section \ref{der1-gamma-z}, in particular the notation $\LL_\Fp(A/p)$ from \eqref{notation}. The purpose of this section is to give a more detailed description of the derived functors $L_i\Gamma^d(A,1)$ than the one given in theorem \ref{thm-calcul-LG-un-Z}.
By theorem \ref{thm-calcul-LG-un-Z}, we know that for all $d>0$:
$$L_i\Gamma^d(A,1)\simeq \begin{cases}
0 & \text{ if $i>d$ or $i=0$,}\\
\Lambda^d(A) & \text{ if $i=d$,}\\
\displaystyle SK_\Fdeux(A/2)^d_i\,\oplus \,\bigoplus_{\text{$p$ odd prime}} K_\Fp(A/p)^d_i & \text{ if $0<i<d$,}
\end{cases}$$
where the functors $K_\Fp(A/p)^d_i$ and $SK_\Fdeux(A/2)^d_i$ denote the homogeneous components of weight $d$ and degree $i$ of the (skew) Koszul kernel algebra.
Thus, to give a more detailed description of $L_i\Gamma^d(A,1)$, we need to describe these functors more precisely.

The following description of $K_\Fp(A/p)^d_i$ follows directly from the definition of the Koszul kernel algebra and the computation of its homology in proposition \ref{prop-HKos}.
\begin{proposition}\label{prop-descr-Kos}
Let $p$ be a prime. Then
\begin{enumerate}
\item[(0)] $K_\Fp(A/p)^d_d = \Lambda_\Fp^d(A/p)$
\item[(1)] $K_\Fp(A/p)^d_i = 0$ if $i<d<p$ or if $d-p+1<i<d$.
\item[(2)] The nontrivial component $K_\Fp(A/p)^d_i$ of highest degree $i<d$ is given by
$$K_\Fp(A/p)^d_{d-p+1}=\begin{cases} \Lambda^{d-p}_\Fp(A/p)\otimes A/p^{(1)} & \text{ if $p\ne 2$,}\\
\bigoplus_{1\le k\le d/p} \Lambda^{d-kp}_\Fdeux(A/2)\otimes \Gamma_{\Fdeux}^k(A/2^{(1)}) & \text{ if $p= 2$.}
\end{cases}
 $$
\item[(3)] For  any positive integer $i<d-p+1$, $K_\Fp(A/p)^d_i$ can be equally described as:
\begin{enumerate}
\item the kernel of the map $\partial_{\mathrm{Kos}}:(\LL_\Fp(A/p))^d_i\to (\LL_\Fp(A/p))^d_{i-1}$,
\item the image of the map $\partial_{\mathrm{Kos}}:(\LL_\Fp(A/p))^d_{i+1}\to (\LL_\Fp(A/p))^d_i$,
\item the cokernel of the map $\partial_{\mathrm{Kos}}:(\LL_\Fp(A/p))^d_{i+2}\to (\LL_\Fp(A/p))^d_{i+1}$.
\end{enumerate}
\end{enumerate}
\end{proposition}
\begin{proof}
Let us describe explicitly the terms of highest degrees in the homogeneous component of weight $d$ of $\LL_\Fp(A/p)$ (an expression such as   $F(A)\,[k]$ means  a copy of the functor $F(A)$ placed in degree $k$ and  the exterior powers with a negative weight are by convention equal to zero, so that the following  direct sums are actually finite):
\begin{align*}\LL_\Fp(A/p) = &\;\Lambda^d_\Fp(A/p)\,[d]\;\;
\oplus \bigoplus_{k\ge 1} \Lambda^{d-kp}_\Fp(A/p)\otimes \Gamma_\Fp^{k}(A/p^{(1)})\,[d-kp+2k]\;\\
&\oplus \;\bigoplus_{k\ge 1} \Lambda^{d-kp}_\Fp(A/p)\otimes \Gamma^{k-1}_\Fp(A/p^{(1)})\otimes \Lambda^1_\Fp(A/p^{(1)})\,[d-kp+2k-1]\;\\
&\oplus \;\text{terms of degree $< d-kp+2k-1$.}\end{align*}
The differential $\partial_{\mathrm{Kos}}$  vanishes on the term $\Lambda^d_\Fp(A/p)[d]$, and maps the rest of the first line injectively into the terms of the second line. This proves (0) and (1). Moreover, by proposition \ref{prop-HKos} the homogeneous component of weight $d$ of the complex $(\LL_\Fp(A/p),\partial_{\mathrm{Kos}})$ is exact in degrees less than $d$. Thus, the degree $d-p+1$ component  of the kernel of $\partial_{\mathrm{Kos}}$ is given by some terms from  the first line. If $p=2$ all the terms of the second line have degree $d$, so all of them but $\Lambda^d_\Fp(A/p)[d]$ contribute to the degree $d-1$ of the kernel of $\partial_{\mathrm{Kos}}$, whereas the only contribution is the one of $\Lambda^{d-p}_\Fp(A/p)\otimes A/p^{(1)}$ if $p\geq 3$.
Finally, (3) follows from the decomposition of the homogeneous component of weight $d$ of the complex $(\LL_\Fp(A/p),\partial_{\mathrm{Kos}})$ into short exact sequences.
\end{proof}

The following description of the expressions  $SK_\Fp(A/2)^d_i$ is proved exactly in the same fashion as proposition \ref{prop-descr-Kos}.

\begin{proposition}\label{prop-descr-SKos} The following holds.
\begin{enumerate}
\item[(0)] $SK_\Fdeux(A/2)^d_d = \Lambda_\Fdeux^d(A/2)$
\item[(1)] $SK_\Fdeux(A/2)^d_{d-1}=\Phi^d(A)$, where $\Phi^1(A)=0$, $\Phi^2(A)=A/2^{(1)}$, $\Phi^3(A)=A/2\otimes A/2^{(1)}$ and for $d\ge 4$, $\Phi^d(A)$ can be equivalently  described as:
\begin{enumerate}
\item the kernel of the unique nonzero morphism (induced by the comultiplication of $\Gamma_\Fdeux(A/2)$, the Verschiebung morphism $\Gamma^2_\Fdeux(A/2)\twoheadrightarrow A/2^{(1)}$, and the multiplication $(A/2^{(1)})^{\otimes 2}\to \Gamma_\Fdeux^2(A/2^{(1)})$)
$$\Gamma^{d-2}_\Fdeux(A/2)\otimes A/2^{(1)}\to\Gamma^{d-4}_\Fdeux(A/2)\otimes\Gamma^2_\Fdeux(A/2^{(1)})\;.$$
\item the image of the unique nonzero morphism (induced by the comultiplication of $\Gamma_\Fdeux(A/2)$ and the Verschiebung morphism $\Gamma^2_\Fdeux(A/2)\twoheadrightarrow A/2^{(1)}$)
 $$\Gamma^d_\Fdeux(A/2)\to\Gamma^{d-2}_\Fdeux(A/2)\otimes(A/2)^{(1)}\;.$$
\item the cokernel of the canonical inclusion
$$\Lambda^d_\Fdeux(A/2)\hookrightarrow \Gamma^d_\Fdeux(A/2)\;.$$
\end{enumerate}
\item[(2)] More generally, for $i<d-1$, $SK_\Fdeux(A/2)^d_i$ can be equivalently  described as:
\begin{enumerate}
\item the kernel of the map $\partial_{\mathrm{SKos}}:(L_*\Gamma_\Fdeux(A/2,1))^d_i\to (L_*\Gamma_\Fdeux(A/2,1))^d_{i-1}$,
\item the image of the map $\partial_{\mathrm{SKos}}:(L_*\Gamma_\Fdeux(A/2,1))^d_{i+1}\to (L_*\Gamma_\Fdeux(A/2,1))^d_i$,
\item the cokernel of the map $\partial_{\mathrm{SKos}}:(L_*\Gamma_\Fdeux(A/2,1))^d_{i+2}\to (L_*\Gamma_\Fdeux(A/2,1))^d_{i+1}$.
\end{enumerate}
\end{enumerate}
\end{proposition}

Propositions \ref{prop-descr-Kos} and \ref{prop-descr-SKos} yield descriptions of the $L_i\Gamma^d(A,1)$ as kernels, images or cokernels of some very explicit complexes. However, most of these kernels, cokernels or images yield `new' functors which are not direct sums of some familiar functors. 
For example, by corollary \ref{cor-filtr-SK}, the functors $\Phi^d(A)$ are filtered, with associated graded object equal to the functor $K_\Fdeux(A/2)^d_{d-1}$ described in proposition \ref{prop-descr-Kos}. In particular for $d=4$ there is a short exact sequence, which is non split (as we prove it in proposition \ref{prop-nonsplit1}).
\begin{equation}
\label{seq-phi4}
0\to \Lambda^2_\Fdeux(A/2)\otimes A/2^{(1)}\to \Phi^4(A)\to \Gamma^2_\Fdeux(A/2^{(1)})\to 0\;.\end{equation}
To describe the derived functors $L_i\Gamma^d(A,1)$ in more familiar terms (in terms of classical Weyl functors), we will describe them up to a filtration in theorem \ref{thm-expr} below.

By corollary \ref{cor-filtr-SK}, the description of the derived functors $L_i\Gamma^d(A,1)$ up to a filtration reduces to describe the Koszul kernel algebras $K_\Fp(A/p)$ up to a filtration for all primes $p$. These $\PP_\Z$-graded algebras are the cycles of the tensor product of the algebra $\Lambda_\Fp(A/p[1])$ with trivial differential, and of the Koszul algebras $(\Lambda_\Fp(A/p^{(r)}[1])\otimes\Gamma_\Fp(A/p^{(r)}[2]),d_{\mathrm{Kos}})$ for $r\ge 1$. To obtain a description of the cycles of such a tensor product of complexes, we will use the following result from elementary algebra.
\begin{lemma}\label{lm-filtr-cplx}
Let $C_\bullet(A)$ and $D_\bullet(A)$ be finite exact complexes of functors with values in finite dimensional $\Fp$-vector spaces, and denote by $\delta$ the differential on the tensor product $C_\bullet(A)\otimes D_\bullet(A)$. There is a natural filtration of $\ker\delta$, whose associated graded object is given by
$$\gr((\ker\delta)_k)\simeq \bigoplus_{i+j=k-1}\mathrm{Ker}\,d^C_{i}\otimes \mathrm{Ker}\,d^D_{j}\,\oplus\, \bigoplus_{i+j=k}\mathrm{Ker}\,d^C_{i}\otimes \mathrm{Ker}\,d^D_{j}$$
\end{lemma}
\begin{proof}
The proof is very much in the spirit of the quasi-trivial filtrations of section \ref{qt-fil}.
Let us denote by $X_k(A)$ the kernel of the differential $d^C_k:C_k(A)\to C_{k-1}(A)$.
We define a two-step decreasing  filtration of each functor $C_k(A)$ by
\begin{align*}
F^{k+1}C_k(A) =0\;, \quad F^kC_k(A) = X_k(A)\;,\quad F^{k-1}C_k(A) = C_k(A)\;.
\end{align*}
Then $C_\bullet(A)$ becomes a filtered complex, and the associated graded complex is isomorphic to the split exact complex:
\begin{align}  \dots\to\underbrace{X_k(A)\oplus X_{k-1}(A)}_{\text{degree $k$}}\xrightarrow[]{(0,\Id)}\underbrace{X_{k-1}(A)\oplus X_{k-2}(A)}_{\text{degree $k-1$}}\to \dots \label{eq-splitcplx}\end{align}
Moreover there is a \emph{non functorial} isomorphism between $C_\bullet(A)$ and the split complex \eqref{eq-splitcplx}. The complex $D_\bullet(A)$ is filtered similarly. The tensor product of these two filtrations yields a filtration of the complex $C_\bullet(A)\otimes D_\bullet(A)$, whose associated graded complex $(\gr(C_\bullet(A)\otimes D_\bullet(A)),\gr \delta)$ is the tensor product of two split complexes. Moreover, the filtration of the complex $C_\bullet(A)\otimes D_\bullet(A)$ induces a filtration of its cycles $\mathrm{Ker}\, \delta$, defined  by the rule  $F^i(\mathrm{Ker}\, \delta):= \left(\mathrm{Ker}\, \delta\right)\cap F^i\left(C_\bullet(A)\otimes D_\bullet(A)\right)$, and there is a canonical injection:
\begin{align}\gr (\mathrm{Ker}\, \delta)\hookrightarrow  \mathrm{Ker}\,(\gr \delta)\label{eq-can-gr}.\end{align}
Now the complex $(\gr(C_\bullet(A)\otimes D_\bullet(A)),\gr \delta)$ is non-functorially isomorphic to the complex $(C_\bullet(A)\otimes D_\bullet(A),\delta)$, hence the ranks of the differentials $\delta$ and $\gr \delta$ are equal. Thus the source and the target of the canonical morphism \eqref{eq-can-gr} have the same dimension, and the morphism \eqref{eq-can-gr} is an isomorphism.
It follows that, up to a filtration, $\mathrm{Ker}\, \delta$ is the kernel of the differential of the tensor product of the split complex \eqref{eq-splitcplx} and the similar complex for $D_\bullet(A)$. The formula of lemma \ref{lm-filtr-cplx} follows.
\end{proof}
 The following result follows from lemma  \ref{lm-filtr-cplx}  by induction on $n$:
\begin{lemma}\label{lm-filtr-n-cplx}
Let $C^i_\bullet(A)$, $1\le i\le n$, be a family of  finite exact complexes of functors with values in finite dimensional $\Fp$-vector spaces, and let us denote by $\delta$ the induced  differential on the $n$-fold tensor product $\bigotimes_{i=1}^n C^i_\bullet(A)$. There is a natural filtration of $\ker\delta$, whose associated graded object is given by
$$\gr((\ker\delta)_k)\simeq \bigoplus_{j=0}^{n-1}\;\bigoplus_{i_1+\dots+i_n=k-j}\; \left(\;\mathrm{Ker}\,d^{C^1}_{i_1}\otimes\dots \otimes \mathrm{Ker}\,d^{C^n}_{i_n}\;\right)^{\oplus \binom{n-1}{j}}\;.$$
\end{lemma}

To state our description of the derived functors $L_i\Gamma^d(A,1)$, we need to introduce the following combinatorial device. For a fixed prime number $p$, we  can consider all possible  decompositions of a positive integer $k$ as a sum $k=\sum_i k_i p^{r_i}$ where the $k_i$ are positive integers, and the $r_i$ are distinct positive integers. Since the $r_i$ are positive, the existence of such a decomposition implies that $p$ divides $k$. There might however  be many  such decompositions. Each of these  may be identified with a finite sequence of pairs of integers $((r_1,k_1),\dots, (r_i,k_i),\dots)$ satisfying the following conditions:
\begin{enumerate}
\item[(i)] The integers $r_i$ and $k_i$ are positive.
\item[(ii)] For all $i$, $r_i<r_{i+1}$.
\item[(iii)] $\sum k_ip^{r_i}=k$.
\end{enumerate}
We denote by $\mathrm{Decomp}(p,k)$ the set of all such finite sequences.
\begin{theorem}\label{thm-expr}
Let $V$ be a finite dimensional $\Fp$-vector space and let $W_{k,\Fp}^d(V)\in\PP_{d,\Fp}$ denote the kernel of the Koszul differential
$$d_{\mathrm{Kos}}: \Gamma^k_\Fp(V)\otimes \Lambda^{d-k}_\Fp(V)\to \Gamma^{k-1}_\Fp(V)\otimes \Lambda_\Fp^{d-k+1}(V)\;.$$
By convention, $W_{k,\Fp}^d(V)$ is zero if $k>d$ or if $k<0$.
For $0<i<d$, the $p$-primary part of $L_i\Gamma^d(A,1)$ is \emph{up to a filtration} isomorphic to the direct sum:
$$\bigoplus_{k=0}^{d}  \bigoplus_{
\text{\footnotesize
$
\begin{array}{c}
((r_1,k_1)\dots (r_n,k_n))\\
\in\mathrm{Decomp}(p,k)
\end{array}
$
}
}
\bigoplus_{j=0}^{n-1}\;
\bigoplus_{
\text{
\footnotesize
$
\begin{array}{c}
i_1+\dots+i_n\\
=i+k-d-j
\end{array}
$
}
}\;\left(\Lambda^{d-k}_\Fp(A/p)\otimes\bigotimes_{\ell=1}^n W^{k_\ell}_{i_\ell-k_\ell,\Fp}(A/p^{(r_\ell)})\right)^{\,\oplus\binom{n-1}{j}}. $$
\end{theorem}
\begin{proof}
Up to a filtration, the $p$-primary part of $L_i\Gamma^d(A,1)$ is given by the functor $K_\Fp(A/p)^d_i$. Let us denote by $\kappa^d_\bullet(V)$ the complex given by the homogeneous component of weight $d$ of the Koszul algebra $(\Lambda_\Fp(V[1])\otimes\Gamma_\Fp(V[2]),d_{\mathrm{Kos}})$. By definition, $K_\Fp(A/p)^d_i$ is equal to the cycles of degree $i$ in the complex
$$ \bigoplus_{k=0}^{d} \Lambda^{d-k}_\Fp(A/p[1])\otimes \left(\bigoplus_{
(r_i,k_i)\in\mathrm{Decomp}(p,k)
} \kappa^{k_1}_\bullet(A/p^{(r_1)})\otimes \dots \otimes \kappa^{k_n}_\bullet(A/p^{(r_n)})\right)\;, $$
with the convention that direct sums over empty sets are equal to zero.
The left-hand factors $\Lambda^{d-k}_\Fp(A/p[1])$ are concentrated in degree $d-k$. We therefore  need to compute the cycles of degree $i+k-d$ of the right-hand factors. The homogeneous summand of degree $i_\ell$ of the cycles of the complex $\kappa^{k_\ell}_\bullet(A/p^{(r_\ell)})$ is $W^{k_\ell}_{i_\ell-k_\ell,\Fp}(A/p^{(r_\ell)})$.
Given a sequence $((r_1,k_1),\dots, (r_n,k_n))$, we use lemma \ref{lm-filtr-n-cplx} to verify  that the   cycles of degree $i+k-d$ in the tensor product $\kappa^{k_1}_\bullet(A/p^{(r_1)})\otimes \dots \otimes \kappa^{k_n}_\bullet(A/p^{(r_n)})$ are (up to a filtration) isomorphic to
$$\bigoplus_{j=0}^{n-1}\;\bigoplus_{i_1+\dots+i_n=(i+k-d)-j}\;
W^{k_1}_{i_1-k_1,\Fp}(A/p^{(r_1)})\otimes \cdots\otimes W^{k_n}_{i_n-k_n,\Fp}(A/p^{(r_n)})^{\;\oplus\binom{n-1}{j}}.$$
This proves the formula of theorem \ref{thm-expr}.
\end{proof}

The notation $W_{k,\Fp}^d(V)$ has been chosen in order to remind the reader that these kernels of the Koszul complex are well-known to representation theorists as Weyl functors \cite{ABW,BB}. To be more specific,  recall that given a partition $\lambda=(\lambda_1,\dots,\lambda_n)$ of $d$ and a commutative ring $R$,  the Weyl functor $W_\lambda(M)\in \PP_{d,R}$ is the dual of the Schur functor associated to the partition $\lambda$. The functor $W_\lambda(M)$ may be defined as the image of a certain composite map
$$\Gamma^{\lambda_1}(M)\otimes\dots\otimes\Gamma^{\lambda_n}(M)\hookrightarrow M^{\otimes d}\xrightarrow[]{\sigma_{\lambda'}}M^{\otimes d}\twoheadrightarrow \Lambda^{\lambda'_1}(M)\otimes\dots\otimes\Lambda^{\lambda'_k}(M)\;,$$
where $\sigma_{\lambda'}$ is a specific combinatorial isomorphism and $\lambda'=(\lambda'_1,\dots,\lambda'_k)$ is the partition dual to $\lambda$. The Weyl functor $W_\lambda(M)$ is denoted by $K_{\lambda}(M)$ in \cite{ABW,BB}. For example $K_{(d)}(M)=\Gamma^d_R(M)$ and $K_{(1^d)}(M)=\Lambda^d_R(M)$. In particular, our  functors $W^d_{k,\Fp}(V)$  are the  Weyl functors associated to hook partitions, i.e. there is an isomorphism: $W^d_{k,\Fp}(V)\simeq W_{(k+1,1^{d-k-1})}(V)$ (
see e.g. \cite[Chap III.1]{BB} for more details).

\section{The maximal filtration on  $\Gamma(A)$}\label{sec-max}
In this section, we work over the ground ring $\Z$. We use the same notations as in section \ref{der1-gamma-z}
In particular, we write $\Gamma^d$,
$\Lambda^d$ and $S^d$ for $\Gamma^d_\Z$, $S^d_\Z$ and $\Lambda^d_\Z$.
A generic free finitely generated abelian group will be denoted by the
letter `$A$'. We will denote by `$A/p$' the quotient $A/pA$
and we will abuse notations and write `$A/p^{(r)}$' instead of $(A/pA)^{(r)}$ for shortening.
The purpose in this section is to introduce the maximal filtration of $\Gamma(A)$, and the associated spectral sequence, which will be our main tool for the computation of the derived functors $L_*\Gamma^d(A,n)$ for low $d$ and all $n$. As a warm-up, we finish the section by running the spectral sequence in the baby cases $d=2$ and $d=3$.

\subsection{The maximal filtration}
We denote by $\mathcal{J}(A)$ the augmentation ideal of the divided power algebra $\Gamma(A)$:
$$\mathcal{J}(A): = \Ga^{>0}(A) =  \ker (\Ga(A) \rightarrow \Ga^0(A)=\Z)\;.$$
The adic filtration relative to $\mathcal{J}(A)$ will be called the maximal filtration  on $\Gamma(A)$ (even though $\mathcal{J}(A)$ is not strictly speaking a maximal ideal in this algebra).
The associated graded object is the
$\PP_\Z$-algebra:
\[
\gr \Gamma(A):= \bigoplus_{i\ge 0} \gr_{-i} (\Gamma(A)) = \bigoplus_{i\ge
0}\mathcal{J}(A)^i/\mathcal{J}(A)^{i+1}\;.
\]
\begin{remark}
The maximal filtration is different from the principal filtration on $\Gamma(A)$ defined in section \ref{filt-1} (compare the definition of $J(A)$ with lemma \ref{prop-cokernel}).
\end{remark}

Restricting ourselves to the weight $d$ component of $\Gamma(A)$, the principal filtration yields a filtration
of $\Ga^d(A)$, with
\[
F_{-i}\,\Gamma^d(A) := \mathcal{J}(A)^i \cap \Ga^d(A).
\]
and associated graded components
\[gr_{-i}\Gamma^d(A):= F_{-i}\Gamma^d(A)/F_{-i-1}\Gamma^d(A).\]
By definition, the terms of the filtration can be concretely described as follows \bee
\label{deffi1}
F_{-i}\Ga^d(A) = 
\mathrm{Im}( \bigoplus \Ga^{k_1}(A) \ot \dots \ot \Ga^{k_i}(A) \la \Ga^m(A)) \ee
where the sum is taken over all $i$-tuples of positive integers $(k_1,\dots,k_i)$ whose sum equals $d$. In particular $F_{-i}\Ga^d(A)=0$ for $i>d$, so that the filtration is bounded and
\bee
\label{deffi3}gr_{-d}\,\Gamma^d(A) = F_{-d}\Ga^d(A) = \mathrm{Im}(A^{\ot \,d} \la \Ga^d(A)) =
S^d(A),\ee
where  the inclusion of $S^d(A)$ into  $\Ga^d(A)$  is determined by the commutative algebra structure on $\Ga(A)$.
It is also easy to identify the graded component $gr_{-1}\Ga^d(A)$.
\begin{lemma}\label{lm-gr-1}
For any free abelian group $A$, $\gr_{-1}\Gamma^d(A)=0$ if $d$ is not a power of a prime $p$, and $\gr_{-1}\Gamma^d(A)=A/p^{(r)}$ if $d=p^r$.
\end{lemma}
\begin{proof}
The composite $\Gamma^d(A)\xrightarrow[]{comult} \Gamma^{k}(A)\otimes \Gamma^{\ell}(A)\xrightarrow[]{mult} \Gamma^d(A)$ equals  the multiplication by $\binom{d}{k}$, so by \eqref{deffi1} the integral torsion of $gr_{-1}\Ga^d(A)$ is bounded by the g.c.d. of the binomials $\binom{d}{k}$, $0<k<d$. This g.c.d. equals $p$ if $d$ is a power of a prime $p$, and $1$ otherwise. Hence
$gr_{-1}\Gamma^d(A) = 0$ if $d$ is not a power of a prime, and is a $\Fp$-vector space if $d=p^r$. In the latter case, by base change  $gr_{-1}\Gamma^d(A)$ identifies with the cokernel of the map $\bigoplus \Gamma^k_\Fp(A/pA)\otimes \Gamma^\ell_\Fp(A/pA)\xrightarrow[]{mult} \Gamma^d_\Fp(A/pA)$, where the sum is taken over all pairs $(k,\ell)$ of positive integers with $k+\ell=d$. Whence the result.
\end{proof}

For $1<i<d$, the graded components $gr_{-i}\Gamma^d(A)$ are more complicated and their description involves
new classes of functors. We will use the strict polynomial functors $\sigma_{(1,n)}(V)$, defined for $n\ge 2$ on the category of $\Fdeux$-vector spaces by
\begin{align}
&\label{sig12} \sigma_{(1,n)}(V)=\mathrm{Coker}\;\left(\Lambda^2_\Fdeux(V^{(1)})\otimes S^{n-2}_\Fdeux(V)\xrightarrow[]{u} V^{(1)}\otimes S^{n}_\Fdeux(V)\right)
\end{align}
where $u: (x\wedge y)\otimes z \mapsto   x \ot (y^2z) - y \ot (x^2z)$. The map $u$ appears in the resolution of truncated polynomials introduced in section \ref{subsec-trunc}. In particular, proposition \ref{prop-qis-trunc} implies that
$\sigma_{(1,n)}(V)$ lives in a characteristic $2$ exact sequence
\begin{align}\label{sig12bis}0\to \sigma_{(1,n)}(V)\to S^{n+2}_\Fdeux(V)\to \Lambda_\Fdeux^{n+2}(V)\to 0\;.\end{align}
\begin{remark}
The functors $\sigma_{(1,n)}(V)$ belong to a family of functors $^p\sigma^e_{(\alpha, \beta)}(V)$ introduced by F. Jean in \cite[Appendix A]{jean}. 
\end{remark}
\begin{lemma}\label{lm-gr-d+1}
Let $d\ge 4$. For any free abelian group $A$,
$\gr_{-d+1}\Gamma^d(A)\simeq \sigma_{(1,d-2)}(A/2)$.
\end{lemma}
\begin{proof}
By definition of the maximal filtration, there is a commutative diagram with exact rows (the exactness of the upper row follows from lemma \ref{lm-gr-1} for $d=2$):
$$\xymatrix{
S^{d-2}(A)\otimes S^2(A)\ar[r]^-{mult}\ar@{->>}[d]^-{mult} &
S^{d-2}(A)\otimes\Gamma^2(A)\ar[r]\ar@{->>}[d]^-{mult} &
S^{d-2}(A)\otimes A/2^{(1)}\ar[r]&0\\
F_{-d}\Gamma^d(A)\ar[r]&F_{-d+1}\Gamma^d(A)\ar[r]&gr_{-d+1}\Gamma^d(A)\ar[r]&0
}.$$
Hence the surjective morphism $S^{d-2}(A)\otimes\Gamma^2(A)\twoheadrightarrow gr_{-d+1}\Gamma^d(A)$ factors into a surjective morphism
\begin{align} S^{d-2}(A)\otimes A/2^{(1)}\twoheadrightarrow gr_{-d+1}\Gamma^d(A).\label{eq-surj}\end{align}

Now we check that the composite of the morphism $u$ appearing in \eqref{sig12} and the morphism \eqref{eq-surj} is zero. For this, let us take $x\in S^{d-4}(A)$, $y,z\in A$ and let us denote by $\overline{y}$, resp. $\overline{z}$ the image of $\gamma_2(y)$, resp $\gamma_2(z)$ in $A/2^{(1)}$. Then it suffices to check that the surjection \eqref{eq-surj} sends $xy^2\otimes \overline{z}-xz^2\otimes\overline{y}$ to zero.
This follows from the fact that the elements $xy^2\otimes \gamma_2(z)$ and $xz^2\otimes \gamma_2(y)$ are both sent to the same element $2x\gamma^2(y)\gamma_2(z)$ in $F_{-d+1}\Gamma^d(A)$.
Hence the morphism \eqref{eq-surj} factors to a surjective morphism
\begin{align} \sigma_{(1,d-2)}(A/2)\twoheadrightarrow gr_{-d+1}\Gamma^d(A).\label{eq-surj2}\end{align}

Finally we check that \eqref{eq-surj2} is an isomorphism by dimension counting (the source and the target are $\Fdeux$-vector spaces). Let $(a_1,\dots,a_n)$ be a basis of the free abelian group $A$. Then the products $a_{i_1}\dots a_{i_d}$ with $i_1\le\dots\le i_d$ form a basis of $F_{-d}\Gamma^d(A)$. A basis of $F_{-d+1}\Gamma^d(A)$ is given by the products $a_{i_1}\dots a_{i_d}$ for $i_1<\dots<i_d$, and the products $a_{i_1}\dots a_{i_{d-2}}\gamma_2(a_{i_{d-1}})$ for $i_1\le \dots \le i_{d-2}$. The latter elements can be written as the elements $\frac{1}{2}a_{i_1}\dots a_{i_d}$ for $i_1\le \dots \le i_{d}$ where at least two of the
 $i_k$ are equal. Thus the dimension of $gr_{-d+1}\Gamma^d(A)$ is equal to the number of $d$-tuples $(i_1,\dots,i_d)$ with $i_1\le \dots\le i_d$ and at least two $i_k$ are equal. By the short exact sequence \eqref{sig12bis}, this is exactly the dimension of $\sigma_{(1,n)}(A/2)$.
\end{proof}

We can identify $\gr_{-2}\Gamma^d(A)$ in a similar (and slightly simpler) fashion.

\begin{lemma}
Let $d\ge 3$. For any free abelian group $A$, $\gr_{-2}\Gamma^d(A)$ is a torsion abelian group, whose $p$-primary part $_\p\gr_{-2}\Gamma^d(A)$ is given by:
$$_\p\gr_{-2}\Gamma^d(A)= \begin{cases} 0 & \text{ if $d$ is not a sum of two powers of $p$,}\\ A/p^{(k)}\otimes A/p^{(\ell)} & \text{ if $d=p^k+p^\ell$ with $k\ne \ell$,}\\
S^2_\Fp(A/p^{(k)})& \text{ if $d=2p^k$.}
 \end{cases}$$
\end{lemma}
\begin{proof}
Let $T^2(A)=\bigoplus_{k+\ell=d}\Gamma^k(A)\otimes \Gamma^\ell(A)$ and let $T^3(A)$ denote the direct sum
$$T^3(A):=\bigoplus_{k+\ell=d}\left(\bigoplus_{k_1+k_2=k} \Gamma^{k_1}(A)\otimes\Gamma^{k_2}(A)\otimes \Gamma^\ell(A)\, \oplus\, \bigoplus_{\ell_1+\ell_2=\ell} \Gamma^{k}(A)\otimes\Gamma^{\ell_1}(A)\otimes \Gamma^{\ell_2}(A)\right).$$
The multiplication of the divided power algebra defines a map $T^3(A)\to T^2(A)$. Let us denote by $C(A)$ its cokernel. By lemma \ref{lm-gr-1}, $C(A)$ is a torsion abelian group, and its $p$-primary part equals $A/p^{(k)}\otimes A/p^{(\ell)}$ if $d=p^k+p^\ell$ for a prime $p$ and zero otherwise.
By definition of the maximal filtration, there is a commutative diagram
$$\xymatrix{
T^3(A)\ar[r]\ar[d]^-{mult}& T^2(A)\ar@{->>}[d]^-{mult}\\
F_{-3}\Gamma^d(A)\ar[r]& F_{-2}(A)
}
$$
so the canonical surjection $T^2(A)\twoheadrightarrow gr_{-2}\Gamma^d(A)$ factors into a surjective morphism
$C(A)\twoheadrightarrow gr_{-2}\Gamma^d(A)$. In particular, $gr_{-2}\Gamma^d(A)$ is a torsion abelian group, and given prime $p$, we have $_\p gr_{-2}\Gamma^d(A)=0$ if $d$ is not of the form $p^k+p^\ell$, and we have a surjective morphism
\begin{align}
A/p^{(k)}\otimes A/p^{(\ell)}\twoheadrightarrow _\p gr_{-2}\Gamma^d(A)\;. \label{eq-surj3}
\end{align}
If $k=\ell$, the commutativity of the product in $\Gamma(A)$ implies that the morphism \eqref{eq-surj3} factors further as a surjective morphism
\begin{align}
S^2_\Fp(A/p^{(k)})\twoheadrightarrow _\p gr_{-2}\Gamma^d(A)\;. \label{eq-surj4}
\end{align}
To finish the proof, it suffices to check that the morphism \eqref{eq-surj3}, resp. \eqref{eq-surj4}, are isomorphisms for $d=p^k+p^\ell$ with $k\ne \ell$, resp. $d=2p^k$. This follows from a dimension counting argument similar to the one in the proof of lemma \ref{lm-gr-d+1}.
\end{proof}

We collect  the descriptions of the graded components $gr_{-i}\Gamma^d(A)$ for low $d$ in the following example. The cases $d\le 3$ already appear in \cite{BM}.

\begin{example}
\label{ga23}
For any free abelian group  $A$, $gr_{-1}\,\Ga^1_\Z(A)$ and
\begin{align}
 &gr_{-i}\,\Ga^2(A)&& = \begin{cases}A/2^{(1)} & i=1
\label{s2}
\\ S^2 (A)   \ \ \ \ \ & i= 2
\end{cases}
\\
&gr_{-i}\,\Ga^3(A) &&= \begin{cases}A/3^{(1)}  & i=1\\ A\ot A/2^{(1)}
 &i=2\hspace{1cm}\\ S^3(A)  &  i=3 \label{s3}\end{cases}\\
&\label{filg4}
gr_{-i} \Ga^4(A) &&= \begin{cases}\ A/2^{(2)}  & i = 1\\
(A  \ot A/3^{(1)}) \oplus S^2_\Fdeux(A/2^{(1)}) & i= 2\\
\;\sigma_{(1,2)}(A/2) & i = 3\\
\;S^4(A) & i=4
\end{cases}
\end{align}
\end{example}

\subsection{The  spectral sequence associated to  the maximal  filtration}

In order to study derived functors of $\Gamma^d(A)$ for $d >2$, we consider the spectral sequence associated to the maximal filtration of $\Gamma^d(A)$. This is obtained by filtering the simplicial abelian group $\Gamma^dK(A[n])$ componentwise according to the maximal filtration. This yields the following  homological spectral sequence:
\bee
\label{filss}
E^1_{p,q} = L_{p+q}(gr_p\Ga^d)(A,n) \Longrightarrow L_{p+q}\Ga^d(A,n)
\ee
with $d^r_{p,q}: E^r_{p,q} \to E^r_{p-r,q+r-1}$ as usual.
The maximal filtration on $\Gamma^d(A)$ is bounded for
any fixed integer  $d$, so that the associated  spectral sequence \eqref{filss} converges to $ L_{\ast}\Ga^d(A,n)$ in the strong sense.

We now use this spectral sequence in order to compute the  most elementary  cases, that is the derived functors $L_i\Gamma^d(A,n)$ for $A$ free and $d=2$ or $3$.

\subsection{The  derived functors of  $\Gamma^2(A)$ for $A$ free}
For $d=2$, there are only two non-trivial terms in the graded object associated to the maximal filtration of $\Gamma^2(A)$ by \eqref{s2}, so that the spectral sequence \eqref{filss} has only two non-zero columns, namely:
\begin{align}
&E^1_{-1,q} = A/2^{(1)}\; \text{ if $q=n+1$, and zero otherwise.}
\\
&E^2_{-2,q}=L_{q-2}S^2(A,n)\;.
\end{align}
If $n=1$, then $L_iS^2(A,1)=\Lambda^2(A)$ if $i=4$ and zero otherwise. The spectral sequence degenerates at $E^1$ for lacunary reasons and we obtain:
\bee
\label{quad1}
 L_*\Ga^2(A,1) = A/2^{(1)}[1]\;\oplus\;\Lam^2(A)[2]\;.
\ee
If $n\ge 2$, $L_{i}S^2(A,n)=L_{i-4}\Gamma^2(A,n-2)$ by double d\'ecalage.
So the spectral sequence  gives a relation between $L_*\Gamma^2(A,n)$ and $L_*\Gamma^2(A,n-2)$. This gives us the following result.
\begin{proposition}\label{prop-calcul-d2Z}
For any $n \geq 0$ and all free abelian group $A$, the only nontrivial values of $ L_i\Ga^2(A,n)$ are:
\bee
\label{quadratic}
 L_i\Ga^2(A,n) = \begin{cases}A/2^{(1)} & i = n, n+2, n+4, \ldots, 2n-2 \qquad n \ \text{even}\\A/2^{(1)} & i = n, n+2, n+4, \ldots, 2n-1 \ \qquad  n\ \text{odd}
 \\ \Ga^2(A) & i= 2n \qquad n \ \text{even}\\
\Lam^2(A)& i= 2n \ \qquad  n\ \text{odd.}\end{cases} \ee
\end{proposition}
\begin{proof}
The assertion is trivial for $n=0$, and it is known for $n=1$ by \eqref{quad1}. We prove by induction that if it is
true  for  $n-2$, then  it is also valid for $n$.
Indeed, the spectral sequence degenerates at $E^1$ for lacunary reasons. Thus the result for $n$ is valid up to a filtration. And since there is at most one nonzero term in each total degree in the spectral sequence, the filtration on the abutment is trivial. 
\end{proof}

\subsection{The derived functors of $\Gamma^3(A)$ for $A$ free}
\label{subsec:ga3}
By \eqref{s3}, there are only three nontrivial terms in the graded object associated to the maximal filtration of $\Gamma^3(A)$. So for $d=3$, the spectral sequence \eqref{filss}
has only three nonzero columns, namely:
\begin{align}
E^1_{-1,q} &=
 A/3^{(1)} \quad\text{ if  $q =n+1$ and zero if $q\neq n+1$}
\\
E^1_{-2,q} &= 
A\otimes (A/2)^{(1)} \quad\text{ if $q= 4$ and zero if $ q \neq 4$}
\\
E^3_{-3,q}&=L_{q-3}S^3(A,n)
\end{align}

If $n=1$, then by d\'ecalage $L_{*}S^3(A,1)=\Lambda^3(A)[3]$, so the spectral sequence degenerates at $E^1$ for lacunary reasons, and we obtain
\bee
\label{lig3a1} L_*\Ga^3(A,1) = A/3^{(1)}[1]\;\oplus\;A \ot A/2^{(1)}[2]\;\oplus\;
\Lambda^3(A)[3]\;.
\ee
If $n\ge 2$, then $L_{i}S^3(A,n)=L_{i-6}\Gamma^3(A,n-2)$ by  double d\'ecalage, so the spectral sequence essentially gives a relation between $L_*\Gamma^3(A,n)$ and $L_*\Gamma^3(A,n-2)$. The following result is proved exactly in the same way as proposition \ref{prop-calcul-d2Z}.
\begin{proposition}
\label{genliga3} Let $A$ be a free abelian group.

{\it i}) For any odd positive integer $n$, the only non trivial values of $L_i\Gamma^3(A,n)$ are:
\bee
\label{liga3odd}
L_{i}\Ga^3(A,n) = \begin{cases}
A/3^{(1)} & i = n, n+4,n+8, \dots,3n-2\\
A \ot A/2^{(1)}& i = 2n, 2n + 2, 2n+4, \dots, 3n-1\\
\Lambda^3(A) & i=3n.
\end{cases}
\ee

{\it ii}) For any even nonnegative integer $n$, the only non trivial values of $L_i\Gamma^3(A,n)$ are:
\bee
\label{liga3even}
L_{i}\Ga^3(A,n) = \begin{cases}
A/3^{(1)} &    i= n, n+4, n+8, \dots, 3n -4\\
A \ot A/2^{(1)} & i = 2n, 2n+2, 2n+4, \dots,3n -2\\
\Ga^3(A) & i=3n.
\end{cases}
\ee
\end{proposition}

\begin{remark}
For a  general abelian group $A$ the  values of the derived functors $L_{*}\Ga^2(A,n)$ and $L_{*}\Ga^3(A,n)$ are more complicated and we  refer to \cite[Prop 4.1, Thm. 5.2]{BM} for a complete discussion. 
The present  discussion of the derived functors $L_*\Gamma^3(A,n)$ is
related to  that in \cite{BM} since the
 since the functor $W_3(A)$ emphasized there is simply the quotient group  $F_{-1}\Ga^3(A)/F_{-3}\Ga^3(A)$ for the maximal filtration.
\end{remark}

\section{ The derived functors $L_*\Gamma^4(A,1)$ and $L_*\Gamma^4(A,2)$ for $A$ free}\label{sec-derG12}
In this section, we keep the notations of section \ref{sec-max}, in particular $A$ denotes a generic free finitely generated abelian group and $\Gamma(A)$ stands for $\Gamma_\Z(A)$. We compute the derived functors $L_*\Gamma^4(A,1)$ and $L_*\Gamma^4(A,2)$, by the same methods as  in propositions \ref{prop-calcul-d2Z} and \ref{genliga3}. The situation is slightly more complicated here. Indeed, although the spectral  sequence \eqref{filss} degenerates at $E^1$ for lacunary reasons as in the computations of the derived functors of $\Gamma^2(A)$ and $\Gamma^3(A)$, we will have to solve nontrivial extension problems both to compute the $E^1$ page of the spectral sequence and  to recover the the derived functors of $\Gamma^4(A)$ from the $E^\infty$-term of the spectral sequence.

The derived functors $L_*\Gamma^4(A,1)$ were already computed in section \ref{der1-gamma-z} and \ref{sec-der1-gamma-z-bis}. The proof given here is more elementary, and independent from the techniques developed there.

\subsection{The derived functor $L_*\Gamma^4(A,1)$ for $A$ free}
\label{derga4-12}
The description of $\gr\Gamma^4(A)$ is given in example \ref{ga23}. Most of the graded terms are elementary, so that the corresponding initial terms of the spectral sequence \eqref{filss} for $m=4$ and $n=1$ are easy to compute:
\begin{align*}
E^1_{-1,q} & =  A/2^{(2)} \qquad  \text{ if $q=2$, and zero if $q\neq 2$,}\\
E^1_{-2,q}& =  A \ot A/3^{(1)}\; \oplus \;\Lam^2_\Fdeux(A/2^{(1)})\qquad\text{if $q=4$, and zero if $q \neq 4$,}\\
E^1_{-4,q}&= \Lambda^4(A) \qquad \text{if $q=8$, and zero if $q\neq 8$.}
\end{align*}
To complete the description of the first page, we have to describe the column $E^1_{-3,*}$, that is the derived functors of  $\gr_{-3}\Gamma^4(A)=\sigma_{(1,2)}(A/2)$.
The presentation \eqref{sig12} of $\sigma_{(1,2)}(A/2)$
determines for each $n$ an  exact sequence of simplicial $\f2$-vector spaces (since for $\sigma_{(1,2)}(A/2)$, the map $u$ is injective):
\begin{equation}\label{swws}
0\to \Lambda^2_\Fdeux K(A/2^{(1)}[n])\to S^2_\Fdeux K(A/2[n])\ot K(A/2^{(1)}[n])
\to \sigma_{(1,2)}K(A/2[n])\to 0 \end{equation}
 For  $n=1$ this induces by d\'ecalage    a short exact sequence
\bee
\label{gr3ga4a1}
0 \to \Lambda^2_\Fdeux( A/2)\ot (A/2)^{(1)} \to \pi_3(\sigma_{1,2}K(A/2[1])) \to \Gamma^2_{\f2}(A/2^{(1)})  \to 0
\ee
 in the category of $\f2$-vector spaces which describes  $E^1_{-3,6} $ as the middle term  in \eqref{gr3ga4a1}. It also follows from \eqref{swws} that the terms $E^1_{-3,q}$ are trivial for $q \neq 3$
 and this also shows that  $E^1_{-3,q}= 0 $ for $q\neq 6$.
The spectral sequence degenerates for lacunary reasons, and since it has only has one non-trivial term in each total degree, there is no filtration issue on the abutment. So we immediately obtain the following result.
\begin{proposition}\label{prop-GA41}
\bee
\label{l4ga4a21}
L_i\Gamma^4(A,1) = \left\{  \begin{array}{ll} A/2^{(2)} & i=1\\
(A \ot A/3^{(1)}) \oplus \Lam^2_\Fdeux(A/2^{(1)}) & i=2\\
\Lam^4(A) & i=4\end{array} \right.
\ee
and there is a short exact sequence
\bee
\label{gr3ga4a11}
0 \to \Lambda^2_\Fdeux (A/2) \ot A/2^{(1)} \to L_3\,\Gamma^4(A,1) \to \Gamma^2_{\f2} (A/2^{(1)}) \to 0
\ee
\end{proposition}
To solve  the extension issue involved in this description of $L_3\Gamma^4(A,1)$, we prove the following statement.
\begin{proposition}\label{prop-nonsplit1}
The extension \eqref{gr3ga4a11} is not split. In fact
\begin{equation}\label{extgroup} \Ext^1(\Gamma^2_\Fdeux(A/2^{(1)}), \Lambda^2_\Fdeux(A/2)\otimes A/2^{(1)})\simeq \mathbb Z/2\;,
\end{equation}
where the extension group is computed in the category of strict polynomial functors or in the category of ordinary functors. Thus $L_3\,\Gamma^4(A,1)$ can be characterized as the unique non-trivial extension of $\Gamma^2_{\f2} (A/2^{(1)})$ by $\Lambda^2_\Fdeux (A/2) \ot A/2^{(1)}$.
\end{proposition}
\begin{proof}
Suppose that the exact sequence \eqref{gr3ga4a11} is functorially split, so that the $2$-primary component $_{(2)}L_3\Gamma^4(A,2)$ is  isomorphic to $ \Lambda^2_\Fdeux(A/2)\otimes A/2^{(1)} \,\oplus \,\Gamma^2_{\f2}(A/2^{(1)})$. In that case  the universal coefficient theorem yields a short exact sequence:
\begin{align*}
0 \to & L_3\Gamma^4_\Z(A,1) \ot \Z/2 \to L_3\Gamma^4_{\f2}(A/2,1)
\to \tor(L_2\Gamma^4_\Z(A,1), \Z/2) \to 0\;.
\end{align*}
The term in the middle is the value for  $V:=A/2$ of $L_3\Gamma^4_{\f2}(V,1)$. By
example \ref{ex-G4}, this  is equal to  $\Gamma^2_\Fdeux(A/2)\otimes A/2^{(1)}$. We would therefore have an injective map
$$ \Lambda^2_\Fdeux(A/2)\otimes A/2^{(1)} \,\oplus \,\Gamma^2_{\f2}(A/2^{(1)})\hookrightarrow \Gamma^2_\Fdeux(A/2)\otimes A/2^{(1)}$$
By the methods of paragraph \ref{comput-hom}, we see that there is no non-zero morphism from $\Gamma^2_{\f2}(A/2^{(1)})$ to $\Gamma^2_\Fdeux(A/2)\otimes A/2^{(1)}$, hence
the exact sequence \eqref{gr3ga4a11} cannot be functorially split. We refer to lemma \ref{lm-calc1} for the proof of  formula \eqref{extgroup}.
\end{proof}

\subsection{The derived functor $L_*\Gamma^4(A,2)$ for $A$ free}\label{filga4}
We compute the derived functors of $\gr_{-3}\Gamma^4(A)=\sigma_{(1,2)}(A/2)$ as in section \ref{derga4-12}, that is by the exact sequence (\ref{swws}). If $n\ge 2$, there is no extension problem arising from the induced long exact sequence and we obtain
\bee
\label{lisig12}
L_i\,\sigma_{(1,2)}(A/2,n)=\begin{cases}
\Gamma^2_\Fdeux(A/2)\otimes A/2^{(1)},\ i=3n
\\A/2^{(1)}\otimes A/2^{(1)},\ 2n+2 \leq i \leq 3n-1
\\\Gamma^2_\Fdeux(A/2^{(1)}),\ i=2n+1
 \\A/2^{(2)}, n+2 \leq i \leq  2n
\end{cases}
\ee
So we can compute the $E^1$ page of the spectral sequence \eqref{filss} for $\Gamma^4$ for $n=2$. The result is displayed in table \ref{table-ga4a2} below.
\begin{table}[ht]
 {\small
\renewcommand{\arraystretch}{1.8}
\begin{tabular}{|r||c|c|c|c|c|}
\hline $E^1_{p,q}$ & $p=-4$&-3&-2&-1
\\\hline \hline
$q=12$ & $\Gamma^4(A)$ & 0 & 0 & 0
\\
\hline 11 & 0 & $0$ & $0$ & 0
\\ \hline 10 & 0 & 0 & 0 & 0
\\ \hline 9 & 0 & $\Gamma^2_\Fdeux(A/2)\otimes A/2^{(1)}$ &
0 & 0
\\ \hline 8 & 0 & $\Gamma^2_\Fdeux(A/2^{(1)})$ &
0 & 0
\\ \hline 7 & 0 & $A/2^{(2)}$ &
0 & 0
\\ \hline 6 & 0 & 0 & $\Gamma^2_{\f2}(A/2^{(1)})\oplus (A\otimes A/3^{(1)})$ & 0
\\ \hline 5 & 0 & 0 & 0 & 0
\\ \hline 4 & 0 & 0 & 0 & 0
\\ \hline 3 & 0 & 0 &
0 & $A/2^{(2)}$
\\ \hline
\end{tabular}
}
\vspace{.5cm} \caption{The initial terms of the maximal filtration spectral sequence for $L\Gamma^4(A,2)$}
\label{table-ga4a2}
\end{table}
In particular, the spectral sequence degenerates at $E^1$ for lacunary reasons.
So we  obtain the following result.
\begin{proposition}\label{prop-GA42}
For $A$ free we have:
\bee
\label{liga4table}
L_i\Gamma^4(A,2)=\begin{cases} \Gamma^4(A),\ i=8\\ 0,\ i=3,7\\
\Gamma^2_\Fdeux(A/2)\otimes A/2^{(1)},\ i=6\\
\Gamma^2_{\f2}(A/2^{(1)}),\ i=5\\ A/2^{(2)},\
i=2,\end{cases}
\ee
as well as  a short exact sequence \bee \label{l4ga4a2} 0\to
A/2^{(2)}\to L_4\Gamma^4(A,2)\to \Gamma^2_{\f2}(A/2^{(1)})\oplus (A\otimes A/3^{(1)})\to 0. \ee
\end{proposition}

To complete the description of $L_*\Gamma^4(A,2)$, we have to solve the extension issue involved in the description of the functor $L_4\Gamma^4(A,2)$. This is the purpose of the following result.
\begin{proposition}\label{prop-nonsplit2}
The extension  \eqref{l4ga4a2} is not split. Actually we have
\begin{equation}\label{extgroup2} \Ext^1_{\PP_\Z}(\Gamma^2_\Fdeux(A/2^{(1)}), A/2^{(2)}\,\oplus\, A\otimes A/3^{(1)})=\mathbb Z/2\;,
\end{equation}
so that $L_4\,\Gamma^4(A,2)$ can be characterized as the unique nontrivial extension of $\Gamma^2_{\f2} (A/2^{(1)})$ by $A/2^{(2)}\oplus A\otimes A/3^{(1)}$. In particular, we obtain
$$L_4\Gamma^4(A,2)\simeq \Gamma^2_\Z(A/2^{(1)})\oplus A\ot A/3^{(1)}\;.$$
\end{proposition}
\begin{proof}The universal coefficient theorem yields a short exact sequence:
\begin{align}
0 \to L_5\Gamma^4_\Z(A,2) \ot \Z/2 \to L_5\Gamma^4_{\f2}(A/2,2)
\to \tor(L_4\Gamma^4_\Z(A,2), \Z/2) \to 0\;.
\label{kunl5}
\end{align}
The term in the middle was already computed in
example \ref{ex-G4}. If the extension  \eqref{l4ga4a2} was functorially split, then $_{(2)}L_4\Gamma^4(A,2)$ would be isomorphic to $ A/2^{(2)} \,\oplus \,\Gamma^2_{\f2}(A/2^{(1)})$, so that the short exact sequence \eqref{kunl5} could be restated as
\[
\xymatrix{
0 \ar[r] &\Gamma^2_{\f2}(A/2^{(1)}) \ar[r]^(.35)i &  \left(A/2^{(1)} \ot A/2^{(1)} \right)\oplus A/2^{(2)}
\ar[r]^(.55)j &  A/2^{(2)} \,\oplus \,\Gamma^2_{\f2}(A/2^{(1)}) \ar[r] & 0}
\]
A dimension count in the category of $\f2$-vector spaces  makes it clear that such a short exact sequence cannot exist. It follows that  \eqref{l4ga4a2} is not split.
The formula \eqref{extgroup2} and the identification of $L_4\Gamma^4(A,2)$ follow from  lemma \ref{lm-calc2}.
\end{proof}


\section{The derived functors of $\Gamma^4(A)$ for $A$ free}
\label{der-gamma4-sec}
In this section we assume that $A$ is a free abelian group. We will
give a complete description of the derived functors
$L_i\Gamma^4(A,n)$ for all $i \geq 0 $ and  $n\geq 1$.

\subsection{The description of $L_*\Gamma^4(A,n)$ for $A$ free}
The main result of section \ref{der-gamma4-sec} is the following computation.
\begin{theorem}\label{thm-G4Z}
Let $n$ be a positive integer. If $n=2m+1$, then we have an  isomorphism of graded strict polynomial functors (where $F(A)[k]$ denotes a copy of $F(A)$ placed in degree $k$, and sums over empty sets mean zero): 
\begin{align*}
L_*\Gamma^4(A,n) \simeq & \Lambda^4(A)[4n] \oplus \Phi^4(A)[4n-1]\oplus \bigoplus_{i=0}^{m-1}\Gamma^2_\Fdeux(A/2)\otimes A/2^{(1)}[3n+2i]\\
&  \oplus \bigoplus_{i=0}^{m} \Lambda^2_\Fdeux(A/2^{(1)})[2n+4i] \oplus \bigoplus_{i=0}^{m-1}\Gamma^2_\Fdeux(A/2^{(1)})[2n+4i+1]\\
& \oplus \bigoplus_{i=0}^{m-1}\bigoplus_{j=2i}^{n-3}A/2^{(1)}\otimes A/2^{(1)}[2n+2i+j+2]
 \oplus \bigoplus_{i=0}^{m}A/3\otimes A/3^{(1)}[2n+4i]\\
& \oplus \bigoplus_{i=0}^{m} A/2^{(2)}[n+6i]\oplus \bigoplus_{i=0}^{m-1}\bigoplus_{j=2i}^{n-2}A/2^{(2)}[n+4i+j+2]\;,
\end{align*}
Here $\Phi^4(A):=L_3\Gamma^4(A,1)$ is, as shown in propositions \ref{prop-GA41} and \ref{prop-nonsplit1}, the unique nontrivial extension of $\Gamma^2_\Fdeux(A/2^{(1)})$ by $\Lambda^2_\Fdeux(A/2)\otimes A/2^{(1)}$.
 Similarly, if $n=2m$ there is an isomorphism of graded strict polynomial functors:
\begin{align*}
L_*\Gamma^4(A,n) \simeq & \Gamma^4(A)[4n] \oplus \bigoplus_{i=0}^{m-1}\Gamma^2_\Fdeux(A/2)\otimes A/2^{(1)}[3n+2i]\\
&  \oplus \bigoplus_{i=0}^{m-1} \Gamma^2_\Z(A/2^{(1)})[2n+4i] \oplus \bigoplus_{i=0}^{m-1}\Gamma^2_\Fdeux(A/2^{(1)})[2n+4i+1]\\
& \oplus \bigoplus_{i=0}^{m-2}\bigoplus_{j=2i}^{n-3}A/2^{(1)}\otimes A/2^{(1)}[2n+2i+j+2]
 \oplus \bigoplus_{i=0}^{m-1}A/3\otimes A/3^{(1)}[2n+4i]\\
& \oplus \bigoplus_{i=0}^{m-1} A/2^{(2)}[n+6i]\oplus \bigoplus_{i=0}^{m-2}\bigoplus_{j=2i}^{n-3}A/2^{(2)}[n+4i+j+2]\;.
\end{align*}
\end{theorem}

Although the formulas describing the derived functors  $L_*\Gamma^4(A,n)$ look very similar for $n$ even and odd, there are at least two major differences for the 
$2$-primary part. First of all, if $n$ is even there is some $4$-torsion in $L_*\Gamma^4(A,n)$, provided by the summands $\Gamma_\Z^2(A/2^{(1)})$. On the contrary, when  $n$ is odd, $L_*\Gamma^4(A,n)$  has only $2$-torsion. Secondly, the functor $\Phi^4(A)$ appears as a direct summand (with multiplicity one) in $L_*\Gamma^4(A,n)$ when  $n$ is odd  and does not appear in the formula when  $n$ is even. We observe that the formula for $L_*\Gamma^4_\Fdeux(V,n)$ in example \ref{ex-G4} did not depend on the parity of $n$.

For $n=1$, theorem \ref{thm-G4Z} is equivalent to proposition \ref{prop-GA41}, and for $n=2$, it is equivalent to proposition \ref{prop-GA42}. Now theorem \ref{thm-G4Z} easily follows, by induction on $n$, from the following statement.
\begin{theorem}\label{deriveddescr}
Let $n\ge 3$. If $n$ is odd, there is an isomorphism of graded strict polynomial functors
\begin{multline}
\label{desg4}
L_\ast \Gamma^4(A,n)=L_\ast\Gamma^4(A,n-2)[8]\oplus\bigoplus_{i=n,\ i\neq
n+1}^{2n}A/2^{(2)}[i]\oplus
\bigoplus_{i=2n+2}^{3n-1}A/2^{(1)}\otimes A/2^{(1)}[i]\oplus\\
\Gamma^2_\Fdeux(A/2)\otimes A/2^{(1)}[3n]\oplus \Lambda^2_\Fdeux(A/2^{(1)})[2n]\oplus \Gamma^2_\Fdeux(A/2^{(1)})[2n+1]\oplus A\otimes
A/3[2n].
\end{multline}
Similarly, if $n$ is even, there is an isomorphism
\begin{multline}
\label{desg41}
L_\ast\Gamma^4(A,n)=L_\ast\Gamma^4(A,n-2)[8]\oplus\bigoplus_{i=n,\ i\neq
n+1}^{2n-1}A/2^{(2)}[i]\oplus
\bigoplus_{i=2n+2}^{3n-1}A/2^{(1)}\otimes A/2^{(1)}[i]\oplus\\
\Gamma^2_\Fdeux(A/2)\otimes A/2^{(1)}[3n]\oplus \Gamma^2_\Z(A/2^{(1)})[2n]\oplus
\Gamma^2_\Fdeux(A/2^{(1)})[2n+1]\oplus A\otimes A/3[2n]
\end{multline}
\end{theorem}
The remainder of section \ref{der-gamma4-sec} is devoted to the proof of theorem \ref{deriveddescr}. The proof goes along the same lines as the computation of $L_*\Gamma^4(A,1)$ and $L_*\Gamma^4(A,2)$ in section \ref{sec-derG12}, that is the relation between $L_\ast\Gamma^4(A,n)$ and $L_\ast\Gamma^4(A,n-2)$ is provided by the analysis of the spectral sequence \eqref{filss} induced by the maximal filtration of $\Gamma^4(A)$. However, for $n\ge 3$, the analysis of the spectral sequence is more delicate than in section \ref{sec-derG12} as there now are nontrivial differentials in the spectral sequence. Also, we will have to solve extension problems in order to recover $L_\ast\Gamma^4(A,n)$ from the $E^\infty$-page of the spectral sequence.

\subsection{Proof of theorem \ref{deriveddescr}}
From now on $n\ge 3$  and we assume that theorem \ref{thm-G4Z} has been  proved for $n-2$. We are going to prove that theorem \ref{deriveddescr} holds for $n$. The first page of the maximal filtration spectral sequence
(\ref{filss}) for $d=4$ and $n \ge 3$  can be computed by the  methods
introduced in paragraphs \ref{derga4-12} and  \ref{filga4} for $n= 1, 2$. In particular, the terms in the $p= -3$ column follow from  \eqref{lisig12} and those in the $p= -4$ column from the double d\'ecalage formula \eqref{dec}. The $E^1$ page therefore
has the form depicted in table 2.
\begin{table}[H]
 {\small
\renewcommand{\arraystretch}{1.8}
\begin{tabular}{|r||c|c|c|c|c|}
\hline $E^1_{p,q}$ & $p=-4$&-3&-2&-1
\\\hline \hline
$q=4n+4$ & $L_{4n-8}\Gamma^4(A,n-2)$ & 0 & 0 & 0
\\
\hline $4n+3$ & $L_{4n-9}\Gamma^4(A,n-2)$ & $0$ & $0$ & 0
\\ \hline \dots & \dots & \dots & \dots & \dots
\\ \hline $3n+4$ & $L_{3n-8}\Gamma^4(A,n-2)$ & 0 &
0 & 0
\\ \hline 3n+3 & $L_{3n-9}\Gamma^4(A,n-2)$ & $\Gamma^2_{\f2}(A/2)\otimes A/2^{(1)}$ &
0 & 0
\\ \hline 3n+2 & $L_{3n-10}\Gamma^4(A,n-2)$ & $A/2^{(1)}\otimes A/2^{(1)} $ &
0 & 0
\\ \hline \dots & \dots & \dots & \dots & \dots
\\ \hline 2n+5 & $L_{2n-7}\Gamma^4(A,n-2)$ & $  A/2^{(1)}\otimes A/2^{(1)} $ & 0 & 0
\\ \hline 2n+4 & $\mathbf{L_{2n-8}\Gamma^4(A,n-2)}$ & $\Gamma^2_{\f2}(A/2^{(1)})$ &
0 & 0
\\ \hline 2n+3 & $L_{2n-9}\Gamma^4(A,n-2)$ & $\mathbf{A/2^{(2)}}$ & 0 &
0
\\ \hline 2n+2 & $L_{2n-10}\Gamma^4(A,n-2)$ & $A/2^{(2)}$ & $\begin{array}{c}\mathbf{\Gamma^2_{\f2}(A/2^{(1)})} \\\oplus \mathbf{A\otimes A/3^{(1)}}\end{array}$ & 0
\\ \hline 2n+1 & $L_{2n-11}\Gamma^4(A,n-2)$ & $A/2^{(2)}$ & $A/2^{(2)}$ & $\mathbf{0}$
\\ \hline \dots & \dots & \dots & \dots & \dots
\\ \hline n+5 & $L_{n-7}\Gamma^4(A,n-2)=0$ & $A/2^{(2)}$ & $A/2^{(2)}$ & 0
\\ \hline n+4 & 0 & 0 & $A/2^{(2)}$ & 0
\\ \hline n+3 & 0 & 0 & 0 & 0
\\ \hline n+2 & 0 & 0 & 0 & 0
\\ \hline n+1 & 0 & 0 & 0 & $A/2^{(2)}$
\\ \hline
\end{tabular}
}
\vspace{.5cm} \caption{The initial terms of the maximal filtration spectral sequence for $\Gamma^4(A,n)$ with $n >2$}
\label{table-ga4an}\end{table}
The  boldface expressions  in table 2 are the terms of total degree $2n$. They will play a special role in the proof. We will indeed prove that all the differentials of the spectral sequence are zero, except some of the differentials with terms of total degree $2n$ as source or target.
In order to obtain some  information regarding the differentials of the spectral sequence, we are going to use mod $2$ reduction, in the spirit of the proof of propositions \ref{prop-nonsplit1} and \ref{prop-nonsplit2}. The universal coefficient theorem yields short exact sequences of strict polynomial functors (where $_2G$ denotes the $2$-torsion subgroup of an abelian group $G$):
\begin{align}
&0\to L_i\Gamma^4(A,n)\otimes\Fdeux\to L_i\Gamma^4_\Fdeux(A/2,n)\to {_2L_{i-1}\Gamma^4(A,n)}\to 0\;.\label{uct}\\
&0\to L_i\Gamma^4(A,n-2)\otimes\Fdeux\to L_i\Gamma^4_\Fdeux(A/2,n-2)\to {_2L_{i-1}\Gamma^4(A,n-2)}\to 0\;.\label{uctbis}
\end{align}
Moreover, we have already computed $L_i\Gamma^4_\Fdeux(A/2,n)$ in example \ref{ex-G4}. The following mod $2$ analogue of theorem \ref{thm-G4Z} is a straightforward consequence of example \ref{ex-G4}.
\begin{lemma}\label{lm-m2descr}
There are isomorphisms of strict polynomial functors
$$L_i\Gamma^4_\Fdeux(A/2,n)\simeq L_{i-8}\Gamma^4_\Fdeux(A/2,n-2)\oplus C_i(A,n)$$
where
$$
C_i(A,n)\simeq
\begin{cases}
0 & i>3n+1\\
\Gamma^2_\Fdeux(A/2)\otimes A/2^{(1)} & i=3n+1\\
\Gamma^2_\Fdeux(A/2)\otimes A/2^{(1)}\,\,\oplus \,\,\, A/2^{(1)}\otimes A/2^{(1)}\,\,\, & i=3n\\
A/2^{(1)}\otimes A/2^{(1)}\,\, \oplus \,\,\, A/2^{(1)}\otimes A/2^{(1)} & 2n+2<i<3n \\
A/2^{(1)}\otimes A/2^{(1)}\,\, \oplus \,\,\, \Gamma^2_\Fdeux(A/2^{(1)}) & i=2n+2 \\
A/2^{(1)}\otimes A/2^{(1)}\,\, \oplus \,\,\, A/2^{(2)} & i=2n+1\\
\Gamma^2_\Fdeux(A/2^{(1)}) \oplus A/2^{(2)} & i=2n\\
A/2^{(2)}\oplus A/2^{(2)} & n+2<i<2n\\
A/2^{(2)} & n\le i\le n+2
\end{cases}$$
\end{lemma}
Let $E_i^r$ denote the part of total degree $i$ of the $r$-th page of the spectral sequence, $E_i^r:=\bigoplus_{j+k=i}E^r_{j,k}$, and let $d^r_i:E^r_i\to E^r_{i-1}$ denote the total differential.
We distinguish three steps in the analysis of the spectral sequence.
\begin{enumerate}
\item[$\bullet$] We first analyze the spectral sequence in total degrees $i<2n$. Formulas \eqref{eq-lown}, \eqref{eq-lown+1} and \eqref{eq-prop-basdeg} show that theorem \ref{deriveddescr} holds in degrees $i<2n$.
\item[$\bullet$] Then we analyze the spectral sequence in total degrees degrees $i>2n$. Formulas \eqref{eq-triviale} and \eqref{eq-suiv} show that theorem \ref{deriveddescr} holds in degrees $i>2n$.
\item[$\bullet$] Finally, we analyze the spectral sequence in total degree $2n$. Formulas \eqref{eq-2neven} and \eqref{eq-2nodd} show that theorem \ref{deriveddescr} holds in degree $2n$.
\end{enumerate}

\subsubsection{The spectral sequence in total degree $i<2n$}
For $i=n$ or $n+1$, we have $E^1_i=E^\infty_i$ for lacunary reasons. Since there is only one nontrivial term in total degree $i$ there is no extension issue to recover the abutment. Hence we obtain:
\begin{align}&L_n\Gamma^4(A,n)=  A/2^{(2)}= A/2^{(2)}\oplus L_{n-8}\Gamma^4(A,n-2)\;,\label{eq-lown}
\\
&L_{n+1}\Gamma^4(A,n)=0=  L_{n+1-8}\Gamma^4(A,n-2)\;.\label{eq-lown+1}
\end{align}
The case $i\ge n+2$ is  slightly more involved.
\begin{proposition}\label{prop-lowdeg}
For $n+2\le i\le 2n-1$ we have
\begin{align}L_i\Gamma^4(A,n)\simeq  L_i\Gamma^4(A,n-2)\oplus A/2^{(2)}\;.\label{eq-prop-basdeg}
\end{align}
Moreover, if $n$ is even, the differentials $d^1_{2n}:E^1_{2n}\to E^1_{2n-1}$ and $d^2_{2n}:E^2_{2n}\to E^2_{2n-1}$ are zero. If $n$ is odd, one of the two  differentials $d^1_{2n}$, $d^2_{2n}$ has  $A/2^{(2)}$ as its image and the other one is the zero map.
\end{proposition}
\begin{proof}
First of all, by theorem \ref{thm-G4Z} (which we assume to be proved for $n-2$) we know the terms in the column $E^1_{-4,*}$. In particular all the expressions $E^1_i$ for $i<2n$ are direct sums of terms $A/2^{(2)}$. A subquotient of a direct sum of copies of $A/2^{(2)}$ is once again a direct sum of copies of $A/2^{(2)}$, so that $E_i^\infty$ is also a direct sum of copies of $A/2^{(2)}$ if $i<2n$. Finally the strict polynomial functor $A/2^{(2)}$ has no self-extensions of degree $1$, so that:
$$E^\infty_i\simeq L_i\Gamma^4(A,n)\quad\text{for $i<2n$}\;.$$

Since only functors of the form $A/2^{(2)}$ appear in total degrees $i<2n$ in the spectral sequence, there is no functorial issues involved in these degrees and analyzing this part of the spectral sequence amounts to analyze a spectral sequence of $\Fdeux$-vector spaces.
To be more specific, if $d_i(n)$ denotes the dimension of the $\Fdeux$-vector space $L_i\Gamma^4(\Z,n)$ for $i<2n$, then the formula \eqref{eq-prop-basdeg} is equivalent to the following equality for $n+2\le i\le 2n-1$:
\begin{align} d_i(n)=d_i(n-2)+1\;.\label{eq-interm}\end{align}

We now prove \eqref{eq-interm} by induction on $i$. We have $d_{n+1}(n)=d_{n+1}(n-2)=0$ by \eqref{eq-lown+1}, and if we denote by 
$\delta_i(n)$ the dimension of the $\Fdeux$-vector space $L_i\Gamma^4_\Fdeux(\Z/2,n)$, the exact sequences \eqref{uct} and \eqref{uctbis} and lemma \ref{lm-m2descr} yield equalities:
$$d_{n+2}(n)=\delta_{n+2}(n),\quad d_{n+2}(n-2)=\delta_{n+2}(n-2),\quad \delta_{n+2}(n)=\delta_{n+2}(n-2)+1\;.$$
This shows that \eqref{eq-interm} holds for $i=n+2$. Now assume that $n+2<i<2n$. Then
the exact sequences \eqref{uct} and \eqref{uctbis} and lemma \ref{lm-m2descr} yield equalities:
$$d_{i}(n)=\delta_{i}(n)-d_{i-1}(n),\quad d_{i}(n-2)=\delta_{i}(n-2)-d_{i-1}(n-2),\quad \delta_{n+2}(n)=\delta_{n+2}(n-2)+2\;.$$
Thus, assuming that \eqref{eq-interm} holds for $i-1$, we obtain that \eqref{eq-interm} holds for $i$.

It remains to prove the assertion on the differentials of the spectral sequence. By the formula \eqref{eq-prop-basdeg}, for $n+2\le i<2n$, $E^1_i$ has one more copy of $A/2^{(2)}$ than $E^\infty_i$. Hence there are only two possibilities for each $i$.
\begin{enumerate}
\item[$(a_i)$] The maps $d^1_i$ and $d^2_i$ are zero, one of the maps $d^1_{i+1}$, $d^2_{i+1}$ is zero, and the other one has image $A/2^{(2)}$.
\item[$(b_i)$] One of the maps $d^1_{i}$, $d^2_{i}$ is zero and the other one has image $A/2^{(2)}$, and the maps $d^1_{i+1}$ and $d^2_{i+1}$ are zero.
\end{enumerate}
We observe that $(a_i)\Longrightarrow (b_{i+1})\Longrightarrow (a_{i+2})$.
Since $E_{n+1}^1=0$, $(a_{n+2})$ holds. We can therefore deduce the result by induction on $i$.
\end{proof}

\subsubsection{The spectral sequence in total degree $i>2n$}
If $i>3n$ then $E^1_i=E^\infty_i$ for lacunary reasons, and since there is only one nontrivial term in total degree $i$ we have:
\begin{align}
L_i\Gamma(A,n)\simeq L_{i-8}\Gamma(A,n-2)\;. \label{eq-triviale}
\end{align}

To analyze the spectral sequence in total degree $i$ with $2n<i\le 3n$, we will use mod $2$ reduction. So we first recall basic facts regarding mod $p$ reduction. First of all, for any finite abelian group $G$, the $p$-torsion subgroup $_pG$ has the same dimension (as an $\Fp$-vector space) as the mod $p$ reduction $G\otimes \Fp$. Using this basic fact, one easily proves the following very rough estimation, which will be useful to compare the $E^1$ and the $E^\infty$ pages of the spectral sequence modulo $p$.
\begin{lemma}\label{lm-modp2}
Let $(G_*,\partial_*)$ be a degreewise finite differential graded abelian group with $\partial_i:G_i\to G_{i-1}$. Given a prime $p$, we denote by $(_{p}G_*,{_{p}\partial})$ the subcomplex of $p$-torsion elements of $G_*$. Then we have
\begin{align*}
\dim_{\Fp}(H_i(G)\otimes\Fp)\le \dim_\Fp( G_i\otimes\Fp) - \mathrm{rk}(_p\partial_i)\;.
\end{align*}
\end{lemma}
The next elementary lemma is useful for a modulo $p$ comparison of $L_*\Gamma^4(A,n)$ and the $E^\infty$ page of the spectral sequence.
\begin{lemma}\label{lm-modp1}
Let $G$ be a filtered finite abelian group. For all primes $p$ we have
\begin{align*}\dim_\Fp\left( G\otimes\Fp\right)\le \dim_\Fp \left(\gr(G)\otimes\Fp\right)\;.
\end{align*}
Moreover, the equality holds if and only if there exists an isomorphism of groups between the $p$-primary parts $_\p G\simeq {_\p \gr(G)}$.
\end{lemma}
\begin{proposition}\label{prop-high-deg}Let $2n<i\le 3n$. There is an isomorphism:
\begin{align}L_i\Gamma^4(A,n)\simeq  L_{i-8}\Gamma^4(A,n-2)\oplus \begin{cases}
 \Gamma^2_\Fdeux(A/2)\otimes A/2^{(1)} & i=3n\\
 A/2^{(1)}\otimes A/2^{(1)} & 2n+1<i<3n\\
 \Gamma^2_\Fdeux(A/2^{(1)}) & i=2n+1\\
\end{cases}\label{eq-suiv}\end{align}
Moreover, the differential $d^1_{-3,2n+4}:\Gamma^2_\Fdeux(A/2^{(1)})\to L_{2n-8}\Gamma^4(A,n-2)$ is zero.
\end{proposition}
\begin{proof}Lemmas \ref{lm-modp1} and \ref{lm-modp2} yield inequalities:
\begin{align}&\dim \left(E^1_i\otimes \Fdeux\right) - \mathrm{rk} ( d^1_{i}) \ge   \dim \left(E^\infty_i\otimes \Fdeux\right)\ge \dim \left(L_i\Gamma^4(A,n)\otimes \Fdeux\right),\label{eq-ineqbis}\\
&\dim \left(_2E^1_{i-1}\right) - \mathrm{rk} ( d^1_{i-1}) \ge   \dim \left(_2E^\infty_{i-1}\right)\ge \dim \left(_2L_{i-1}\Gamma^4(A,n)\right).\label{eq-ineqter}
\end{align}
We will now verify that the expressions \eqref{eq-ineqbis}, \eqref{eq-ineqter} are actually equalities for $2n+1< i\le 3n$, thereby proving that the total differential $d^1_i$ is zero in degrees $2n<i\le 3n$. The 
universal coefficient exact sequence \eqref{uct} yields an equality
\begin{align}\dim \left(L_i\Gamma^4(A,n)\otimes\Fdeux\right) +  \dim \left({_2L_{i-1}\Gamma^4(A,n)}\right)= \dim L_i\Gamma^4_\Fdeux(A/2,n)\;.\label{eq-dep1}\end{align}
The short exact sequence \eqref{uctbis} and lemma \ref{lm-m2descr} imply that
\begin{align*}
\dim L_i\Gamma^4_\Fdeux(A/2,n)=\dim \left(E^1_{-4,i+4}\otimes \Fdeux\right)+\dim \left(_2E^1_{-4,i+3}\right)+\dim C_i(A/2,n)\;.
\end{align*}
We observe that $E^1_{-3,i+3}\oplus E^1_{-3,i+2}\simeq C_i(A/2,n)$ for $2n+1<i\le 3n$. It follows that
\begin{align}
\dim L_i\Gamma^4_\Fdeux(A/2,n)=\dim \left(E^1_i\otimes \Fdeux\right)+\dim \left(_2E^1_{i-1}\right).\label{eq-dep2}
\end{align}
By comparing the sum of the inequalities \eqref{eq-ineqbis} and \eqref{eq-ineqter} with the equality provided by \eqref{eq-dep1} and \eqref{eq-dep2}, we can now conclude that the expression \eqref{eq-ineqbis} and \eqref{eq-ineqter} are actually equalities for $2n+1< i\le 3n$.

Since total differential $d^1_i$ is zero in degrees $2n<i\le 3n$ (and $d^1_{3n+1}$ is zero by lacunarity), we have $d^1_{-3,2n+4}=0$ and $E^1_i=E^\infty_i$ for $2n<i\le 3n$. Furthermore, since \eqref{eq-ineqbis} and \eqref{eq-ineqter} are equalities, lemma \ref{lm-modp1} yields a \emph{non functorial} isomorphism
\begin{align}E^\infty_{i-4,4}\oplus E^\infty_{i-3,3}= E^\infty_i\simeq L_i\Gamma^4(A,n)\label{eq-nonfunct}\;.\end{align}
To finish the proof, we have to prove a \emph{functorial} isomorphism
$E^\infty_{i-4,4}\oplus E^\infty_{i-4,3}\simeq L_i\Gamma^4(A,n)$. But the $2$-primary part of $E^\infty_{i-4,4}$ is a direct sum of functors of the following types:
$$A/2^{(2)}\;,\; \Gamma^2_\Fdeux(A/2^{(1)})\;,\; A/2^{(1)}\otimes A/2^{(1)}\;,\;\Gamma^2_\Z(A/2^{(1)})\;,\; \Lambda^2_\Fdeux(A/2^{(1)}) $$
 (as no term of the form $\Gamma^4(A)$, $\Lambda^4(A)$, $\Gamma^2_\Fdeux(A/2)\otimes A/2^{(1)}$ or $\Phi^4(A)$ occurs in the degrees which we are considering here).
In addition,  $E^\infty_{i-4,3}$ is one of the following functors:
$$\Gamma^2_\Fdeux(A/2^{(1)})\;,\; A/2^{(1)}\otimes A/2^{(1)}\;,\; \Gamma^2_\Fdeux(A/2)\otimes A/2^{(1)}\;.$$
It follows from the $\Ext^1$ computations of appendix \ref{app-comput} that there can be no nonsplit extension of $E^\infty_{i-4,3}$ by $E^\infty_{i-4,4}$, except in the case $E^\infty_{i-4,3}=\Gamma^2_\Fdeux(A/2^{(1)})$. In the latter case, the only  possible nontrivial extension is an extension of $\Gamma^2_\Fdeux(A/2^{(1)})$ by a functor of the form $(A/2^{(2)})^{\oplus \,k}$. The middle term of such a nontrivial extension is a functor $\Gamma_\Z^2(A/2^{(1)})\oplus (A/2^{(2)})^{\oplus \, k-1}$, which has $4$-torsion. Such a nontrivial extension is therefore excluded, as this would contradict the isomorphism \eqref{eq-nonfunct}. Since all possible extensions are split, we obtain an isomorphism of functors $E^\infty_i\simeq L_i\Gamma^4(A,n)$. This finishes the proof of proposition \ref{prop-high-deg}.
\end{proof}

\subsubsection{The spectral sequence in total degree $i= 2n$}
The study of the spectral sequence in total degrees $i>2n$ and $i<2n$ has already provided  us with some partial information regarding the situation for $i=2n$.
Let us sum up what we know so far.
\begin{itemize}
\item The differential $d^1_{2n+1}$ is zero, hence $E^\infty_{2n}$ is a subfunctor of $E^1_{2n}$.
\item If $n$ is even, then $d^1_{2n}$ and $d^2_{2n}$ are zero, hence $E^\infty_{2n}=E^1_{2n}$.
\item If $n$ is odd, then one of the differentials $d^1_{2n}$ and $d^2_{2n}$ is zero, and the other has image equal to $A/2^{(2)}$. Hence we have only two possibilities:
\begin{enumerate}
\item[(a)] $E^\infty_{2n} = L_{2n-8}\Gamma^4(A,n-2)\oplus \Gamma^2_\Fdeux(A/2^{(1)})\oplus A\otimes A/3^{(1)}$
\item[(b)] $E^\infty_{2n} = L_{2n-8}\Gamma^4(A,n-2)\oplus A/2^{(1)}\oplus \Lambda_\Fdeux^2(A/2^{(1)})\oplus A\otimes A/3^{(1)}$
\end{enumerate}
\end{itemize}

\begin{proposition}\label{prop-middegneven}
If $n$ is even, then
\begin{align}L_{2n}\Gamma^4(A,n)= L_{2n-8}\Gamma^4(A,n-2)\oplus \Gamma^2_\Z(A/2^{(1)})\oplus A\otimes A/3^{(1)}\;.\label{eq-2neven}\end{align}
\end{proposition}
\begin{proof}
We already know that $E_{2n}^1=E_{2n}^\infty$, so we just have to retrieve $L_{2n}\Gamma^4(A,n)$ from $E_{2n}^\infty$.
By theorem \ref{thm-G4Z} (which we assume to be proved for $n-2$), $L_{2n-8}\Gamma^4(A,n-2)$ is a sum of copies of $A/2^{(2)}$. Since $A/2^{(2)}$ has no self-extension of degree one, there is no extension problem between the columns $p=-4$ and $p=-3$. Hence the extension problem on the abutment can be restated as a short exact sequence:
\begin{align*}0\to L_{2n-8}\Gamma^4(A,n-2)\oplus A/2^{(2)}\to L_{2n}\Gamma^4(A,n) \to \Gamma^2_\Fdeux(A/2^{(1)})\oplus A\otimes A/3^{(1)}\to 0\;. \end{align*}
But the only nontrivial extension of $\Gamma^{2}_\Fdeux(A/2^{(1)})$ by a functor of the form $(A/2^{(2)})^{\oplus \, k}$  is given by a functor of the form $\Gamma^{2}_\Z(A/2^{(1)})\oplus (A/2^{(2)})^{\oplus\, k-1}$. Thus we have only two possibilities for $L_{2n}\Gamma^4(A,n)$, namely:
\begin{align*}
&(i)\quad L_{2n-8}\Gamma^4(A,n-2)\oplus \Gamma_\Z(A/2^{(1)})\oplus A\otimes A/3^{(1)}\;,\\
&(ii)\quad L_{2n-8}\Gamma^4(A,n-2)\oplus A/2^{(2)}\oplus \Gamma^2_\Fdeux(A/2^{(1)})\oplus  A\otimes A/3^{(1)}\;.
\end{align*}

The difference between $(i)$ and $(ii)$ can be seen in the dimension of the $\Fdeux$-vector space $L_{2n}\Gamma^4(A,n)\otimes\Fdeux$. So we can use mod $2$ reduction to determine which of the two possibilities is correct.
Let $d_i(n)$, resp. $\delta_i(n)$, be the dimension of $L_{i}\Gamma^4(A,n)\otimes\Fdeux$, resp. $L_{i}\Gamma^4_\Fdeux(A/2,n)\otimes\Fdeux$. By proposition \ref{prop-lowdeg} $L_{2n-1}\Gamma^4(A,n)$ is an $\Fdeux$-vector space, hence  $d_{2n-1}(n)=\dim ({_2}L_{2n-1}\Gamma^4(A,n))$. Similarly, $d_{2n-9}(n-2)=\dim ({_2}L_{2n-1}\Gamma^4(A,n))$. We have
\begin{align*}
d_{2n}(n)&= \delta_{2n}(n) - d_{2n-1}(n) && \text{ by \eqref{uct}}\\
&= \delta_{2n-8}(n-2)+\dim\Gamma^2_\Fdeux(A/2^{(1)})+\dim A/2^{(2)} - d_{2n-1}(n) && \text{ by lemma \ref{lm-m2descr}}\\
&= \delta_{2n-8}(n-2)+\dim\Gamma^2_\Fdeux(A/2^{(1)})+d_{2n-9}(n-2) && \text{ by prop. \ref{prop-lowdeg}}\\
&= d_{2n-8}(n-2)+\dim\Gamma^2_\Fdeux(A/2)  && \text{ by \eqref{uctbis}}
\end{align*}
Thus the possibility $(ii)$ is excluded for dimension reasons so that $(i)$ holds.
\end{proof}

\begin{proposition}\label{prop-middegnodd}
If $n$ is odd, then
\begin{align}L_{2n}\Gamma^4(A,n)= L_{2n-8}\Gamma^4(A,n-2)\oplus \Lambda^2_\Fdeux(A/2^{(1)})\oplus A/2^{(2)}\oplus A\otimes A/3^{(1)}\;.\label{eq-2nodd}\end{align}
\end{proposition}
\begin{proof}
Recall that we have only two possibilities for $E_{2n}^\infty$, namely
\begin{enumerate}
\item[(a)] $E^\infty_{2n} = L_{2n-8}\Gamma^4(A,n-2)\oplus \Gamma^2_\Fdeux(A/2^{(1)})\oplus A\otimes A/3^{(1)}$
\item[(b)] $E^\infty_{2n} = L_{2n-8}\Gamma^4(A,n-2)\oplus A/2^{(1)}\oplus \Lambda_\Fdeux^2(A/2^{(1)})\oplus A\otimes A/3^{(1)}$
\end{enumerate}
We now list the possibilities for $L_{2n}\Gamma^4(A,n)$. By the same reasoning as in proposition \ref{prop-middegneven}, one finds the following three possibilities, where $L_{2n-8}\Gamma^4(A,n-2)'$ denotes the expression $L_{2n-8}\Gamma^4(A,n-2)$ with one copy of $A/2^{(2)}$ deleted:
\begin{align*}
&(i)\quad L_{2n-8}\Gamma^4(A,n-2)'\oplus \Gamma_\Z(A/2^{(1)})\oplus A\otimes A/3^{(1)}\;,\\
&(ii)\quad L_{2n-8}\Gamma^4(A,n-2)\oplus \Gamma^2_\Fdeux(A/2^{(1)})\oplus  A\otimes A/3^{(1)}\;,\\
&(iii)\quad L_{2n-8}\Gamma^4(A,n-2)\oplus A/2^{(2)}\oplus \Lambda^2_\Fdeux(A/2^{(1)})\oplus  A\otimes A/3^{(1)}\;.
\end{align*}
To be more specific, $(i)$ and $(ii)$ correspond to the possible reconstructions of $L_{2n}\Gamma^4(A,n)$ if (a) holds, and $(ii)$ and $(iii)$ correspond to the possible reconstructions if (b) holds.

We can exclude $(i)$ for dimension reasons. Indeed, we can compute the dimension of $L_{2n}\Gamma^4_{\Fdeux}(A/2,n)\otimes\Fdeux$ as in the proof of proposition \ref{prop-middegneven}. We find that $\dim L_{2n}\Gamma^4(A,n)\otimes\Fdeux = \dim E^\infty_{2n}\otimes\Fdeux$. By lemma \ref{lm-modp1} this implies that $L_{2n}\Gamma^4(A,n)$ is (non functorially) isomorphic  to $E^\infty_{2n}$, hence it is a $\Fdeux$-vector space. Thus $(i)$ does not hold.

On the other hand, we can exclude $(ii)$ for functoriality reasons. Indeed, the universal coefficient theorem yields a surjective morphism of strict polynomial functors
\begin{align*}
L_{2n+1}\Gamma^4(A/2,n)\twoheadrightarrow {_2}L_{2n}\Gamma^4(A,n)\;.
\end{align*}
If $(ii)$ holds, then we can compose this surjective morphism with the projection onto the summand $\Gamma^2_\Fdeux(A/2^{(1)})$ of $L_{2n}\Gamma^4(A,n)$ to get a surjective morphism
\begin{align}
L_{2n+1}\Gamma^4(A/2,n)\twoheadrightarrow \Gamma^2_\Fdeux(A/2^{(1)})\;.\label{eq-surjection}
\end{align}
But the source of this map was computed in example \ref{ex-G4}. It is of  the form
$$L_{2n+1}\Gamma^4(A/2,n) \simeq (A/2^{(1)}\otimes A/2^{(1)})\, \oplus\, (A/2^{(2)})^{\oplus k}  $$
for some value of $k$. By the $\hom$ computations of appendix \ref{app-comput}, there is no nonzero map from $A/2^{(2)}$ to $\Gamma^2_\Fdeux(A/2^{(1)})$, and the only nonzero morphism from $A/2^{(1)}\otimes A/2^{(1)}$ to $\Gamma^2_\Fdeux(A/2^{(1)})$ is not surjective. This contradicts the existence of the surjective morphism \eqref{eq-surjection}. Thus $(ii)$ does not hold.
It follows that $(iii)$ holds.
\end{proof}


\section{A conjectural description of the functors $L_i\Gamma^d_\Z(A,n)$}
\label{conj}

In this section, we return to the study of the derived functors $L_i\Gamma^d(A,n)$ for arbitrary abelian groups $A$. We therefore drop strict polynomial structures, and consider the derived functors as genuine functors of the abelian group $A$. The combinatorics of weights however, will still be present in the picture.
We begin with a new description of the stable homology of Eilenberg-Mac Lane spaces, equivalent to the Cartan's classical one. We then use this
new parametrization of the stable homology groups in order to
formulate a conjectural functorial description of the derived functors $L_*\Gamma^d(A,n)$ for all abelian groups $A$ and positive integers $n,d$.
We finally show that the computations of the present article as well as some computations of \cite{BM} agree with the conjecture.

\subsection{A new description of the stable homology} \label{new-descr}
We begin this section by quoting Cartan's computation of the stable homology: 
$$H^{\mathrm{st}}_i(A):=\lim_n H_{n+i}(K(A,n),\Z)\simeq H_{n+i}(K(A,n),\Z)\text{ if $i<n$}\;.$$
To this purpose, we first recall Cartan's admissible words for the reader's convenience. Fix a prime number $p$. A $p$-admissible word is a non empty word $\alpha$ of finite length, formed with the letters $\phi_p$, $\gamma_p$ and $\sigma$, satisfying the following two conditions:
\begin{enumerate}
\item[(1)] the word $\alpha$ starts with the letter $\sigma$ or $\phi_p$,
\item[(2)] the number of letters $\sigma$ on the right of each letter $\gamma_p$ or $\phi_p$ in $\alpha$ is even.
\end{enumerate}
A $p$-admissible word is of first type (`de premi\`ere esp\`ece') if it ends with a $\sigma$, and of second type (`de deuxi\`eme esp\`ece') if it ends with a $\phi_p$ (words finishing on the right with the letter $\gamma_p$ will not be considered here).
There are two basic integers associated to a $p$-admissible word $\alpha$.
\begin{enumerate}
\item[$\bullet$] The degree of $\alpha$ is the integer $\deg(\alpha)$ defined recursively as follows. The degree of the empty word is zero, and 
$$\qquad\deg(\phi_p\beta)=2+p\deg(\beta)\;,\quad \deg(\sigma\beta)=1+\deg(\beta)\;,\quad \deg(\phi_p\beta)=p\deg(\beta)\;.$$ 
\item[$\bullet$] The height of $\alpha$ is the integer $h(\alpha)$ corresponding to the number of letters $\sigma$ and $\phi_p$ in $\alpha$.
\end{enumerate}
\begin{theorem}[{\cite[Exp. 11, Thm. 2]{cartan}}]\label{thm-cartan-stable}
There is a graded isomorphism, functorial with respect to the abelian group $A$:
$$H^{\mathrm{st}}_*(A)\simeq A[0]\,\oplus\, \bigoplus_{p\text{ prime}}\left(\bigoplus_{\alpha\in X_1(p)} A/p\,[\deg(\alpha)-h(\alpha)]\;\oplus\;\bigoplus_{\alpha\in X_2(p)} {_p A}[\deg(\alpha)-h(\alpha)]\right)\;, $$
where $X_i(p)$ stands for the set of $p$-admissible words of $i$-th type, starting on the left by the letters $\sigma\gamma_p$.
\end{theorem}
\begin{remark}
Cartan's proof of theorem \ref{thm-cartan-stable} is based on the integral computation of \cite[Exp. 11, Thm. 1]{cartan}. However, it can also be deduced from the mod $p$ computations of \cite[Exp. 9 and exp. 10]{cartan} (see also \cite{betley}) and the universal coefficient theorem, if one knows in advance that the stable homology only contains $p$-torsion for primes $p$. A simple proof of the latter fact is contained in \cite[Korollar 10.2]{d-p}.
\end{remark}

We now propose a compact reformulation of theorem \ref{thm-cartan-stable}.
Denote by $\EuScript A$ the set of  sequences
$\alpha=(t_1,\dots, t_m)$ (for some $m\geq 1$) such that $t_1\geq
\dots \geq t_m>0.$ For every $\alpha\in {\EuScript A}$ let $o(\alpha)$ denote the
number of distinct strictly positive integers  $t_j$ in the sequence  $\alpha=(t_1,\dots, t_m)$.
For any abelian group $A$, we define an object $\EuScript B(A)$ of the derived category of abelian groups by the following formula:
\bee
\label{def-ba}
{\EuScript S}\mathrm{t}(A):=\bigoplus_{\alpha=(t_1,\dots, t_m)\in \EuScript A}\,\,
\bigoplus_{p\ \text{prime}}A\lotimes \mathbb Z/p^{\lotimes
o(\alpha)}[2(p^{t_1}+\dots+p^{t_m}-m)]
\ee
Cartan's description of the integeral stable  homology $H^{\rm st}_\ast(A)$ of an Eilenberg-Mac Lane space $K(A,n)$ may  then be rephrased as follows:

\begin{theorem}\label{thm-new-formula}
There exists, functorially in the abelian group $A$,   a graded isomorphism:
\bee
\label{hista}
H_*^{\rm st}(A)\simeq H_*(A\oplus {\EuScript
S}\mathrm{t}(A)).
\ee
\end{theorem}

\begin{proof}
We observe that there is a bijection $\xi:X_1(p)\xrightarrow[]{\simeq} X_2(p)$ where $\xi(\alpha)$ is obtained from $\alpha$ by replacing the last two letters $\sigma^2$ of $\alpha$ by the single letter $\phi_p$. This bijection preserves the degree, and decreases the height by one, so that the $p$-primary part in the description of $H_*^{\mathrm{st}}(A)$ theorem \ref{thm-cartan-stable} can be rewritten as:
\begin{align}\bigoplus_{\alpha\in X_1(p)} \left(\; A/p\,[\deg(\alpha)-h(\alpha)]\;\oplus\;{_p A}[\deg(\alpha)-h(\alpha)+1]\;\right)\;.\label{eq-pr-1}
\end{align}

We define an equivalence relation $\mathcal{R}$ on $X_1(p)$ in the following way. For a word $\alpha\in X_1(p)$ its \emph{$\sigma^2\gamma_p$-substitution} is the word of $X_1(p)$ obtained by replacing each occurence of the letter $\phi_p$ by the group of letters $\sigma^2\gamma_p$. For example, the $\sigma^2\gamma_p$-substitution of $\sigma\gamma_p\phi_p\sigma^4\phi_p\sigma^2$ is the word $\sigma\gamma_p\sigma^2\gamma_p\sigma^6\gamma_p\sigma^2$. We say that two words of $X_1(p)$ are equivalent if they have the same $\sigma^2\gamma_p$-substitution.
Let $\alpha$ be a $p$-admissible word beginning with the letter $\sigma\gamma_p$. We say that a word  $\alpha$ is \emph{restricted} if it is composed only of the letters $\sigma$ and $\gamma_p$ (i.e. no $\phi_p$ occurs in the word $\alpha$).
The set $R(p)$ of restricted $p$-admissible words form a subset of $X_1(p)$. Each equivalent class of $X_1(p)$ contains exactly one restricted admissible word so that we can rewrite the direct sum \eqref{eq-pr-1} as:
\begin{align}\bigoplus_{\alpha\in R(p)}\, \bigoplus_{\beta\mathcal{R}\alpha} \left(\; A/p\,[\deg(\beta)-h(\beta)]\;\oplus\;{_p A}[\deg(\beta)-h(\beta)+1]\;\right)\;.\label{eq-pr-2}
\end{align}

We can replace  in \eqref{eq-pr-2} the indexing set $R(p)$ by the set   $\EuScript A$. Indeed there is a bijection $\chi:R(p)\xrightarrow[]{\simeq} \EuScript A$ defined as follows. Each restricted $p$-admissible word has the form:
\begin{align}\sigma\gamma_p\underbrace{(\sigma^2)\dots(\sigma^2)}_{\text{$k_1$ terms}}\gamma_p\underbrace{(\sigma^2)\dots(\sigma^2)}_{\text{$k_2$ terms}}\gamma_p\dots \gamma_p\underbrace{(\sigma^2)\dots(\sigma^2)}_{\text{$k_s$ terms}}\;, \label{eq-pr-3}\end{align}
where $s$ is the number of occurences of $\gamma_p$ and the $k_i$ are nonnegative. Our bijection $\chi$ is defined by sending such a restricted $p$-admissible word to the sequence of integers:
\begin{align}(\underbrace{s,\dots,s}_{\text{$k_s$ terms}}, \underbrace{s-1,\dots,s-1}_{\text{$k_{s-1}$ terms}},\dots,\underbrace{1,\dots,1}_{\text{$k_1$ terms}})\;. \label{eq-pr-4}\end{align}
If we define the degree of a sequence $(t_1,\dots,t_n)\in\EuScript A$ as the sum $2(p^{t_1}+\dots+p^{t_n})$, and its height as $2n$, then the bijection $\chi$ preserves the degree and the height. 

Now we describe the $\mathcal{R}$-equivalent classes in $X_1(p)$. Let $C$ be one such equivalent class, containing a restricted admissible word $\alpha$ of the form \eqref{eq-pr-3}. The element in $C$ are all the words which can be obtained from $\alpha$ by substituting some groups of letters $(\sigma^2)\gamma_p$ by the letter $\phi_p$. Observe that the number of groups of letters $(\sigma^2)\gamma_p$ available for substitution is equal to the number of positive $k_i$, which equals $o(\alpha)-1$ (where $o(\alpha)$ is the number of distinct positive integers in the associated sequence \eqref{eq-pr-4}). Thus $C$ contains $2^{o(\alpha)-1}$ elements. Moreover there are $\binom{o(\alpha)-1}{i}$ distinct words $\beta$ obtained from $\alpha$ by exactly $i$ substitutions, and they all satisfy  the conditions $\deg(\beta)=\deg(\alpha)$ and $h(\beta)=h(\alpha)-i$. As a consequence \eqref{eq-pr-2} can be rewritten as a direct sum:
\begin{align}\bigoplus_{\alpha\in \EuScript A}\, \bigoplus_{i=0}^{o(\alpha)-1} \left(\; A/p\,[\deg(\alpha)-h(\alpha)-i]\;\oplus\;{_p A}[\deg(\alpha)-h(\alpha)-i+1]\;\right)^{\oplus\binom{o(\alpha)-1}{i}}\;.\label{eq-pr-5}
\end{align}

Finally the graded abelian group $_pA[1]\oplus A/p[0]$ is the homology of the complex $A\lotimes \Z/p$, and  the object $\Z/p\lotimes \Z/p$ is isomorphic  to $\Z/p[0]\oplus \Z/p[1]$ in the derived category. This determines  a functorial isomorphism of graded abelian groups by induction on $n$:
\begin{align}
H_*(A\lotimes \Z/p^{\lotimes n})\simeq  \bigoplus_{i=0}^{n-1} ({_pA}[1+i]\oplus A/p[i])^{\oplus\binom{n-1}{i}}\label{eq-pr-6}
\end{align}
The statement of theorem \ref{thm-new-formula} follows 
by combining \eqref{eq-pr-5} and \eqref{eq-pr-6}.
\end{proof}

Just as one speaks of stable homology, one defines the stable derived functors \cite[(8.3)]{d-p}:
$$ L^{\rm st}_iF(A)=\lim_n L_{n+i}F(A,n)\simeq L_{n+i}F(A,n) \text{ if $i<n$}\;. $$
As explained in appendix \ref{han}, there exists a graded isomorphism, natural with respect to the abelian group $A$:
\begin{align}
H_*^{\rm st}(A)\simeq \bigoplus_{d\ge 0} L_*^{\rm st}S^d(A)\;.\label{eq-isost}
\end{align}
By d\'ecalage the (stable) derived functors of symmetric powers are isomorphic to the (stable) derived functors of the divided power functors. Hence \eqref{eq-isost} can be used to obtain a description of $L_*^{\rm st}\Gamma^d(A)$. However, to obtain such a description, we have to separate the various summands. In other words, we have to determine which summands $_pA$ and $A/p$ contribute to which stable derived functors $L_*^{\rm st}\Gamma^d(A)$. To this purpose, we define yet another integer associated to $p$-admissible words.
\begin{enumerate}
\item[$\bullet$] Let $r(\alpha)$ be the number of occurences of the letters $\phi_p$ and $\gamma_p$ in a $p$-admissible word $\alpha$. Then the weight $w(\alpha)$ is defined by $w(\alpha)=p^{r(\alpha)}$ if $\alpha$ is of first type, and $w(\alpha)=p^{r(\alpha)-1}$ if $\alpha$ is of second type.
\end{enumerate}
Then as explained e.g. in \cite[section 10.2]{antoine} the direct summand of $H_*^{\rm st}(A)$ corresponding to $L_*^{\rm st}S^d(A)$ is given by the admissible words of weight $d$. We have to translate this in the new indexation provided by theorem \ref{thm-new-formula}. We obtain the following description 
of the stable derived functors of the functor $\Gamma^d(A)$.
\begin{proposition}\label{prop-new} If the integer $d$ is not a prime power, then 
\label{lprop-new}
\[L_i^{\mathrm{st}}\Gamma^d(A)=0\ \  {\it for \ all} \ i\ge 0 .\]
For any prime number $p$ and any integer  $r>0$ the following is true:
$$L_i^{\mathrm{st}}\Gamma^{p^r}(A)=H_i\left(\bigoplus_{\alpha=(t_1,\dots, t_m)\in \EuScript A,\ t_1=r} A\lotimes \mathbb Z/p^{\lotimes
o(\alpha)}[2(p^{t_2}+\dots+p^{t_m}-m)]\right).
$$
\end{proposition}
\begin{proof}
Returning to the proof of theorem \ref{thm-new-formula}, we see that the bijection $\xi$ preserves the weights. Moreover, all the words in the same equivalent class in $X_1(p)$ also have the same weight. Thus all the terms corresponding to the homology of the summand $A\lotimes\Z/p^{\otimes o(\alpha)}$ have the same weight as $\alpha$. Now if $\alpha$ is a restricted $p$-admissible word corresponding to the sequence $\chi(\alpha)=(t_1,\dots,t_n)\in \EuScript A$ then $w(\alpha)=p^{t_1}$. Proposition \ref{lprop-new} follows.
\end{proof}

\subsection{The conjecture}
We  observed in  proposition \ref{lprop-new} that the stable derived functors of $\Gamma^d(A)$ may conveniently be repackaged by appropriately deriving those summands $A/p$ which most obviously occur in Cartan's computation, these being the summands which correspond to {\it restricted} admissible words, that is the words which do not involve the transpotence operation $\phi_p$. 
Our contention in what follows is that a similar mechanism also works unstably. 
To be more specific, let 
$\EuScript A_{\le m}\subset \EuScript A$ denote the subset containing all the sequences of length less or equal to $m$. We can reformulate the definition of ${\EuScript A}_{\le m}$ as:
\begin{align}{\EuScript A}_{\le m}=\{ (t_1,\dots,t_m)\;:\; t_1\ge t_2\dots\ge t_m\ge 0 \text{ and } t_1\ne 0\}\;.\end{align}
For every $\alpha\in {\EuScript A}_{\le m}$ we still denote by $o(\alpha)$ the
number of distinct strictly positive integers $t_j$ in the sequence  $\alpha=(t_1,\dots, t_m)$. The $\EuScript A_{\le m}$ form an increasing family of subsets which exhaust $\EuScript A$. Moreover $\EuScript A_{\le m}$ can be interpreted as an indexing set for those stable summands $A/p$ corresponding to restricted admissible  words which already appear in the unstable homology of $K(A,2m+1)$, or equivalently in the derived functors $L_*\Gamma(A,2m-1)$.

Our conjecture asserts that the derived functors $L_*\Gamma(A,n)$ are filtered, and that the associated graded pieces have the following description. If $n=2m+1$ is odd there is an isomorphism (functorial with respect to an arbitrary abelian group $A$):
\begin{align}
\gr ( L_*\Gamma(A,n))\simeq \pi_*\left(L\Lambda(A)\otimes \bigotimes_{\text{$p$ prime}}\;\bigotimes_{\alpha\in {\EuScript A}_{\le m+1}} \, L\Lambda(A\lotimes \Z/p^{\lotimes o(\alpha)})\right)\;,\label{conj-u1}
\end{align}
and if $n=2m$ there is an isomorphism (functorial with respect to an arbitrary abelian group $A$):
\begin{align}
\gr ( L_*\Gamma(A,n))\simeq \pi_*\left(L\Gamma(A)\otimes \bigotimes_{\text{$p$ prime}}\;\bigotimes_{\alpha\in {\EuScript A}_{\le m}} \, L\Gamma(A\lotimes \Z/p^{\lotimes o(\alpha)})\right)\;.\label{conj-u2}
\end{align}
The two isomorphisms above do not preserve the homological grading. We will explain below how to introduce appropriate shifts of degrees on the right hand side so as to  obtain graded isomorphisms. For the moment, we keep things simple by discussing the ungraded version of the conjecture.

Both sides of \eqref{conj-u1} and \eqref{conj-u2} are equipped with weights and our conjectural isomorphisms preserve the weights. To be more specific, the weight is defined on the left hand sides of \eqref{conj-u1} and \eqref{conj-u2} by viewing $\gr(L_*\Gamma^d(A,n))$ as the homogeneous summand of weight $d$. The weights on the right hand sides of \eqref{conj-u1} and \eqref{conj-u2} are defined as follows. We view $L_*\Lambda^d(A)$ and $L_*\Gamma^d(A)$ as functors a weight $d$, we wiew the factors $L\Lambda^d(A\lotimes \Z/p^{\lotimes o(\alpha)})$ and $L\Gamma^d(A\lotimes \Z/p^{\lotimes o(\alpha)})$ corresponding to a sequence $\alpha=(t_1,\dots)$ as functors of weight $p^{t_1}$ (compare proposition \ref{prop-new}) and weights are additive with respect to tensor products. 

To describe more concretely the homogeneous component of weight $d$ of the right hand sides of isomorphisms \eqref{conj-u1} and \eqref{conj-u2}, we introduce the following notation.
Given a prime integer $p$, a nonnegative integer $m$, a nonnegative integer $d_0$, and a family of nonnegative integers $(d_\alpha)_{\alpha\in {\EuScript A}_{\le m}}$ which is not identically zero, we denote by $d(d_0,(d_\alpha),m;p)$ the positive integer defined by
\begin{align} d(d_0,(d_\alpha),m;p):= d_0+\sum_{\alpha\in {\EuScript A}_{\le m}} d_\alpha p^{t_1(\alpha)}\;,\end{align}
where $t_1(\alpha)$ denotes the first integer in the sequence $\alpha$, that is $\alpha=(t_1(\alpha),\dots)$. The integer $d(d_0,(d_\alpha),m;p)$ is the weight of the following objects of the derived category
\begin{align}&\mathcal{E}(d_0,(d_\alpha),m;p) := L\Lambda^{d_0}(A)\,\otimes  \,\bigotimes_{\alpha\in{\EuScript A}_{\le m}}\, L\Lambda^{d_\alpha}(A\lotimes \Z/p^{\lotimes o(\alpha)})\;,
\label{conj-termE}
\\&\mathcal{D}(d_0,(d_\alpha),m;p) := L\Gamma^{d_0}(A)\,\otimes  \,\bigotimes_{\alpha\in{\EuScript A}_{\le m}}\, L\Gamma^{d_\alpha}(A\lotimes \Z/p^{\lotimes o(\alpha)})\;.\label{conj-termD}
\end{align}
With these notations, the homogeneous component of weight $d$ of the right hand side of the isomorphism \eqref{conj-u1} is given by the homotopy groups of the direct sum of $L\Lambda^d(A)$ and all the terms $\mathcal{E}(d_0,(d_\alpha),m+1;p)$ such that $d(d_0,(d_\alpha),m+1;p)=d$, for all nonnegative integers $d_0$, all families of integers $(d_\alpha)$ which are not identically zero and all prime integers $p$. The right hand side of the isomorphism \eqref{conj-u2} has an obvious similar description.

Finally, we introduce suitable shifts in order to transform the isomorphisms  \eqref{conj-u1} and \eqref{conj-u2} into graded isomorphisms. 
Given a prime integer $p$, a nonnegative integer $m$ and a nonnegative integer $d_0$, and a family of nonnegative integers $(d_\alpha)_{\alpha\in {\EuScript A}_{\le m}}$ which is not identically zero, we set:
\begin{align}
& \ell(d_0,(d_\alpha),m;p):=(2m+1)d_0+\sum_{\alpha\in {\EuScript A}_{\le m}} \ell(\alpha;p)\,d_\alpha
\end{align}
where the integer $\ell(\alpha;p)$ associated to a sequence $\alpha=(t_1,\dots,t_m)$ is given by 
\begin{align}
& \ell(\alpha;p):=\begin{cases}2p^{t_2}+...+2p^{t_m}+1 & \text{if}\ m>1\;,\\
1 & \text{if} \ m=1.
\end{cases}
\end{align}
We also set:
\begin{align}
e(d_0,(d_\alpha),m;p):=\left(\sum_{\alpha\in {\EuScript A}_{\le m}} d_\alpha\right)-d_0\;.
\end{align}
We may now state the graded version of our conjecture.
\begin{conjecture}
\label{conject}
Let $s$ be a positive integer, an let $n$ be a positive integer.
\begin{enumerate}
\item Assume that $n=2m+1$ is odd. Then there exists a filtration on $ L_s\Gamma^d(A, n)$ such that
the associated graded functor is isomorphic to 
the direct sum of the term
\begin{align*}
L_{s-nd}\Lambda^d(A)
\end{align*}
together with the terms
\begin{align*}
\pi_s\left(\;\mathcal{E}(d_0,(d_\alpha),m+1;p)\; [\ell(d_0,(d_\alpha),m+1;p)] \;\right)
\end{align*}
for all primes $p$, all nonnegative integers $d_0$ and all famillies of integers $(d_\alpha)_{\alpha\in {\EuScript A}_{\le m+1}}$ which are not identically zero, and satisfying
$d(d_0,(d_\alpha),m;p)=d$.

\item  Similarly, if $n=2m$, there exists a filtration on
$L_s\Gamma^d(A,n)$ such that the associated graded functor is isomorphic to 
  the direct sum of the term 
$$L_{s-nd}\Gamma^d(A)$$ 
together with the terms 
\begin{align*}
\pi_s\left(\;\mathcal{D}(d_0,(d_\alpha),m;p)\; [\ell(d_0,(d_\alpha),m;p)+ e(d_0,(d_\alpha),m;p)] \;\right)
\end{align*}
for all primes $p$, all nonnegative integers $d_0$ and all famillies of integers $(d_\alpha)_{\alpha\in {\EuScript A}_{\le m}}$ which are not identically zero, and satisfying
$d(d_0,(d_\alpha),m;p)=d$.
\end{enumerate}
\end{conjecture}

The remainder of the present section is devoted to proving that the conjecture holds in a certain number of cases.

\subsection{The cases $d=2$, $d=3$, for all $A$ and all $n$.}

For $s,m\geq 1,$ and $d= 2,3$, there is no filtration to consider, and conjecture \ref{conject} reduces
to the  natural
isomorphisms
\begin{equation}\label{gamma2description}
L_s\Gamma^2(A,2m+1)\simeq
\pi_s\left(\bigoplus_{i=0}^{m}\left(A\lotimes
\Z/2[2m+2i+1]\right)\oplus L\Lambda^2(A)[4m+2]\right)
\end{equation}
\begin{equation}
L_s\Gamma^2(A,2m)\simeq
\pi_s\left(\bigoplus_{i=0}^{m-1}\left(A\lotimes
\Z/2[2m+2i]\right)\oplus L\Gamma^2(A)[4m]\right)
\end{equation}
and
\begin{multline}
L_s\Gamma^3(A,2m+1)\simeq\\
\pi_s\left(\bigoplus_{i=0}^{m}\left(A\lotimes \Z/3[2m+4i+1]\oplus
A\lotimes A\lotimes \Z/2[4m+2i+2]\right)\oplus
L\Lambda^3(A)[6m+3]\right),\label{gamma3description}
\end{multline}
\begin{multline}
L_s\Gamma^3(A,2m)\simeq\\
\pi_s\left(\bigoplus_{i=0}^{m-1}\left(A\lotimes \Z/3[2m+4i]\oplus
A\lotimes A\lotimes \Z/2[4m+2i]\right)\oplus
L\Gamma^3(A)[6m]\right).\label{gamma3description-a}
\end{multline}
consistently with the results in paragraphs  4 and 5 of \cite{BM}.

\subsection{The case $d=4$,  for  $A$ free and $n$ odd.}
We now proceed to prove that the conjecture agrees with our previous computations for $d=4$, $A$ free and $n=2m+1$ with $m\ge 0$. In that case the conjecture says that (up to a filtration), $L_*\Gamma^4(A,n)$ is isomorphic to the homology groups of the following complexes \eqref{eq:1A}-\eqref{eq:1G}. 
At  the prime $p=3$,  there is a single sort of complex:
\begin{subequations}\label{eq:1}
\renewcommand\theequation{\roman{equation}}
\begin{align} A\lotimes A\lotimes \mathbb Z/3\ [4m+4k+2] &\hspace{1cm} 0\leq k \leq m. \label{eq:1A}
\intertext{At the prime $p=2$, we have the five following types of complexes:}
L\Lambda^2(A)\lotimes A\lotimes \mathbb
Z/2\ [6m+2k+3] & \hspace{1cm} 0 \leq k \leq m    \label{eq:1B}   \\
 L\Lambda^2(A\lotimes \mathbb Z/2)[4m+4k+2]&\hspace{1cm} 0 \leq k \leq  m
 \label{eq:1C}
\\
 A\lotimes \mathbb Z/2\lotimes \mathbb
Z/2\ [2m+6k+2l+1]& \hspace{1cm}0 \leq  k+l\leq m, \ \ l
\neq 0
 \label{eq:1D}
\\
A\lotimes \mathbb Z/2\ [2m+6k+1] &\hspace{1cm}
0 \leq k \leq m\label{eq:1E}  \\
 A\lotimes A\lotimes \mathbb Z/2\lotimes
\mathbb Z/2\ [4m+2k+2l+2] & \hspace{1cm} 0\leq l < k \leq m \label{eq:1F}
\intertext{together with the final term:}
L\Lambda^4(A)[8m+4]. \label{eq:1G} &
\end{align}
\end{subequations}

To verify that these expressions coincide with our  computations from section \ref{der-gamma4-sec}, we proceed by induction on $m$. 
For $m=0$, the formula for $L_*\Gamma^4(A,1)$ was described  in proposition \ref{prop-GA41}.  We  filter the term $L_3\Gamma^4(A,1) =  \Phi^4(A)$  in  \eqref{gr3ga4a11} 
and  replace it here by the direct sum $\Lambda^2(A)\otimes A/2)\oplus\Gamma^2_{\Fdeux}(A/2)$.  The result of proposition \ref{prop-GA41} then coincides with the present   $m=0$ case, once we observe that $L_\ast\Lambda^2(A/2)\simeq \Lambda^2(A/2)[0]\oplus\Gamma^2_{\Fdeux}(A/2)[1]$ (\cite{BM} \S 2.2, 
\cite{baupira}).

\bigskip

To prove that the formulas provided by the conjecture agree  for a general $n=2m+1$ with the computations of theorem \ref{thm-G4Z}, it suffices to show  that the additional summands predicted by the conjecture when passing from $L_*\Gamma^4(A,n-2)[8]$ to $L_*\Gamma^4(A,n)$ agree with those obtained in  \eqref{desg4}. Let us denote by $G(m)$ the sum of all the terms (i)-(vii). One verifies that $ G(m)=G(m-1)[8] \oplus \Delta(m) $, where $\Delta(m)$ is given by:
\begin{align*}\Delta(m)=\,  &\, A\otimes A/3[2n]\,\oplus\, L\Lambda^2(A)\otimes A/2[3n]\,\oplus\, L\Lambda^2(A/2)[2n]\,\\
&\oplus \left(\bigoplus_{\ell=1}^m A\lotimes\Z/2^{\lotimes 2}\right)\,\oplus\, A/2[n]\oplus \left(\bigoplus_{k=1}^m A\lotimes A\lotimes \Z/2^{\lotimes 2}[2n+2k]\right)\;.\end{align*}
In $\Delta(m)$ we replace  the expression $\Z/2^{\lotimes 2}$ by $\Z/2[0]\oplus\Z/2[1]$,  $L\Lambda^2(A/2)$ by $\Lambda^2(A/2)[0]\oplus\Gamma^2_{\Fdeux}(A/2)[1]$. Then $\Delta(m)$ coincides with the additional summands occuring when one passes for $n$ odd  from $L_*\Gamma^4(A,n-2)[8]$ to $L_*\Gamma^4(A,n-2)$ in theorem \ref{deriveddescr}, provided we replace the summand $\Gamma^2_\Fdeux(A/2)\otimes A/2^{(1)}[3n]$ in \eqref{desg4}  by the  direct sum $\Lambda^2(A/2)\otimes A/2[3n]\oplus A/2\otimes A/2[3n]$. This proves the conjecture for $d=4$, $A$ free and $n=2m+1$ odd.

\subsection{The case $d=4$, for  $A$ free and $n$ even.}
We now proceed to prove that the conjecture agrees with our previous computations for $d=4$ and $A$ free and $n=2m$ with $m\ge 1$.
It is straightforward to verify that the conjecture agrees for $m=1$ with the computation of proposition \ref{prop-GA42}, since we know that $L_\ast\Gamma^2(A/2)\simeq \Gamma^2_\Z(A/2)[0]\oplus \Gamma^2_\Fdeux(A/2)[1]$ (\cite{BM}, \cite{baupira} \S 4 ) (there is no filtration to consider in this situation).
 The conjecture says that (up to a filtration) $L_*\Gamma^4(A,n)$ is isomorphic to the homology groups of the following complexes \eqref{eq:2A}-\eqref{eq:2G}. 
At  the prime $p=3$,  there is a single sort of complex:
\begin{subequations}\label{eq:2}
\renewcommand\theequation{\alph{equation}}
\begin{align}
\renewcommand\theequation{\roman{equation}}
 A\lotimes A\lotimes \mathbb Z/3\ [4m+4k] & \hspace{1cm} 0 \leq k \leq  m-1 \label{eq:2A}\\
\intertext{For the prime $p=2$, we have the following complexes:}
 L\Gamma^2(A)\lotimes A\lotimes \mathbb
Z/2\ [6m+2k]& \hspace{1cm} 0 \leq k \leq m-1  \label{eq:2B}\\
L\Gamma^2(A\lotimes \mathbb Z/2)\ [4m+4k]& \hspace{1cm} 0 \leq k \leq  m-1 \label{eq:2C}
\\A\lotimes \mathbb Z/2\lotimes \mathbb
Z/2\ [2m+6k+2l] &   \hspace{1cm}  0 \leq  k+l \leq m-1,
 \  l\neq 0 \label{eq:2D}\\
A\lotimes \mathbb Z/2\ [2m+6k] &\hspace{1cm}
0 \leq k \leq  m-1 \label{eq:2E}\\
A\lotimes A\lotimes \mathbb Z/2\lotimes
\mathbb Z/2\ [4m+2k+2l] & \hspace{1cm} 0\leq l  < k \leq m-1 \label{eq:2F}  \\
\intertext{together with the  final term:}
 L\Gamma^4(A)[8m] & \hspace{1cm}\label{eq:2G}
\end{align}
\label{lm-calc3}
\end{subequations}
To prove that the formulas provided by the conjecture agree  for   $n = 2m$  with the computations of theorem \ref{thm-G4Z}, it suffices   to show  as above that the additional summands predicted by the conjecture when passing from $L_*\Gamma^4(A,n-2)[8]$ to $L_*\Gamma^4(A,n)$ agree with those obtained in  \eqref{desg41}. Let us denote by $J(m)$ the sum of all the terms (a)-(g).  One verifies that $ J(m)=J(m-1)[8] \oplus \Delta'(m)$ where 
\begin{align*}
\Delta'(m)  =\, &\, A \ot A/3[2n] \, \oplus \,  \Gamma^2 A \ot A/2 [3n] \, \oplus \, L\Gamma^2(A/2)[2n] \\
& \oplus \, \left( \bigoplus_{l=1}^{m-1} \, A \lot \Z/2 \lot \Z/2 [2n+2l] \right)  \, \oplus \, A/2 \, \oplus \, \left( \bigoplus_{k=1}^{m-1} A/2 \lot A/2 [4m+2k]\right)   .
\end{align*}
 We replace in   $\Delta'(m)$  the summand  $L\Gamma^2(A/2)$ by its value $\Gamma^2_\Z(A/2)[0] \oplus \Gamma^2_{\f2}(A/2)[1]$  (as in \cite{BM} \S 2.2, 
\cite{baupira} \S 4),   and once more  replace 
 $ \Z/2 \lot \Z/2$ by $\Z/2[0]\oplus\Z/2[1]$. Then $\Delta'(m)$ coincides with the additional summands occuring when one passes for $n$ even  from $L_*\Gamma^4(A,n-2)[8]$ to $L_*\Gamma^4(A,n-2)$ in theorem \ref{deriveddescr}.  Note that in this $n$ even  case, no summand in $\Delta'(m) $   requires any  filtering.

\subsection{The case $n=1$, $A$ free and all $d$}
In the case $n=1$ and $A$ free, the conjecture asserts that, up to a filtration, the $p$-primary part of $L_*\Gamma^d(A,1)$ is isomorphic to
\begin{align}\pi_*\left(\bigoplus_{(k_0,,k_1,\dots,k_d)}\Lambda^{k_0}(A)\lotimes L\Lambda^{k_1}(A/p)\lotimes\dots\lotimes L\Lambda^{k_d}(A/p)\;[k_0+\dots+k_d]\right)\;,\label{eq-conjn1}\end{align}
where the sum runs over all sequences of nonnegative integers $(k_0,k_1,\dots, k_d)$ of length exactly $d+1$, satisfying $\sum k_ip^{i}=d$.

On the other hand, we have explicitly computed the derived functors $L_*\Gamma(A,1)$ in section \ref{der1-gamma-z}. By theorem \ref{thm-calcul-LG-un-Z} and proposition \ref{prop-qt-SK}, the $p$-primary part of $L_i\Gamma^d(A,1)$ is concentrated in degrees $i<d$ and it is isomorphic (up to a filtration) in these degrees to the homogeneous component of degree $i$ and weight $d$ of the cycles of the tensor product of an exterior algebra with trivial differential and a family of Koszul algebras:
\begin{align}\left(\Lambda_\Fp(A/p[1]),0\right)\otimes \bigotimes_{r\ge 1} \left(\Lambda_\Fp(A/p^{(r)}[1])\otimes\Gamma_\Fp(A/p^{(r)}[2]),d_\mathrm{Kos}\right)\;. \label{eq-Kos-conj}\end{align}
We will now  reformulate this result in a different form, closer to \eqref{eq-conjn1}. For this, we consider the following modification of the Koszul algebra over $\Z$, namely the dg-$\PP_\Z$-algebra $(\Gamma(A[2])\otimes \Lambda(A[1]),pd_{\mathrm{Kos}})$ with the same underlying graded $\PP_\Z$-algebra, but whose differential is  the Koszul differential multiplied by $p$.
We denote by $C^k(A)$ the complex of functors given by its homogeneous component of weight $k$.
\begin{lemma}\label{lm-transfo-1}
For $i<d$, the $p$-primary part of the functor $L_i\Gamma^d(A,1)$ is isomorphic to the homology of the complex
$$\bigoplus_{(k_0,\dots,k_d)}\left(\Lambda^{k_0}(A)[k_0],0\right)\otimes C^{k_1}(A)\otimes \dots \otimes C^{k_d}(A) $$
where the sum runs over all sequences of nonnegative integers $(k_0,k_1,\dots, k_d)$ satisfying $\sum k_ip^{i}=d$.
\end{lemma}
\begin{remark}
The isomorphism of lemma \ref{lm-transfo-1} is really an isomorphism of functors, not an isomorphism of strict polynomial functors. For example, this isomorphism does not preserve the weights.
\end{remark}

\begin{proof}[Proof of lemma \ref{lm-transfo-1}]
{\bf Step 1.} Since we are interested in the homogeneous component of weight $d$ of \eqref{eq-Kos-conj}, we can limit ourselves to the differential graded subalgebra of \eqref{eq-Kos-conj}:
\begin{align}\left(\Lambda_\Fp(A/p[1]),0\right)\otimes \bigotimes_{1\le r\le d} \left(\Lambda_\Fp(A/p^{(r)}[1])\otimes\Gamma_\Fp(A/p^{(r)}[2]),d_\mathrm{Kos}\right)\;. \label{eq-Kos-sub}\end{align}
By forgeting the strict polynomial structure, the strict polynomial functors $A/p^{(r)}$ become isomorphic to $A/p$, and the differential graded algebra \eqref{eq-Kos-sub} is functorially isomorphic to the differential graded algebra:
\begin{align}\left(\Lambda_\Fp(A/p[1]),0\right)\otimes \bigotimes_{1\le r\le d} \left(\Lambda_\Fp(A/p[1])\otimes\Gamma_\Fp(A/p[2]),d_\mathrm{Kos}\right)\;. \label{eq-Kos-sub2}\end{align}
Moreover, under this isomorphism, the homogeneous summand of weight $d$ of \eqref{eq-Kos-sub} corresponds to the homogeneous summand of \eqref{eq-Kos-sub2} supported by the subfunctors
\begin{align}
\bigoplus  \Lambda_\Fp^{k_0}(A/p[1])\otimes \Lambda_\Fp^{a_1}(A/p)\otimes \Gamma^{b_1}_\Fp(A/p)\otimes \dots \otimes \Lambda_\Fp^{a_\ell}(A/p)\otimes \Gamma^{b_\ell}_\Fp(A/p)\;.
\label{eq-summand1}
\end{align}
where the sum runs over all sequences $(k_0,a_1,b_1,\dots,a_{\ell},b_{\ell})$ satisfying $k_0+\sum (a_i+b_i)p^i=d$.

{\bf Step 2.} We claim that the subalgebra of cycles of positive degree of functorial graded algebra \eqref{eq-Kos-sub2} is isomorphic to the homology algebra of:
\begin{align}\left(\Lambda_\Z(A[1]),0\right)\otimes \bigotimes_{1\le r\le d} \left(\Lambda_\Z(A[1])\otimes\Gamma_\Z(A[2]),pd_\mathrm{Kos}\right)\;, \label{eq-Kos-sub3}\end{align}
and we claim that the isomorphism sends the kernels of the summand \eqref{eq-summand1} of the differential graded algebra \eqref{eq-Kos-sub2} isomorphically to the homology of the following summand of the differential graded algebra \eqref{eq-Kos-sub3}
$$\bigoplus_{(k_0,\dots,k_d)}\left(\Lambda^{k_0}(A[1]),0\right)\otimes C^{k_1}(A)\otimes \dots \otimes C^{k_d}(A), $$
where the sum runs over all sequences of nonnegative integers $(k_0,k_1,\dots, k_d)$ satisfying $\sum k_ip^{i}=d$.
The statement of lemma \ref{lm-transfo-1} follows from this claim, so to finish the proof of lemma \ref{lm-transfo-1}, we only have to justify our claim.

Koszul algebras are acyclic in positive degrees by proposition \ref{prop-homology-Koszul}. Hence, the positive degree homology of the differential graded algebra \eqref{eq-Kos-sub3} is equal to the mod $p$ reduction of the algebra formed by the cycles of positive degree of
$$\left(\Lambda_\Z(A[1]),0\right)\otimes \bigotimes_{1\le r\le d} \left(\Lambda_\Z(A[1])\otimes\Gamma_\Z(A[2]),d_\mathrm{Kos}\right)\;.$$
The latter is isomorphic to the algebra formed by the cycles of positive degree of the algebra \eqref{eq-Kos-sub2}. This justifies our claim.
\end{proof}

 We are now going to check that the graded functor \eqref{eq-conjn1} can also be rewritten as the homology of the complex of lemma \ref{lm-transfo-1}. This will follow from the following lemma. Recall that $C^k(A)$ is the homogeneous part of weight $k$ of $(\Gamma(A[2])\otimes \Lambda(A[1]),pd_{\mathrm{Kos}})$, hence its desuspension $C^k(A)[-k]$ is the homogeneous part of weight $k$ of $(\Gamma(A[1])\otimes \Lambda(A[0]),pd_{\mathrm{Kos}})$.
\begin{lemma}\label{lm-smallmodel}
Let $A$ be a free abelian group, and let $k$ be a positive integer. The normalized chains of $L\Lambda^k(A/p)$ are naturally isomorphic to the complex $C^k(A)[-k]$.
\end{lemma}
\begin{proof}
We have $L\Lambda(A/p)=\Lambda(K(A\xrightarrow[]{\times p}A))$. The functor $K$ is explicit, and we readily check that for all complexes $C_1\xrightarrow[]{\partial} C_0$, the normalized chains of the simplicial object $\Lambda(K(C_1\xrightarrow[]{\partial} C_0))$ is the complex whose degree $n$ component is
$$\Lambda^{>0}(C_1)^{\otimes n}\otimes \Lambda(C_0)\;,$$
where $\Lambda^{>0}(C_i)$ stands for the augmentation ideal of the exterior algebra,
and whose differential maps an element $x_1\otimes \dots \otimes x_n\otimes y$ to the sum:
$$\sum_{i=1}^{n-1}(-1)^i x_1\otimes \dots\otimes x_i x_{i+1}\otimes\dots\otimes x_n\otimes y+(-1)^n x_1\otimes\dots\otimes x_{n-1}\otimes \Lambda(\partial)(x_n)y\;. $$

Consider the normalized chains $\NN\Lambda(A/p)$ of $L\Lambda(A/p)$ as a differential graded $\PP_\Z$-algebra. Its homogeneous component of weight $k$ yields the normalized chains $\NN\Lambda^k(A/p)$ of $L\Lambda^k(A/p)$. These normalized chains contain $C^k(A)[-k]$ as a subcomplex. Indeed the inclusion is simply given by seeing the object of degree $i$ of $C^k(A)[-k]$, that is the functor $\Gamma^i(A)\otimes \Lambda^{k-i}(A)$, as a subfunctor of $(\Lambda^1(A))^{\otimes i}\otimes \Lambda^{k-i}(A)$ by the canonical inclusion of invariants into the tensor product.
One readily checks that the differential of $C^k(A)[-k]$ coincides with the restriction of the differential of $\NN\Lambda^k(A/p)$.

To finish the proof, it remains to show that the inclusion of complexes
\begin{align}C^k(A)[-k]\hookrightarrow \NN\Lambda^k(A/p)\label{eq-qios}\end{align}
is a quasi-isomorphism. For this, we filter both complexes by the weight of the exterior power on the right, that is the term
$F_s(C^k(A)[-k])$ of the filtration is the subcomplex of $C_k(A)$ supported by the $\Gamma^i(A)\otimes \Lambda^{k-i}(A)$ for $k-i\ge s$, and the term $F_s(\NN\Lambda^k(A/p))$ is the subcomplex of $\NN\Lambda^k(A/p)$ supported by the $\Lambda^{i_1}(A)\otimes\dots\otimes \Lambda^{i_n}(A)\otimes\Lambda^{k-\sum i_j}(A)$ for $k-\sum i_j\ge s$. Both filtrations have finite length, and the inclusion of complexes preserves the filtrations, whence a morphism:
\begin{align}\gr  (C^k(A)[-k])\to \gr (\NN\Lambda^k(A/p))\;.\label{eq-qiosgr}\end{align}
Since the filtrations are finite, the fact that \eqref{eq-qios} is a quasi-isomorphism will follow from the fact that \eqref{eq-qiosgr} is. But we readily check that $\gr  (C^k(A)[-k])$ is equal to  the complex $(\bigoplus_{i=0}^k\Gamma^i(A)\otimes\Lambda^k(A),0)$ (with zero differential) that $ \gr (\NN\Lambda^k(A/p))$ is equal to  the complex
$\bigoplus_{i=0}^k \NN\Lambda^i(A)\otimes  (\Lambda^k(A),0)$, and that the morphism \eqref{eq-qiosgr} is simply  the morphism constructed from the quasi-isomorphisms $\Gamma^i(A)[i]\hookrightarrow \NN\Lambda^i(A)$. This proves that \eqref{eq-qiosgr} is a quasi-isomorphism, hence that \eqref{eq-qios} is a quasi-isomorphism.
\end{proof}
The results of lemma \ref{lm-transfo-1} and \ref{lm-smallmodel} together prove the following proposition.
\begin{proposition}
The conjecture holds for $n=1$ and $A$  a free abelian group.
\end{proposition}

\appendix

\section{Some computations of $\hom$ and $\Ext^1$ in functor categories.}
\label{app-comput}

In this appendix, we review some elementary computations of $\hom$ and $\Ext^1$ in functor categories. The functor categories which we are considering are the following.
\begin{enumerate}
\item The category $\FF_\Z$ of functors from free abelian groups of finite type to abelian groups.
\item The category $\PP_\k$ of strict polynomial functors defined over a field $\k$.
\item The category $\PP_\Z$ of strict polynomial functors defined over $\Z$.
\end{enumerate}
These categories are related by exact faithful functors
$$\PP_\Fp\to \PP_\Z\;,\qquad  \PP_\Z\to \FF_\Z\;.$$
To be more specific, the first one sends a functor $F(V)\in\PP_\Fp$ to $F(A/p)$ considered as a functor of the variable $A$. The second one is the forgetful functor from strict polynomial functors to ordinary functors.
Since they are  exact, these functors induce morphisms on the level of $\Ext$-groups. In particular, the computation of a  $\hom$ or $\Ext^1$ in any of these three categories gives us partial information regarding the $\hom$ or the $\Ext^1$ in the  other two. We will be more precise about this below.

\subsection{General techniques for $\Ext^*$-computations}\label{sec-generaltechniques}
We first recall some techniques from  standard homological algebra  which are efficient when computing   $\Ext$-groups in the functor categories that we are considering. We refer the reader to the articles \cite{FS,FFSS,antoine,piraPS}  for more details and further techniques for computing $\Ext$-groups.

First of all, by the Yoneda lemma, the functors $P^B(A)=\Z\hom_\Z(B,A)$ with $A$ as the variable and a free finitely generated abelian group $B$ as a parameter are projective generators in $\FF_\Z$, and we have for any functor $F(A)$, 
\begin{align}
\Ext^i_{\FF_\Z}(P^B(A),F(A))= F(B) \quad\text{ if $i=0$, and $0$ if $i>0$.}\label{eq-YonF}
\end{align}
The functors $\Gamma^{d,N}_R(M)=\Gamma^d_R(\hom_R(N,M))$ where $N$, $M$ are finitely generated free $R$-modules and $d$ is a nonnegative integer provide a family of projective generators of the categories $\PP_R$ of strict polynomial functors over a commutative ring $R$. For any strict polynomial functor $F(M)$ in $\PP_{R}$ we have a formula analoguous to \eqref{eq-YonF}, where $F^d(M)$ denotes the homogeneous component of weight $d$ of $F(M)$:
\begin{align}
\Ext^i_{\PP_R}(\Gamma^{d,N}_R(M),F(M))= F^d(N) \quad \text{ if $i=0$, and $0$ if $i>0$.}\label{eq-YonP}
\end{align}

Compared to the category $\FF_\Z$, the categories of strict polynomial functors have the pleasant additional feature that functors are equipped with weights and that for any pair of homogeneous strict polynomial functors $F(M)$, $G(M)$ of distinct weights we have
 $$\Ext^*_{\PP_R}(F(M),G(M))=0.$$

Among the classical techniques for computing $\Ext$-groups in our categories, one of the most important ones  is the sum-diagonal adjunction which we will now recall.  For all functors $F(A)$, $G(A)$, $H(A)$ in $\FF_\Z$, there are isomorphisms
\begin{align}
&\Ext^*_{\FF_\Z}(H(A),F(A)\otimes G(A))\simeq  \Ext^*_{\FF_\Z(2)}(H(A\oplus B),F(A)\otimes G(B))\;,\\\label{eq-iso-sum-diag}
&\Ext^*_{\FF_\Z}(F(A)\otimes G(A),H(A))\simeq  \Ext^*_{\FF_\Z(2)}(F(A)\otimes G(B),H(A\oplus B))\;,
\end{align}
where the terms on the right hand side denote $\Ext$-groups computed in the category $\FF_\Z(2)$ of \emph{bifunctors} (whose variables are denoted here by $A$ and $B$). In many cases (for example if $H(A)=\Gamma^2_\Z(A)$) one can express $H(A\oplus B)$ as a direct sum of functors of the form $H_1(A)\otimes H_2(B)$. Such bifunctors are sometimes called \emph{of separable type}. Thus \eqref{eq-iso-sum-diag} leads us to consider extension groups of the form
$$\Ext^*_{\FF_\Z(2)}(H_1(A)\otimes H_2(B),F(A)\otimes G(B)).$$
There are several situations in which  we can actually compute such  $\Ext$-groups. For example these extension groups vanish  if $H_2(B)$ is a constant functor and $G(0)=0$.
Together with the sum-diagonal adjunction, this yields the following  fundamental vanishing lemma.
\begin{lemma}[\cite{Pira}]\label{lm-pira}
Let $H(A)$ be an additive functor, and let $F(A),G(A)$ be a pair of functors satisfying $F(0)=G(0)=0$. Then we have
$$\Ext^*_{\FF_\Z}(H(A),F(A)\otimes G(A))=0=\Ext^*_{\FF_\Z}(F(A)\otimes G(A),H(A))\;. $$
\end{lemma}
The following proposition gives another example for which  we can compute $\Ext$-groups between functors of separable type.

\begin{proposition}[K\"unneth formulas]\label{prop-Kunneth}
Let $F(A),G(A)\in\FF_\Z$ be a pair of  functors with values in $\Fp$-vector spaces, and let $H_1(A),H_2(A)\in\FF_\Z$. Assume that $H_1(A)$ or $H_2(A)$ has values in free abelian groups. Then there is an isomorphism:
\begin{align*}\Ext^*_{\FF_\Z}(H_1,F)\otimes \Ext^*_{\FF_\Z}(H_2,G)\simeq \Ext^*_{\FF_\Z(2)}(H_1\boxtimes H_2,F\boxtimes G)
\end{align*}
To be concise, we have dropped the variables in the previous formula, and denoted by `$F\boxtimes G$' the bifunctor $(F\boxtimes G)(A,B)=F(A)\otimes G(B)$.
Assume instead that $H_1(A)$ and $H_2(A)$ both take values in $\Fp$-vector spaces, and denote by $E^*$ the graded $\Fp$-vector space:
$$E^*:=\Ext^*_{\FF_\Z}(H_1,F)\otimes \Ext^*_{\FF_\Z}(H_2,G)\;.$$
Then there is an isomorphism
$\Ext^0_{\FF_\Z(2)}(H_1\boxtimes H_2,F\boxtimes G)\simeq E^0$,
and a long exact sequence of $\Fp$-vector spaces:
\begin{align*}
0\to &  \Ext^1_{\FF_\Z(2)}(H_1\boxtimes H_2,F\boxtimes G)\to E^1\to \Ext^0_{\FF_\Z(2)}(H_1\boxtimes H_2,F\boxtimes G)\\&
\xrightarrow[]{\partial} \Ext^2_{\FF_\Z(2)}(H_1\boxtimes H_2,F\boxtimes G)\to E^2\to \Ext^1_{\FF_\Z(2)}(H_1\boxtimes H_2,F\boxtimes G)\xrightarrow[]{\partial} \dots
\end{align*}
\end{proposition}
\begin{proof}
In this proof we use of the concise notation $[H_1(A),F(A)]$ for $\hom$-groups between $H_1(A)$ and $F(A)$.
Let $P^X(A)$ and $P^Y(A)$ be standard projectives of $\FF_\Z$. The bifunctor $P^X(A)\otimes P^Y(B)=\Z(\hom_\Z(X,A)\times \hom_\Z(Y,B))$ is a projective object of $\FF_\Z(2)$ and the Yoneda lemma yields an isomorphism:
$$[P^X(A)\otimes P^Y(B),F(A)\otimes G(B)]\simeq F(X)\otimes G(Y)\;.$$
Moreover, the canonical morphism
\begin{align}[P^X(A),F(A)]\otimes[P^Y(A),G(A)]\to [P^X(A)\otimes P^Y(B),F(A)\otimes G(B)] \label{eq-is}
\end{align}
can be identified via the Yoneda lemma with the identity morphism of $F(X)\otimes G(Y)$. In particular, the map  \eqref{eq-is} is an isomorphism.
Let $P_1(A)_\bullet$, resp. $P_2(A)_\bullet$ be a projective resolution of $H_1(A)$, resp. $H_2(A)$. By \eqref{eq-is}, there is an isomorphism of complexes
\begin{align}[P_1(A)_\bullet,F(A)]\otimes [P_2(B)_\bullet,G(B)]\simeq [P_1(A)_\bullet\otimes P_2(B)_\bullet,F(A)\otimes G(B)]\;.\label{eq-is2}\end{align}
The left hand side of \eqref{eq-is2} is a tensor product of two complexes of $\Fp$-vector spaces. Hence by the K\"unneth formula, its homology is isomorphic to
$$\Ext^*_{\FF_\Z}(H_1(A),F(A))\otimes \Ext^*_{\FF_\Z}(H_2(A),G(A))\;.$$

Assume that the values of one of the functors $H_1(A),H_2(A)$ are free abelian groups. Then the complex $P_1(A)_\bullet\otimes P_2(B)_\bullet$ is a projective resolution of $H_1(A)\otimes H_2(B)$ hence the  homology of the right hand side of \eqref{eq-is2} is isomorphic to $\Ext^*_{\FF_\Z(2)}(H_1(A)\otimes H_2(B),F(A)\otimes G(B))$. We thus  obtain the first assertion of proposition \ref{prop-Kunneth}.

Assume instead that both functors $H_1(A),H_2(A)$ have values in $\Fp$-vector spaces. Then $P_1(A)_\bullet\otimes P_2(B)_\bullet$ is complex of projectives, whose homology is  equal to $H_1(A)\otimes H_2(B)$ in degrees $0$ and $1$, and is trivial elsewhere. It follows that there is a hypercohomology spectral sequence \cite[5.7.9]{weibel} with second page
$$E^{p,q}_2=\Ext^p_{\FF_\Z(2)}(H_1(A)\otimes H_2(B),F(A)\otimes G(B))\quad\text{if $q=1,2$, and $0$ if $q\ne 1,2$}
$$
and  differential $d_2:E^{p,q}_2\to E^{p+2,q-1}_2$, and  which converges to the homology of the right-hand side of \eqref{eq-is2}. This implies the second statement in proposition \ref{prop-Kunneth}.
\end{proof}

The sum-diagonal adjunction works exactly in the same way  for strict polynomial functors over any commutative ring $R$. The isomorphism \eqref{eq-iso-sum-diag} remains valid  when `$\FF_\Z$' is replaced by `$\PP_R$'. For extension groups between strict polynomial bifunctors of separable type, the K\"unneth formulas of proposition \ref{prop-Kunneth} remain valid if `$\FF_\Z$' is replaced by `$\PP_\Z$'. The vanishing lemma \ref{lm-pira} also holds. If we are interested in strict polynomial functors defined over a field $\k$ the situation is even nicer: in that case,  we always have an isomorphism
\begin{align*}
\Ext^*_{\PP_\k}(H_1,F)\otimes \Ext^*_{\PP_\k}(H_2,G)\simeq \Ext^*_{\PP_\k(2)}(H_1\boxtimes H_2,F\boxtimes G)\;.
\end{align*}

\subsection{Some $\hom$ computations}
\label{comput-hom}
The following elementary lemma allows us to compare  $\hom$s between the functor categories under consideration.

\begin{lemma}\label{lm-hom1}
The functor $\PP_\Fp\to \PP_\Z$ defined by evaluation on $A/p$ is full and faithful. The forgetful functor $\PP_\Z\to \FF_\Z$ is faithful.
\end{lemma}
\begin{proof}
It is clear that both functors are faithful. The only thing that we have to prove is the following isomorphism, for all $F(V),G(V)\in\PP_\Fp$:
\begin{align}\hom_{\PP_\Fp}(F(V),G(V))\simeq \hom_{\PP_\Z}(F(A/p),G(A/p))\;.\label{eq-cequonveut}\end{align}
By left exactness of $\hom$, the proof reduces to the case where $F(V)$ is a standard projective, i.e. $F(V)=\Gamma^{d,U}_\Fp(V)$, with $U=\mathbb{F}_p^n$. Now observe that $\Gamma^{d,U}_\Fp(A/p)=\Gamma^{d,B}_\Z(A)\otimes\Fp$, with $B=\Z^n$. Thus we have a commutative diagram
$$\xymatrix{
\hom_{\PP_\Fp}(\Gamma^{d,U}_\Fp(V),G(V))\ar[d]^-{\simeq}_-{(1)}\ar[rr]&&\hom_{\PP_\Z}(\Gamma^{d,U}_\Fp(A/p),G(A/p))\ar[d]^{\simeq}\\
G^d(U)\ar[r]^{\simeq}\ar[r]^{\simeq} & G^d(B/p) & \hom_{\PP_\Z}(\Gamma^{d,B}_\Z(A/p),G(A/p))\ar[l]_-\simeq^-{(2)}
}\;,$$
where $G^d(V)$ is the homogeneous component of weight $d$ of $G(V)$,  the maps $(1)$ and $(2)$ are provided by the Yoneda lemma, and the vertical map on the right is induced by the canonical projection $\Gamma^{d,B}_\Z(A)\to \Gamma^{d,U}_\Fp(A/p)$. Hence the isomorphism \eqref{eq-cequonveut} holds.
\end{proof}

The forgetful functor $\PP_\Z\to \FF_\Z$ is not full in general. For example, there is no non-zero morphism of strict polynomial functors from $A/p^{(1)}$ to $A/p$, because these strict polynomial functors are homogeneous of different weights. However the forgetful functor sends both of these to the same ordinary functor $A/p$, and the identity morphism is a non-zero morphism $A/p\to A/p$. The following lemma is proved in a similar manner  as  lemma \ref{lm-hom1} (reduce to the case where $F(V)$ is a standard projective and then use the Yoneda lemma).
\begin{lemma}
Let $\k$ be a field. For all strict polynomial functors $F(V),G(V)$, precomposition by the Frobenius twist $V \mapsto V^{(1)}$ induces an isomorphism
\begin{align*}
\hom_{\PP_\k}(F(V),G(V))\xrightarrow[]{\simeq}\hom_{\PP_\k}(F(V^{(1)}),G(V^{(1)}))\;.
\end{align*}
\end{lemma}
Let $\k$ be a field. Table \ref{homcomput} gathers some elementary $\hom$ computations in $\PP_\k$ which follow from the left-exactness of $\hom$ and the techniques recalled in section \ref{sec-generaltechniques}. By lemma \ref{lm-hom1}, these also provide $\hom$ computations in $\PP_\Z$. One can also verify by using the techniques recalled in section \ref{sec-generaltechniques} that the computations of $\hom_{\FF_\Z}(F(A/p),G(A/p))$ for the  functors $F(V)$ and $G(V)$ listed  in the following  table give the same result as in $ \PP_\Z$.

\begin{table}[ht]
\renewcommand{\arraystretch}{1.8}
\begin{tabular}{|c||c|c|c|c|c|}
\hline
$\hom_{\PP_\Fdeux}(F(V),G(V))$ & $F(V)=\Gamma_{\k}^{p^r}(V)$ & $V^{\otimes p^r}$ &$V^{(r)}$ & $S^{p^r}(V)$ & $\Lambda^{p^r}(V)$  \\
\hline
\hline
 $G(V)=\Gamma_{\k}^{p^r}(V)$ & $\k$ & $\k$ &  $0$ & $\k$ & {$\begin{cases} 0 \text{ if $p\ne 2$}\\ \k \text{ if $p=2$}\end{cases}$}\\
\hline
 $V^{\otimes p^r}$ &  $\k$ &  $\k^{|\Sigma_{p^r}|}$ & $0$  & $\k$ & $\k$\\
\hline
 $V^{(r)}$ &  $\k$ &  $0$ &  $\k$ & $0$ & $0$ \\
\hline
 $S^{p^r}(V)$ &  $\k$ &  $\k$ &  $\k$ & $\k$ & $0$ \\
\hline
 $\Lambda^{p^r}(V)$ &  $0$ &  $\k$ & $0$ & {$\begin{cases} 0 \text{ if $p\ne 2$}\\ \k \text{ if $p=2$}\end{cases}$} & $\k$\\
\hline
\end{tabular}
\vspace{.5cm}
\caption{Some values  of $\hom_{\PP_\k}(F(V),G(V))$ for functors $F$ and $G$  of weight $p^r$ with $r>0$, where $\k$ is  a field of prime characteristic $p$.}
\label{homcomput}
\end{table}

\subsection{Some $\Ext^1$ computations}

\subsubsection{Computations in $\PP_\k$, with $\k$ a field of positive characteristic.}
The results given here are all well-known as special cases of more general statements, see e.g. \cite{antoine,antoineENS}, but we give here some self-contained and elementary proofs.

\begin{lemma}\label{lm-precomp-ext1}
Let $\k$ be a field of positive characteristic and  consider $F(V),G(V)\in\PP_\k$ with finite-dimensional values. Precomposition by the Frobenius twist induces an isomorphism:
\begin{align*}
\Ext^1_{\PP_\k}(F(V),G(V))\xrightarrow[]{\simeq}\Ext^1_{\PP_\k}(F(V^{(1)}),G(V^{(1)}))\;.
\end{align*}
\end{lemma}
\begin{proof}
There is a short exact sequence $0\to K_F(V)\to P_F(V)\to F(V)\to 0$ where $P_F(V)$ is a finite direct sum of standard projectives (i.e. of functors of the form $\Gamma^{d,U}_\k(V)$). Similarly (applying the  duality $^\sharp$ of section \ref{duality} to the one-step projective resolution of $G(V)^\sharp$), there is a short exact sequence $0\to G(V)\to J_G(V)\to Q_G(V)\to 0$ where $J_G(V)$ is a finite direct sum of functors of the form $S^d_{\k,U'}(V)=S^d_\k(U'\otimes V)$. By considering  the long exact sequences of $\mathrm{Ext}$'s associated to these short exact sequences, we can reduce the proof to showing that $\Ext^{1}_{\PP_\k}(\Gamma^{d,U}_\k(V^{(1)}),S^d_{\k,U'}(V^{(1)}))$ is zero. But we can decompose $\Gamma^{d,U}_\k(V^{(1)})$ as a direct sum of functors of the form $\Gamma^{d_1}_\k(V^{(1)})\otimes\dots\otimes\Gamma^{d_k}_\k(V^{(1)})$, with $\sum d_j=d$, and do the same for $S^d_{\k,U'}(V^{(1)})$. By the sum-diagonal adjunction and the K\"unneth formula, the vanishing of $\Ext^{1}_{\PP_\k}(\Gamma^{d,U}_\k(V^{(1)}),S^d_{\k,U'}(V^{(1)}))$ will then  follow from the vanishing of $\Ext^{1}_{\PP_\k}(\Gamma^{e}_\k(V^{(1)}),S^e_\k(V^{(1)}))$ for all $e\le d$, which we will now prove. For $e=p^r$ with $r\ge 0$,  consider the presentation of  $ \Gamma^{p^{r}}_\k(V^{(1)})$ as in lemma  \ref{prop-cokernel}
and let $K(V)$ be the kernel of the corresponding  canonical projection $\Gamma^{p^{r+1}}_\k(V)\xrightarrow[]{\pi} \Gamma^{p^{r}}_\k(V^{(1)}) $. We have a dimension shifting isomorphism
$$\Ext^1_{\PP_\k}(V^{(1)},V^{(1)})\simeq \hom_{\PP_\k}(K(V),\Gamma^{p^{r}}_\k(V^{(1)}))\;.$$
whose right-hand term  embeds into the group  $\hom_{\PP_\k}(\Gamma^{p^{r+1}-1}_\k(V)\otimes V,S^{p^r}_\k(V^{(1)}))$, which is trivial  by sum-diagonal adjunction. Now let $e= e_kp^k+\dots+ e_0$ be the $p$-adic decomposition of $e$. Then the canonical inclusion
$$\Gamma^e_\k(V^{(1)})\hookrightarrow \bigotimes_{i=0}^k\left(\Gamma_\k^{p^i}(V^{(1)})\right)^{\otimes e_i}=: H(V^{(1)}) $$
admits a rectract, provided by the multiplication in the divided power algebra. In particular $\Ext^{1}_{\PP_\k}(\Gamma^{e}_\k(V^{(1)}),S^e_\k(V^{(1)}))$ is a direct summand of $\Ext^{1}_{\PP_\k}(H(V^{(1)}),S^e_\k(V^{(1)}))$. The latter  group is  zero by the sum-diagonal adjunction.
\end{proof}

\begin{proposition}\label{prop-vanish}
Let $\k$ be a field of characteristic $2$, and let $V$ be a finite dimensional
vector space. Any extension of degree one  in $\PP_\k$
between two functors of the form $\bigotimes_{i=1}^n\Gamma^{d_i}_\k(V^{(r_i)})$ is trivial.
\end{proposition}
\begin{proof}
 By iterated use of
the sum-diagonal adjunction and the K\"unneth formula, the proof reduces to showing that 
\begin{align}\Ext^1_{\PP_\k}(F(V),G(V))= 0\label{eqn-aprouver}\end{align}
for $F(V)=\Gamma^{d}_\k(V^{(r)})$ and $G(V)=\Gamma^e_\k(V^{(s)})$. Since there are no non-trivial extensions  between functors with different weights, we can assume
that $d2^r=e2^s$. By lemma \ref{lm-precomp-ext1}, precomposition by the Frobenius twist induces an isomorphism on the level of $\Ext^1$, so the proof reduces to showing \eqref{eqn-aprouver} when one of the integers $r,s$ is equal to zero.
If $r=0$, then \eqref{eqn-aprouver} holds by projectivity of $F(V)$. Hence it suffices to prove \eqref{eqn-aprouver} for  $F(V)=\Gamma^{d}_\k(V^{(r)})$ and  $G(V)=\Gamma^{dp^r}_\k(V)$ with $r>0$.

By sum-diagonal adjunction and the K\"unneth formula $\Ext^*_{\PP_\k}(F(V),V^{\otimes dp^r})=0$. Hence
\begin{align}\Ext^1_{\PP_\k}(F(V),G(V))\simeq \hom_{\PP_\k}(F(V),C(V))\;.\label{eq-shift}\end{align}
where  $C(V)$ is the cokernel of the canonical inclusion $\Gamma^{dp^r}_\k(V)\to V^{\otimes dp^r}$. Since $C(V)$ embeds into $\Lambda^2_\k(V)\otimes V^{\otimes p^r-2}$ and $\Gamma^{dp^r}_\k(V)$ surjects onto $F(V)$, the right hand side of \eqref{eq-shift} embeds into
$\hom_{\PP_\k}(\Gamma^{dp^r}_\k(V),\Lambda^2_\k(V)\otimes V^{\otimes p^r-2})$. The latter  group is trivial by \eqref{eq-YonP}.
\end{proof}

There is a more general statement than proposition \ref{prop-vanish} over fields of odd characteristic, whose proof is completely similar.
\begin{proposition}\label{prop-vanish-odd}
Let $\k$ be a field of odd characteristic $p$, and let $V$ be a finite
dimensional $\k$ vector space. Any degree one extension in  $\PP_\k$
between functors of the form $\bigotimes_{i=1}^n\Gamma^{d_i}(V^{(r_i)})\otimes
\bigotimes_{j=1}^m\Lambda^{e_j}(V^{(s_j)})$ is trivial.
\end{proposition}
\begin{remark}
Proposition  \ref{prop-vanish-odd} does not hold when $\k$ is a field of characteristic
$p=2$. For example $\Ext^1_{\PP_\k}(V^{(1)},\Lambda^2(V))$ is one-dimensional, generated by the extension $$0\to \Lambda^2(V)\to
\Gamma^2(V)\to V^{(1)}\to 0\;,$$
where the map $\Lambda^2(V)\to \Gamma^2(V)$  sends  $x\wedge y$ to $
x\cdot y$ (and  is only  well defined if $p=2$). This
non-trivial extension shows up in many computations and explains why  our results  take on   different forms, depending  the parity of the characteristic (or of the torsion over $\Z$).
\end{remark}

\subsubsection{Computations in $\PP_\Z$.} The following lemma is a formal consequence of the fact that the functor $\PP_\Fp\to \PP_\Z$, $F(V)\mapsto F(A/p)$ is full and faithful.
\begin{lemma}
For all $F(V),G(V) \in \PP_\Fp$, evaluation on $A/p$ yields an injective map:
$$\Ext^1_{\PP_\Fp}(F(V),G(V))\hookrightarrow  \Ext^1_{\PP_\Z}(F(A/p),G(A/p))\;.$$
\end{lemma}
We cannot expect that evaluation on $A/2$ induces an isomorphism in general. For example, by taking the $\Ext$-long exact sequence associated to the short exact sequence
$0\to \Gamma_\Z^{d}(A)\to \Gamma_\Z^{d}(A) \to  \Gamma_\Fp^d(A/p)\to 0 $
we obtain that for all $G(V)\in\PP_{\Fp,d}$:
\begin{align*}
\Ext^i_{\PP_\Z}(\Gamma^d_\Fp(A/p),G(A/p))=
G(A/p)\quad\text{if $i=0,1$\;, and $0$ if $i>2$\;.}
\end{align*}
On the other hand,   by projectivity of $\Gamma^d_\Fp(V)$ there is no non-zero extension of $\Gamma^d_\Fp(V)$ by $G(V)$  in $\PP_\Fp$ .
We now describe some  explicit elementary computations in $\PP_\Z$.

\begin{lemma}\label{lm-ApAp}
Let $p$ be a prime number. Then  $\Ext^1_{\PP_\Z}(A/p^{(r)}, A/p^{(r)})=0$  for all positive integers $r$.
\end{lemma}
\begin{proof}
There is a short exact sequence
$$\bigoplus_{k+\ell=p^r,\,k,\ell >0}\Gamma_\Z^k(A)\otimes\Gamma_\Z^\ell(A)\to \Gamma^{p^r}_\Z(A)\to A/p^{(r)}\to 0\;,$$
 where  the left-hand map is induced by the multiplication.
Let $K(A)$ be the kernel of the map $\Gamma^{p^r}_\Z(A)\to A/p^{(r)}$. Applying the functor $\Ext_{\PP_\Z}^*(-,A/p^{(r)})$ to the short exact sequence $0\to K(A)\to \Gamma^{p^r}_\Z(A)\to A/p^{(r)}\to 0$   yields an isomorphism
\begin{align}\hom_{\PP_\Z}(K(A),A/p^{(r)})\simeq \Ext^1_{\PP_\Z}(A/p^{(r)},A/p^{(r)})\;.\label{eq-isointerm}\end{align}
Since $\bigoplus_{k+\ell=p^r,k,\ell >0}\Gamma_\Z^k(A)\otimes\Gamma_\Z^\ell(A)$ surjects onto $K(A)$, it follows that  the left-hand side of \eqref{eq-isointerm} embeds into a direct sum of terms $\hom_{\PP_\Z}(\Gamma_\Z^k(A)\otimes\Gamma^\ell_\Z(A),A/p^{(r)})$. But these expressions  are trivial by lemma \ref{lm-pira} whence the result.
\end{proof}
We now compute the extensions of $\Gamma^2_\Fdeux(A/2^{(1)})$ by $A/2^{(2)}$. The cokernel of the map $\Gamma^3_\Z(A)\otimes A\to \Gamma^4_\Z(A)$ induced by the multiplication of the divided powers algebra is a homogeneous strict polynomial functor of weight $4$, whose image under  the   forgetful functor $\PP_\Z\to\FF_\Z$ is the   ordinary functor $\Gamma_\Z^2(A/2)$.  We  therefore will denote this cokernel  by $\Gamma^2_\Z(A/2^{(1)})$. The following computation is used in proposition \ref{prop-nonsplit2}.
\begin{lemma}\label{lm-calc2}
There is an isomorphism  $\Ext^1_{\PP_\Z}(\Gamma^2_\Fdeux(A/2^{(1)}),A/2^{(2)})\simeq \Z/2$. The non-split extension is:
$$0\to A/2^{(2)}\to \Gamma^2_\Z(A/2^{(1)})\to \Gamma^2_\Fdeux(A/2^{(1)})\to 0\;,$$
where the maps $A/2^{(2)}\to \Gamma^2_\Z(A/2^{(1)})$ and $\Gamma^2_\Z(A/2^{(1)})\to \Gamma^2_\Fdeux(A/2^{(1)})$ are the unique non-zero morphisms between these strict polynomial functors.
\end{lemma}
\begin{proof}
Consider the  commutative diagram:
$$\xymatrix{
&{\begin{array}{c}
\Gamma^3_\Z(A)\otimes A\,\oplus\\
\Gamma^2_\Z(A)\otimes \Gamma^2_\Z(A)
\end{array}}
\ar[d]\ar[r]^-{
(2,0)}& \Gamma^3_\Z(A)\otimes A \ar[r]\ar[d] & \Gamma^3_\Fdeux(A/2)\otimes A/2\ar[d] \ar[r] & 0 \\
0\ar[r]&\Gamma^4_\Z(A)\ar[r]^-{\times 2}\ar[d] & \Gamma^4_\Z(A)\ar[r]\ar[d] & \Gamma^4_\Fdeux(A/2)\ar[d]\ar[r]&0\\
0\ar[r]&A/2^{(2)}\ar@{-->}[r]^-{(a)}\ar[d] & \Gamma^2_\Z(A/2^{(1)})\ar@{-->}[r]^-{(b)}\ar[d] & \Gamma^2_\Fdeux(A/2^{(1)})\ar[d]\ar[r]&0\\
&0 & 0 & 0&
\;.}$$
In this diagram, the upper vertical arrows are induced by the multiplication in  the divided power algebra, and the rows and  columns are exact. The  two dashed arrows are produced by the universal property of quotients. It is easy to compute that the  arrow $(a)$ is injective when $A=\Z$, hence it is injective for all $A$ by additivity of $A/2^{(2)}$. Since the middle row is exact, an elementary diagram chase then shows that the bottom row also is. We have thus  obtained an extension of $\Gamma_\Fdeux^2(A/2^{(1)})$ by $A/2^{(2)}$. It is non split because the middle term $\Gamma_\Z^2(A/2^{(1)})$ has $4$-torsion.

We know that  $\hom_{\PP_\Z}(A/2^{(2)},A/2^{(2)})=\Z/2$ and $\hom_{\PP_\Z}(A/2^{(2)},\Gamma_\Fdeux^2(A/2^{(1)})=0$. The  left-exactness of the functor $\hom$ therefore implies that the dashed arrow $(a)$ is the unique non-zero morphism from $A/2^{(2)}$ to $\Gamma^2_\Fdeux(A/2^{(1)})$. To prove that the dashed arrow $(b)$ is also characterized as the unique non-zero morphism available, we make use of the fact  that $\Gamma^4(A)$ surjects onto $\Gamma_\Z^2(A/2^{(1)})$ and that there is only one non-zero morphism from $\Gamma^4(A)$ to $\Gamma_\Fdeux^2(A/2^{(1)})$.

Finally let us   compute the group $\Ext^1_{\PP_\Z}(\Gamma^2_\Fdeux(A/2^{(1)}, A/2^{(2)})$.
 This is an $\Fdeux$-vector space and  we will now show that it is of dimension one. We already know that its dimension is at least $1$.
 The short exact sequence
$0\to \Gamma^2_\Fdeux(A/2^{(1)})\to A/2^{\otimes 2}\to \Lambda^2_\Fdeux(A/2^{(1)})\to 0$  therefore induces  an isomorphism
$$\Z/2 = \; \hom_{\PP_\Z}(\Gamma^2_\Fdeux(A/2^{(1)}),A/2^{(2)}) \simeq \Ext^1_{\PP_\Z}(\Lambda^2_\Fdeux(A/2^{(1)}),A/2^{(2)}).$$
By  lemma \ref{lm-ApAp}, the long exact sequence  of $\Ext$'s associated to the short exact sequence $0\to \Lambda^2_\Fdeux(A/2^{(1)}\to \Gamma^2_\Fdeux(A/2^{(1)})\to A/2^{(2)}\to 0$ yields an injective map:
$$\Ext^1_{\PP_\Z}(\Gamma^2_\Fdeux(A/2^{(1)}),A/2^{(2)})\hookrightarrow \Ext^1_{\PP_\Z}(\Lambda^2_\Fdeux(A/2^{(1)}),A/2^{(2)})\;.$$
It follows that  $\Ext^1_{\PP_\Z}(\Gamma^2_\Fdeux(A/2^{(1)}),A/2^{(2)})$ has dimension exactly one as asserted.
\end{proof}

The following lemma is used in proposition \ref{prop-nonsplit1}. The middle term in the unique non-split extension of $\Gamma^2_\Fdeux(A/2^{(1)})$ by  $\Lambda^2_\Fdeux(A/2)\otimes A/2^{(1)}$ which is provided by this computation  is the functor which is denoted by $\Phi^4(A)$  in proposition \ref{prop-descr-SKos} and elsewhere in the text.
\begin{lemma}\label{lm-calc1}
There is an isomorphism  $\Ext ^1_{\PP_\Z}(\Gamma^2_\Fdeux(A/2^{(1)}),\Lambda^2_\Fdeux(A/2)\otimes A/2^{(1)})\simeq \Z/2$.
\end{lemma}
\begin{proof}
The $\Ext^*_{\PP_\Z}(A/2^{(r+1)},-)$ long exact sequence associated to the short exact sequence $0\to \Lambda^2_\Fdeux(A/2^{(r)})\to A/2^{(r)\;\otimes 2}\to S^2_\Fdeux(A/2^{(r)})\to 0$ yields an isomorphism:
\begin{align*} \Z/2 \simeq \hom_{\PP_\Z}(A/2^{(r+1)},S^2_\Fdeux(A/2^{(r)})) \simeq
\Ext ^1_{\PP_\Z}(A/2^{(r+1)},\Lambda^2_\Fdeux(A/2^{(r)})) \;.
\end{align*}
The result of lemma \ref{lm-calc1} follows from this assertion and lemma \ref{lm-ApAp} by using the sum-diagonal adjunction and proposition \ref{prop-Kunneth}.
\end{proof}

Table \ref{ext1comput} collects the results of lemmas \ref{lm-calc2}, \ref{lm-calc1}, and some other easy computations obtained with the same techniques. Some of these computations are used in section \ref{der-gamma4-sec}.

\begin{table}[ht]
\renewcommand{\arraystretch}{1.8}
\begin{tabular}{|c||c|c|c|}
\hline
$\Ext^1_{\PP_\Z}(F(A),G(A))$ & $F(A)=\Gamma_{\Fdeux}^2(A/2^{(1)})$ & $A/2^{(1)}\otimes A/2^{(1)}$ &$\begin{array}{c}\Gamma^2_{\Fdeux}(A/2)\otimes A/2^{(1)}\\
                                                                                                     =\Gamma^2_\Z(A)\otimes A/2^{(1)}
                                                                                                    \end{array}
$  \\
\hline
\hline
 $G(A)=A/2^{(2)}$ & $\Z/2$ & $0$ &  $0$ \\
\hline
 $A/2^{(1)}\otimes A/2^{(1)}$ &  $0$ &  $0$ & $0$  \\
\hline
 $\Gamma^2_{\Fdeux}(A/2^{(1)})$ &  $0$ &  $0$ &  $0$ \\
\hline
 $\Lambda^2_\Fdeux(A/2^{(1)})$ &  $0$ &  $0$ & $0$ \\
\hline
 $\Gamma^2_\Z(A/2^{(1)})$ &  $0$ &  $0$ &  $0$ \\
\hline
$\Lambda^2_\Fdeux(A/2)\otimes A/2^{(1)}$ &  $\Z/2$ &  $\Z/2^{\oplus 2}$ &  $\Z/2$ \\
\hline
\end{tabular}
\vspace{.5cm}
\caption{Some computations of $\Ext^1_{\PP_\Z}(F(A),G(A))$ for functors $F$ and $G$ of weight $4$.}\label{ext1comput}
\end{table}
\subsubsection{Examples of computations in $\FF_\Z$.}
One can sometimes compute $\Ext^1$ groups in $\FF_\Z$ by methods close  to those which we used  in $\PP_\Z$. For example, the proof of lemma \ref{lm-calc2} carries over  without change in $\FF_\Z$ so that  we obtain the following result.
\begin{lemma}\label{lm-calc2bis}
There is an isomorphism
$\Z/2\simeq  \Ext ^1_{\FF_\Z}(\Gamma^2_\Fdeux(A/2^{(1)}),\Lambda^2_\Fdeux(A/2)\otimes A/2^{(1)}) $.
\end{lemma}
The reasoning of the proof of lemma \ref{lm-calc2} has to be slightly modified in $\FF_\Z$.  Since $\Gamma^2_\Z(A/2)$ has $4$-torsion,  the forgetful functor $\PP_\Z\to \FF_\Z$ sends the extension \[0\to A/2^{(2)}\to \Gamma^2_\Z(A/2^{(1)})\to \Gamma_\Fdeux^2(A/2^{(1)})\to 0\] to a non-split extension in $\FF_\Z$. But the forgetful functor sends $A/2^{(2)}$ to the ordinary functor $A/2$, and  $\Ext^1_{\FF_\Z}(A/2,A/2)=\Z/2$ (the middle term in the non-trivial extension being  the non-split extension  $A/4$). Thus the self-extensions of $A/2$ come into play. Reasoning as in the proof of lemma \ref{lm-calc2}, we obtain a short exact sequence of $\Fdeux$-vector spaces
$$0\to \Ext^1_{\FF_\Z}(A/2,A/2)\to \Ext^1_{\FF_\Z}(\Gamma^2_\Fdeux(A/2),A/2)\to \Ext^1_{\FF_\Z}(\Lambda_\Fdeux^2(A/2),A/2)\to 0\,.$$
This leads to the following result:
\begin{lemma}
\label{ext-gamma2}
 $\Ext^1_{\FF_\Z}(\Gamma^2_\Fdeux(A/2^{(1)}),A/2^{(2)})=\Z/2^{\oplus 2}$.
\end{lemma}

\section{The integral homology of $K(A,n)$ for a free abelian group $A$.}
\label{han}

In this appendix, we translate some the computations of the derived functors of the divided power functors $\Gamma^d(A)$ achieved in this article in terms of the integral homology of Eilenberg-Mac Lane spaces. We begin with a short survey on the relation between the functors $L_*\Gamma^d(A,n)$ and the homology of $K(A,n)$. Then we present a table giving a functorial   description of
the    groups $H_{n+i}(K(A,n);\Z)$ in low degrees.

\subsection{The homology of $K(A,n)$ and the derived functors $L_*\Gamma(A,n)$}
As proved by Dold and Puppe \cite[Satz 4.16]{d-p}, there exist isomorphisms \begin{equation}
\label{isom-h-s}
 \bigoplus_{d\ge 0}\   L_jS^d(A,n)   \simeq  \ H_j(K(A,n);\Z).   \end{equation} 
However, 
the functoriality in $A$ of such an isomorphism relies on the construction of a Moore space $M(A,n)$. 
It is shown in
\cite[Prop. 7.7]{antoine} 
 that for $A$ free  there are  non-unique morphisms  from the complex
$\bigoplus \NN S^d(A,n)$ to $\NN\Z[A,n]$    which  however  induce a functorial
isomorphism \eqref{isom-h-s}
at the level of homology groups.
 By the d\'ecalage isomorphisms \eqref{dec} this reduces the computation of the homology  to that of the derived functors of $\Gamma^d(A)$ for $A$ free. For $n=1$, the functorial  isomorphism \eqref{isom-h-s} can thus be rewritten  as the well-known isomorphism \[\Lambda^j(A) \simeq  H_j(K(A,1);\Z),\] while for $n\ge 2$ one obtains by double d\'ecalage  the isomorphisms:
\begin{equation}
\label{isom-h-s-new}
 \bigoplus_{d\ge 0} \  L_{j-2d} \  \Gamma^d(A,n-2) \  \simeq \  H_j(K(A,n);\Z) .
\end{equation}

The situation is more complicated  if one considers arbitrary abelian groups $B$. In that case the isomorphism \eqref{isom-h-s} must be replaced by the second quadrant  spectral sequence
\bee
\label{aug-ss}
E^1_{p,q} = L_{p+q}S^{-p}(B,n) \Longrightarrow H_{p+q}(K(B,n);\Z)
\ee
which  by \cite{lb-fh} degenerates at $E^1$. As a result, the expressions  involving in this  context the derived functors of the functors $S^k(B)$ (or equivalently by d\'ecalage  of $\Gamma^k(B)$ ) only   describe the abutment $H_{p+q}(K(B,n);\Z)$ up to a  filtration  since they  live in $E^1 = E^{\infty}$.
\begin{align}
\bigoplus_{d\ge 0} \  L_{j-2d} \  \Gamma^d(B,n-2)\simeq  \bigoplus_{d\ge 0} L_j S^d(B,n)\simeq \bigoplus_{d\ge 0}\gr_{d}\left(H_{j}(K(B,n);\Z)\right)\;.
\end{align}
In the stable range, that is for the homology groups $H_{i+n}(K(B,n);\Z)$ with $i<n$, the filtration splits. Indeed in the stable range these homology groups are direct sums of copies of the functors $_pB$ (the $p$-torsion subgroup of $B$) and $B/p$ (the mod $p$-reduction of $B$), for prime integers $p$. This follows for example from Cartan's computation of the stable homology of Eilenberg-Mac Lane spaces \cite[Exp 11, Thm 2]{cartan}. Such functors are simple functors, hence there cannot be non-trivially filtered.
The filtration is usually not split in non stable degrees. This is already visible in the simplest case, that of the homology group  $H_3(K(B,1);\Z)$, in other words  the third homology group of the abelian group $B$. In that case, the spectral sequence \eqref{aug-ss}  reduces
to the  functorial short
 exact sequence
\bee
\label{nonsplit}
\xymatrix{
0 \ar[r] & \Lambda^3(B) \ar[r] & H_3(B) \ar[r] &
L_1\Lambda^2(B,0)\ar[r] & 0.
}
\ee
It is shown in \cite[cor 3.1]{roman2} that this exact  sequence  cannot be
functorially split.

\subsection{The homology of $K(A,n)$ in low degrees for $A$ free.}
The  table below
 gives   a functorial   description of
the    groups $H_{n+i}(K(A,n);\Z)$ as a functor of   the free abelian group  $A$, for all values of the integers $n$ and $i$ such that
$1 \leq n \leq 11$ and $ 0 \leq i \leq 10$.  The convention  in lines $n=1$ to $11$ of  the table is that each empty box   is
actually  filled  with a copy of the expression in the lowest
non-empty  box  above it. These are therefore all  filled with additive groups $A/p$, since each of  these terms  arises by suspension from the  one immediately  above it.
Every  expression on the line ``adm'' displays  Cartan's  labelling of the  stable  groups in the corresponding column  in terms of his    admissible sequences,   and its associated prime. In the line  above this one,  labelled $\EuScript T$,  is displayed (for those admissible word  which we called  restricted in section  \ref{new-descr}) the corresponding labelling by decreasing sequences of integers mentioned  in theorem \ref{thm-new-formula}.
 The order in which the items in each box  in  these two lines are listed reflects the order in which the additive functors to which they correspond occured in the lines above it.

\bigskip

The table is obtained from our computations of derived functors of divided powers and the functorial isomorphism \eqref{isom-h-s-new}. Indeed, since the functors $L_{j-2d}\Gamma^d(A,n-2)$ vanish for $j-2d\le n-2$, we  need only  consider these functors for $d \leq 6$. The required values  the derived functors  for $d \leq 4$ were obtained in sections 8-10 above. For $d= 5,6$, it is only necessary to know the values of the derived functors  $L_r\Gamma^d_\Z(A,n-2)$ for $r\le 2$. The latter can  easily  be obtained, either by a partial analysis of the maximal filtration of $\Gamma^d_\Z(A)$ for such values of $d$,  as  provided  in example \ref{ga23}  for  $d \leq 4$, or by  reliance  for those values of $r$ on the
 mod $p$  reduction method described in section \ref{der1-gamma-z}. This second method reduces the problem to that of a
functorial computation of certain derived functors of $L_*\Gamma^d_{\mathbb{F}_p}(V,n)$ for an
$\fp$-vector space $V$, a question  which we discussed in section \ref{descr-gvn}.

\bigskip

The table shows  that
the homology groups  within the range of values of the pair  $(n,i)$ considered can be  expressed functorially  as  direct sums of tensor and
exterior powers  of the elementary functors $A$ and  $A/p$ for  primes  $p
\leq 5$, together with  divided power functors $\Gamma^i_\Z(A)$ and $\Gamma_\Z^2(A/2)$.  The only exceptions to this
rule are the occurrences  of    direct summands  $\Gamma^2_{\mathbb{F}_2}(A/2)$ in  $H_{13}(K(A,3);\Z)$ and $H_{15}(K(A,5);\Z)$, and  of the new functor  $\Phi^4(A)$ in  $H_{11}(K(A,3);\Z)$, for which   we have
provided  a number of descriptions in the main body of our text.
The fact that  the primes $2,3$ and $5 $ seem to play
a special role here  is simply due to our chosen range of values for
the variables $n$ and $i$. As these values increase the   functors
$A/p$ will occur for additional  primes  $p$.

\bigskip

While we obtain in this way a complete  understanding of the homology of $K(A,n)$ for $A$ free in the range mentioned above, we
wish to  draw the reader's attention to the fact  that the
situation is  more complicated when  one no longer restricts oneself to  functors on the category of  free abelian  groups $A$. Additional functors
of $A$ occur in the general case, whose   values are torsion  groups. The simplest of these functors are the additive  functors ${}_pA := \ker A \stackrel{p}{\to} A$.
Immediately  after these, one encounters
the functors $\Omega(A)$ and $R(A)$,  which Eilenberg and Mac Lane   introduced   in their foundational text
 \cite{hpin} II \S \   13,~ \S \ 22. These would now be denoted
$L_1\Lambda^2(A,0)$ and $L_1\Gamma^2_\Z(A,0)$ (or simply  $L_1\Lambda^2(A)$ and $L_1\Gamma^2_\Z(A)\,$) respectively and they would indeed appear in a table similar to  ours for   a general $A$,  as the derived versions of the functors $\Lambda^2(A)$ and $\Gamma^2(A)$ which are to be found  in positions $(n,i) = (1,1)$ and $(n,i) = (2,2)$  of our table   respectively. The
d\'ecalage morphisms  imply  that for $A$ non-free,  these functors
contribute to the groups
$H_\ast(K(A,1);\Z)$ and $H_\ast(K(A,2);\Z)$ respectively, so that the
first two lines of our table below  are already much more complicated in that case.

\bigskip

In this table we  can already observe many of
the phenomena discussed in the text. As we mentioned above, the first of these is
the  lowest occurrence of  the periodic phenomenon represented by the functor $\Phi^4(A)$, which is  the only new functor within our range of values.
The pair of functors $\Gamma^4_\Z(A/2)$ for the values 4 and 6 of $n$ are also noteworthy, as they produce some 4-torsion in the homology.  Finally, it will be seen that  the two  2-torsion expressions $\Gamma^2_{\f2}(A/2)$,  while isomorphic, actually  correspond to two distinct situations. Indeed they  may be thought of in the spirit of  conjecture\ref{conject} as the first derived functors of $\Gamma^2(A/2)$ and $\Lambda^2(A/2)$  respectively. This is reflected in the different behavior of these two  homology groups under suspension. While the functor $\Gamma^2_{\f2}(A/2)$ which lives in bidegree $(n,i) = (4,9)$ suspends to an $A/2$ in $H^{\rm st}_9(K(A)$ as one would expect, this is not the case for the $\Gamma^2_{\f2}(A/2)$ in bidegree $(5,10)$. One verifies by the long exact sequence of \cite{d-p} Korollar 6.11 that the suspension sends this element of $H_{15}(K(A,5);\Z)$ to the summand $\Gamma_\Z(A/2)$ of $H_{16}(K(A,6)$ by the unique non-trivial transformation from $\Gamma^2_{\f2}(A/2)$ to $\Gamma^2_\Z(A/2)$, in other words the composite map
\[ \xymatrix{\Gamma^2_{\f2}(A/2) \ar[r]^(.6){V} & A/2 \ar[r]^(.42)F & S^2(A/2) \ar[r] & \Gamma^2_\Z(A/2) }  \]  
  where the maps $V$ and $F$ are the Verschiebung and the Frobenius maps respectively. Since the image of this map is decomposable in $\Gamma^2_\Z(A/2)$, it follows that an additional suspension sends this image to zero in  $H_{17}(K(A,7)$. In particular, this functor $\Gamma^2_{\f2}(A/2)$ does not produce an additional  stable summand $A/2$ in $H^{\rm st}_{10}(A;\Z)$.  

\newpage



\bigskip

\begin{turn}{90}
{\tiny
\renewcommand{\arraystretch}{1.8}
\begin{tabular}{|r||c|c|c|c|c|c|c|c|c|c|c|c|c|c|c|}
\hline &$i=0$&1&2&3&4&5&6&7&8&9&10
\\\hline \hline
$n=1$&$A$&$\Lambda^2(A)$& $\Lambda^3(A)$&$ \Lambda^4(
A)$&$\Lambda^5(A)$ & $\Lambda^6(A)$& $\Lambda^7(A)$ &
$\Lambda^8(A)$ & $\Lambda^9(A)$ & $\Lambda^{10}(A)  $ &
$\Lambda^{11}(A)$
\\
\hline 2&& 0 & $\Ga^2(A)$&0 &$\Ga^3(A)$&0 &  $ \Ga^4(A) $&0& $\Gamma^5(A)$ &0 & $\Gamma^6(A)$\\
\hline 3&&&$A/2$& $\Lambda^2(A)$& $ A/3$&$A \ot A/2 $ &
$\Lambda^3(A) \oplus  A/2 $ &$ A \ot A/3$&$\Phi^4(A) \oplus A/5$&
$\La^4(A) \oplus (A \ot A/2) $&$\Lambda^2(A)\ot A/3$\\ &&&&&&&& $ \oplus  \, \Lambda^2(A/2)$
&& &$\oplus\,\,\, A \ot \Lambda^2(A/2) $
\\
  \hline4&&&&0&$\Ga^2(A) \oplus A/3$& 0 & $  (A \ot A/2)  \oplus
 A/2  $ &0&$ \Ga^3(A) \oplus A/5 \,\,\oplus $&
$\Gamma^2_{\f2}(A/2) $&    $\Gamma^2(A)\otimes A/2$   \\
&&&&&&&&&$ A \ot A/3 \oplus \Ga^2(A/2)$&
 &$\oplus A\otimes A/2$
\\
\hline 5&&&&&$A/2 \oplus A/3$ & $ \Lambda^2(A) $ & $ A/2 $ &$ A
\otimes A/2 $&  $A/3 \oplus A/5$&
  $\Lambda^2(A/2) \oplus \, A/2\, \oplus $ &  $\Lambda^3(A)\oplus\Gamma^2_{\f2}(A/2)$
\\&&&&&&&& &$\oplus \,A/2  $ &$ (A\ot A/2) \, \oplus\,( A \ot A/3)$ &
\\
\hline 6&&&&&&0   &$\Ga^2(A) \oplus    A/2$ &0&$ (A\ot A/2) \oplus A/3$
&$A/2$&$\Gamma^2(A/2) \oplus A\, \ot\, A/3$
\\&&&&&&&&&$\oplus \, A/5 \, \oplus\,  A/2$&&
\\  \hline
7&&&&&&&$ A/2 \oplus A/2 $  &$\Lambda^2(A)$ &$A/3 \oplus A/5 \oplus A/2$&$A/2\oplus A\otimes A/2$ &$A/2$
 \\ \hline
8&&&&&&&&0&$\Ga^2(A) \oplus A/3$&$A/2$ &$A/2\oplus A\otimes A/2$\\&&&&&&&&&
$ \oplus A/5 \oplus A/2$& &
 \\ \hline 9 &&&&&&&& &$A/2 \oplus A/3 \oplus A/5 \oplus A/2$&$A/2\oplus
\La^2(A) $ & $A/2$
 \\ \hline
10   &&&&&&&&&& $A/2$ & $\Gamma^2(A)\oplus A/2 $
\\
\hline 11 &&&&&&&&&&& $A/2\oplus A/2$
\\\hline
\hline
$\EuScript T $&&& $\{1;2\}$ & & $\{1,1;2\}, \{1;3\}$ & &\{1,1,1;2\},
\{2;2\}& & $ \{1,1,1,1;2\}, \{1,1;3\}$  &&$\{1,1,1,1,1;2\} $\\&&&&&&&&& $ \{1;5\}, \{2,1;2\}$& & $\{2,1,1;2\}$
 \\
\hline $\text{adm.}$ &&& $(2;2)$ & & $(4;2),(4;3)$& & $(6;2), (4,2;2)$ &&  $(8;2),(8;3),(8;5),(6,2;2)$ & $(6,3;2)$& $(10;2), (8,2;2)$\\
\hline
\end{tabular}
}
\end{turn}


\section{ The derived functors of  $\Gamma^4(A,n)$ for  $A$ free and $1 \leq n \leq 4$.}
\label{app:derG4}
\begin{table}[h]
\vspace{.5cm} {\footnotesize
 \renewcommand{\arraystretch}{1.8}
\begin{tabular}{|r||c|c|c|c|c|}
\hline  & $n=1$&2&3&4
\\\hline \hline
$i=0$ & $A/2^{(2)}$ & $A/2^{(2)}$ & $A/2^{(2)}$ & $A/2^{(2)}$
\\
\hline $1$ & $\Lambda^2(A/2^{(1)})\,\oplus \, A\otimes A/3^{(1)}$
& $0$ & $0$ & 0
\\ \hline 2 & $\Phi^4(A)$ & $\Gamma^2(A/2^{(1)})$ &$A/2^{(2)}$& $A/2^{(2)}$ \\ &&  $\,\oplus \, A\otimes A/3^{(1)}$ &
 &
\\ \hline 3 & $\Lambda^4(A)$ & $\Gamma^2_{\f2}(A/2^{(1)})$ &
$\Lambda^2(A/2^{(1)})\oplus A/2^{(2)}$&  $A/2^{(2)}$\\
&&& $\oplus \,  A\otimes A/3^{(1)}$&
\\ \hline 4 & 0 & $\Gamma^2_{\f2}(A/2)\otimes A/2^{(1)}$ &
$\Gamma^2_{\f2}(A/2^{(1)})$ & $\Gamma^2(A/2^{(1)}) \oplus\, A\ot  A/3^{(1)}$
\\ \hline 5 & 0 & 0 & $A/2^{(1)}\otimes A/2^{(1)}$ & $\Gamma^2_{\f2}(A/2^{(1)})$
\\ \hline 6 & 0 & $\Gamma^4(A)$ & $\Gamma^2_{\f2}(A/2)\otimes
A/2^{(1)}$  &
 $A/2^{(1)}\otimes A/2^{(1)}$\\ &&&  $\,\oplus \,A/2^{(2)}$ &
 $\oplus\, A/2^{(2)}$
\\ \hline 7 & 0 & 0 & $\Lambda^2(A/2^{(1)})\,\oplus \, A\otimes A/3^{(1)}$ & $A/2^{(1)} \ot  A/2^{(1)} $
\\ \hline 8 & 0 & 0 & $\Phi^4(A)$ & $\Gamma^2(A/2^{(1)})\oplus
\Gamma^2(A)\otimes A/2^{(1)}$ \\  &&&& $\oplus\, A\otimes A/3^{(1)} $
\\ \hline 9 & 0 & 0 & $\Lambda^4(A)$ & $\Gamma^2_{\f2}(A/2^{(1)})$
\\ \hline 10 & 0 & 0 & 0 & $\Gamma^2_{\f2}(A/2)\ot A/2^{(1)} $
\\ \hline 11 & 0 & 0 & 0 & 0
\\ \hline 12 & 0 & 0 & 0 & $\Gamma^4(A)$
\\ \hline
\end{tabular}
}

\begin{caption}
{The derived functors $L_{n+i}\Gamma^4(A,n),$ for $A$ free and
$1\leq n \leq 4$.}
\end{caption}
\label{lasttatble}
\end{table}

\bigskip

{\small
L. Breen and A. Touz\'e:
LAGA, CNRS (UMR 7539)

Universit\'e Paris 13

93430  Villetaneuse, France.

\bigskip

R. Mikhailov: Chebyshev Laboratory, St. Petersburg State University,

 14th Line, 29b, Saint Petersburg, 199178 Russia

 \medskip

  and

\medskip

 St. Petersburg Department of the Steklov Mathematical Institute,

 Fontanka 27, Saint Petersburg, 191023 Russia.





}

\end{document}